\documentclass[10pt,a4paper]{article}
\usepackage[top=1in,bottom=1in,left=1in,right=1in]{geometry}
\usepackage[fleqn]{amsmath}
\usepackage{amssymb,color,amsthm}
\usepackage[center]{titlesec}
\usepackage{float}
\usepackage{graphicx}
\usepackage{epstopdf}
\usepackage{booktabs}
\usepackage{longtable}
\usepackage{rotating}
\usepackage{multirow}
\usepackage{indentfirst}
\usepackage{titlesec}
\usepackage{bm}
\usepackage{bbm}
\usepackage{url,colortbl}
\usepackage{subfigure}
\usepackage{appendix,graphicx}
\usepackage{array,booktabs,setspace}

\numberwithin{equation}{section}
\graphicspath{{FIG/}}
\newtheorem{Theorem}{Theorem}[section]

\newtheorem{Remark}{Remark}[section]

\newtheorem{Scheme}{Scheme}[section]

\newcounter{RomanNumber}
\newcommand{\MyRoman}[1]{\setcounter{RomanNumber}{#1}\Roman{RomanNumber}}
\makeatletter
\renewcommand{\section}{\@startsection{section}{1}{0mm}
  {-\baselineskip}{0.5\baselineskip}{\bf\leftline}}
\renewcommand{\subsection}{\@startsection{subsection}{1}{0mm}
  {-\baselineskip}{0.5\baselineskip}{\leftline}}
\makeatother

\title{\large Arbitrarily high-order structure-preserving schemes for the Gross-Pitaevskii equation with angular momentum rotation in three dimensions}
\author{
\small Jin Cui$^{1}$, Yushun Wang$^2$  and Chaolong Jiang$^{3}$\thanks{Corresponding author. E-mail: chaolong\_jiang@126.com.}\\
\small $^1$Department of Basic Sciences, Nanjing Vocational College of Information Technology,\\
\small Nanjing 210023, China \\
\small  $^2$Jiangsu Key Laboratory for Numerical Simulation of Large Scale Complex Systems, \\
\small  School of Mathematical Sciences, Nanjing Normal University, Nanjing 210023, China\\
\small $^3$School of Statistics and Mathematics,  Yunnan University of Finance and Economics, \\
\small Kunming 650221, China
}

\begin{document}

\date{}
\maketitle
\titleformat*{\section}{\centering}{\Large}
\titleformat*{\subsection}{\centering}{\large}

\begin{center}{Abstract} \end{center}

\indent In this paper, we design a novel class of arbitrarily high-order structure-preserving numerical schemes for the time-dependent Gross-Pitaevskii equation with angular momentum rotation in three dimensions. Based on the idea of the scalar auxiliary variable approach which is proposed in the recent papers [J.
Comput. Phys.,  416 (2018) 353-407 and SIAM Rev., 61(2019) 474-506] for developing energy stable schemes for gradient flow systems, we firstly reformulate the Gross-Pitaevskii equation into an equivalent system with a modified energy conservation law. The reformulated system is then discretized  by the Gauss collocation method in time and the standard Fourier pseudo-spectral method in space, respectively. We show that the proposed schemes can preserve the discrete mass and modified energy exactly. Numerical results are addressed to verify the efficiency and high-order accuracy of the proposed schemes.

\section{Introduction}

The formation and dynamics of Bose-Einstein condensate (BEC) are usually modeled by the Gross-Pitaevskii (GP) equation which is essentially a Schr$\ddot{\rm o}$dinger equation involving an additional nonlinear term related to particle-particle interactions~\cite{Lieb01,pita61}. To characterize a rotating BEC, it is common to generalize the model by an angular momentum term. Let $[0,T] \subset \mathbb{R}$ be a time interval and $\mathbb R^d \ ( d=2,3)$ be spatial domain. We consider the dimensionless time-dependent GP equation, and seek the complex-valued wave function $\psi:  \mathbb R^d  \times [0,T] \to \mathbb C$ which characterizes the quantum state of the rotating condensate. The targeted rotational GP equation can be written as follows
\begin{align*}
 \text{{i}} \partial_t\psi= -\frac{1}{2} \Delta \psi + V \psi+ \text{i} {\bf \Omega} \cdot ({\bf x} \times \nabla) \psi+\beta |\psi|^2  \psi, \ \ {\bf x} \in \mathbb R^d  , \ t>0.
\end{align*}
Here $\text{i}=\sqrt{-1}$, $t$ is time variable, and ${\bf x}=(x,y)^{\top}\in \mathbb{R}^2$ or $(x,y,z)^{\top} \in \mathbb{R}^3$ is the Cartesian coordinate vector. Note that $V({\bf x})$ is a real-valued function with respect to the external trap potential and it is experimentally chosen as a harmonic potential, i.e. a quadratic polynomial. The nonlinear term $\beta|\psi|^2\psi$ describes the species of the bosons and how they interact (negative for attractive interaction and positive for repulsive interaction) between particles in experiments. In particular, $\beta$ depends on the number of bosons, their individual mass, and scattering length. Moreover, the term $\text{i}{\bf \Omega} \cdot ({\bf x} \times \nabla) \psi$ means the angular rotation of the condensate, while ${\bf \Omega} \in \mathbb R^3$ characterizes the angular speed of the laser beam. In general, the operator ${\bf L}=(L_x,L_y,L_z):=-\text{i}({\bf x} \times \nabla)={\bf x}\times {\bf P}$ denotes the angular momentum, where ${\bf P}=-\text{i}\nabla$ means the momentum operator. For brevity, we assume that the rotation is around the $z$-axis, which can deduce the simplification $\text{i} {\bf \Omega} \cdot ({\bf x} \times \nabla)=-\Omega L_z$, where $L_z=- { {\text{i}}} (x \partial_y -y \partial_x)$ is the $z$-component of the angular momentum.

\indent In this paper, we focus numerically on the following three-dimensional case:
\begin{align}\label{eq1.1}
\text{i} \partial_t\psi=\left[ -\frac{1}{2} \Delta + V(x,y,z)-\Omega L_z+\beta |\psi|^2 \right] \psi, \ \ (x,y,z) \in  \mathcal{D}, \ \ 0<t \leqslant T,
\end{align}\label{eq1.2}
subject to the $(l_x,l_y,l_z)$-periodic boundary conditions
\begin{align}\label{eq1.2}
&\psi(x,y,z,t)=\psi(x+l_x,y,z,t), \ \ \psi(x,y,z,t)=\psi(x,y+l_y,z,t),  \nonumber \\
&\psi(x,y,z,t)=\psi(x,y,z+l_z,t), \ \ (x,y,z) \in  \mathcal{D},\  \ 0<t\leqslant T,
\end{align}
and the initial condition
\begin{align}
\psi(x,y,z,0)=\psi_0(x,y,z), \ \ (x,y,z) \in  \mathcal {D}, \label{eq1.3}
\end{align}
where $ \mathcal{D}=[0,l_x]\!\times\![0,l_y]\!\times\![0,l_z]$, and $\psi_0(x,y,z)$ is a given ($l_x,l_y,l_z$)-periodic complex-valued function.

 In fact, one can easily verify that the initial-boundary value problem (1.1)-(1.3) preserves the following mass and energy conservation laws

 \begin{align}\label{mass-conservation-law}
 M(\psi(\cdot,t)):=\int_{\mathcal{D}}^{}|\psi(\cdot,t)|^2 d{\bf x} \equiv M(\psi_0), \ \ t\geqslant  0,
\end{align}
 and
 \begin{align}\label{energy-conservation-law}
 E(\psi(\cdot,t)):=\int_{\mathcal{D}}^{} \left[ \frac{1}{2}|\nabla \psi|^2+V({\bf x})|\psi|^2- \Omega \bar \psi  L_z\psi+\frac{\beta}{2} |\psi|^4 \right ] d{\bf x} \equiv E(\psi_0),              \ \ t\geqslant  0,
 \end{align}
 where $|\nabla \psi|^2 = |\psi_x|^2+|\psi_y|^2+|\psi_z|^2, $ and $\bar \psi$  refers to the conjugate of $\psi$.

\indent In the last decades, the model \eqref{eq1.1} has been studied a lot in both theoretical analyses and numerical simulations.
For the derivation, well-posedness and dynamical properties, readers are referred to~\cite{hao,lieb,pita03}.
The existing numerical methods for the GP equation include the finite difference methods~\cite{baoc}, finite element method~\cite{hen}, spectral collocation method~\cite{chenhs}, split-step and implicit corrected parallel SPH method~\cite{Jiangt}, time-splitting  generalized-Laguerre-Fourier-Hermite pseudo-spectral method~\cite{baohs}, Gauss exponential Runge-Kutta (ERK) and exponential integrators (Lawson) methods~\cite{Besse15}, etc. A comparative overview on different numerical methods for solving the nonlinear Schr$\ddot{\rm o}$dinger/GP equations can be found in~\cite{antoine13} and the references therein. However, to our best knowledge, there are few references mentioned above considering  energy-preserving schemes for the GP equation \eqref{eq1.1}.

It is well-known that the energy conservation law plays an important role in the study of solutions of mechanical systems (e.g., see \cite{Ben}), and whether or not can preserve the energy of the original systems is a criterion to judge the success of a numerical method for their solutions. In Ref. \cite{baoc}, Bao and Cai developed a Crank-Nicolson finite difference scheme, which preserve the discrete mass and energy exactly, for the rotational GP equation. In Ref. \cite{wangg19}, Wang et al. showed that the classical Crank-Nicolson-type schemes can preserve a modified mass and energy conservation law by introducing an energy function of the grid functions using  recursive relations. More recently, Cui et al.~\cite{cuijincpc} developed an energy-preserving linearly-implicit Fourier pseudo-spectral scheme for the GP equation.  It is noteworthy that Bao et al.~\cite{bao06} presented an efficient time-splitting method, which preserves the discrete energy for the non-rotational case. Unfortunately, most existing energy-preserving works are at most second-order accuracy in time. In general, the GP equation usually requires longtime computation to obtain a condensate ground state for given iteration criteria, thus apart from the energy conservation law, higher-order accurate schemes are always highly desired, which makes large marching steps practical while preserving the accuracy.
Until now, the literature on developing higher-order schemes for the rotational GP equation is rather limited. Although the Gauss ERK and Lawson methods~\cite{Besse15} have been proposed to achieve such goal, both schemes failed to handle the energy conservation property.

 Over the past decade, there have been many attempts to develop high-order energy-preserving methods for solving conservative systems.
In~\cite{QM08}, Quispel and  McLaren proposed third- and fourth-order averaged vector flied (AVF) methods. Further analyses for the sixth-order AVF method can be found in~\cite{LWQ14}. Subsequently, based on the discrete line integral methods,  Brugnano et al. developed a series of excellent high-order energy-preserving methods, named Hamiltonian boundary value methods (HBVMs) (e.g., see~\cite{Brug,BI16,Brug18}),
for the Hamiltonian system with a polynomial energy function.
In~\cite{H10}, Hairer further introduced a variant of collocation methods, which
can remove the limit of the HBVMs to cover the non-polynomial case. The selected high-order methods can be easily extended to propose high-order energy-preserving schemes for the GP equation \eqref{eq1.1}, which however cannot preserve the energy and mass simultaneously (see Refs. \cite{Brug18,GCW14b,LW15} for the classical Schr\"odinger equation).
In Ref. \cite{JWG19}, Jiang et al. proposed a new high-order energy-preserving method, based on the invariant energy quadratization approach~\cite{GZYW18,YZW17,ZWY16}, for the Camassa-Holm equation. More recently, this idea has been extended to solve nonlinear Schr\" odinger equations in one and two dimensions~\cite{Liuzy}. The proposed schemes can preserve both the energy and the mass, but it is challenging for diagonally implicit Runge-Kutta methods to achieve arbitrary high-order accuracy.


In this paper, we aim to develop a class of arbitrarily high-order schemes for numerically solving the GP system~\eqref{eq1.1}, which can preserve both the mass and energy.
Instead of traditional ideas in which ones design special numerical schemes directly or extend energy-preserving schemes from low-order accuracy to be high-order, we first recast the energy conservation law by introducing a new auxiliary variable and then reformulate the original system into an equivalent system, which is inspired by the idea of the scalar auxiliary variable (SAV) approach. Some classical high-order methods are finally applied to achieve the goal.
Specifically, for the GP equation~\eqref{eq1.1}, we firstly reformulate it into an equivalent system, which conserves the original mass and a modified energy, by introducing an scalar auxiliary variable. The classical Gauss collocation methods are then employed to discretize the reformulated system in time. It is shown the resulting schemes can preserve both the mass and modified energy exactly. Different from~\cite{Liuzy}, the proposed schemes can reach arbitrarily high-order in time and the introduced auxiliary variable is a scalar not a vector in the discrete level, which implies that our methods are more efficient. Moreover, a fast solver is designed for numerical implementations, which can be directly extended to solve the existing schemes~\cite{JWG19,Liuzy} efficiently. Through a set of numerical simulations, we demonstrate the high accuracy and invariants-preserving of the proposed schemes thereafter.


The remainder of this paper is arranged as follows. In Section 2, based on the SAV approach, the GP equation \eqref{eq1.1} is reformulated into an equivalent form. In Section 3, we derive a class of high-order semi-discreted schemes in time by using the Gaussian collocation method, which are proven to be energy-preserving and mass-preserving. In Section 4, the Fourier pseudo-spectral method is further applied for spatial discretization to obtain a fully discrete structure-preserving scheme. A fast solver is designed to implement the proposed schemes efficiently in Section 5. In Section 6, we present several numerical examples. Finally, some concluding remarks are drawn in Section 7.

\section{Model reformulation using the SAV approach}
In this section, we utilize the SAV idea to transform the equation  \eqref{eq1.1} into an equivalent system, which possesses  a modified energy function of the new variable. The reformulated system provides an elegant platform for developing high-order structure-preserving schemes. We define the $L^2$ inner product and its norm as $(f,g)=\int_{\mathcal{D}}^{}fgd{\bf x}$ and $\Vert f \Vert=\sqrt{(f,f)}, \ \forall f,g \in L^2(\mathcal D)$, respectively, and denote the linear part of \eqref{eq1.1} as ${\mathcal L}\psi=-\frac{1}{2} \Delta \psi + V \psi-\Omega L_z \psi$ for simplicity.

The system \eqref{eq1.1} can be rewritten as
\begin{align}\label{2.1}
 \partial_t\psi=-\text{i}\frac{\delta \mathcal{H} }{\delta {\bar \psi}},
\end{align}
where
\begin{align}\label{2.2}
 \mathcal{H}=({\mathcal L}\psi,  \psi)+\frac{\beta}{2} (\psi^2, \psi^2 ),
\end{align}
and $\frac{\delta \mathcal{H} }{\delta {\bar \psi}}$ denotes the variational derivative of $\mathcal H$ with respect to $\bar \psi$
\begin{align*}
\frac{\delta \mathcal{H} }{\delta {\bar \psi}}={\mathcal L}\psi+\beta |\psi|^2  \psi.
\end{align*}
Then, by introducing a scalar auxiliary variable
\begin{align*}
 q:=q(t)=\sqrt{(\psi^2,\psi^2)+C_0},
\end{align*}
the energy conservation law \eqref{energy-conservation-law} can be rewritten as
\begin{align}\label{modified-energy-conservation-law}
 E(\psi(\cdot,t)):=( {\mathcal L}\psi, \psi)+\frac{\beta}{2}q^2-\frac{\beta}{2}C_0,
 \end{align}
where $C_0$ is a constant large enough to make $q$ well-defined for all $\psi$. According to the energy variational, we further reformulate the form \eqref{2.1} to the following equivalent system
\begin{align}\label{eq2.5}
\left\{
\begin{aligned}
&\partial_t \psi=-\text{i}\bigg ( {\mathcal L}\psi+ \frac{\beta|\psi|^2\psi q}{\sqrt{(\psi^2,\psi^2)+C_0}} \bigg), \\
&\frac{d}{dt} q=\frac{(\partial_t \psi,|\psi|^2\psi)+(|\psi|^2\psi,\partial_t \psi)}{\sqrt{(\psi^2,\psi^2)+C_0}},
\end{aligned}
\right.
\end{align}
with the consistent initial condition
\begin{align}\label{eq2.6}
\psi({\bf x},0)=\psi_0({\bf x} ),  \  q(0)=\sqrt{\big(  \psi^2_0 ({\bf x} ), \psi^2_0({\bf x} )) +C_0   },
\end{align}
and the periodic boundary condition \eqref{eq1.2}.
\begin{Theorem} The reformulated system \eqref{eq2.5} possesses the modified energy conservation law \eqref{modified-energy-conservation-law} and the mass conservation law \eqref{mass-conservation-law}, respectively.

\end{Theorem}

\begin{proof} It is clear to see
\begin{align*}
\frac{d E}{dt}&=\big(\partial_t{\mathcal L} \psi,\psi \big)+({\mathcal L}\psi,\partial_t \psi)+\beta q \frac{d}{dt} q  \\
&=(\partial_t \psi,{\mathcal L}\psi)+({\mathcal L}\psi,\partial_t \psi)+\bigg ( \partial_t \psi, \frac{\beta|\psi|^2\psi q}{\sqrt{(\psi^2,\psi^2)+C_0}} \bigg)
+\bigg ( \frac{\beta|\psi|^2\psi q}{\sqrt{(\psi^2,\psi^2)+C_0}},\partial_t \psi\bigg)\\
&=2{\textrm {Re}}\bigg ({\mathcal L}\psi+\frac{\beta|\psi|^2\psi q}{\sqrt{(\psi^2,\psi^2)+C_0}},\partial_t \psi\bigg)\\
&=-2{\textrm {Im}}\bigg ({\mathcal L}\psi+\frac{\beta|\psi|^2\psi q}{\sqrt{(\psi^2,\psi^2)+C_0}},{\mathcal L}\psi+\frac{\beta|\psi|^2\psi q}{\sqrt{(\psi^2,\psi^2)+C_0}} \bigg)\\
&=0,
\end{align*}
where {\textrm Re}($\cdot$) and {\textrm Im}($\cdot$) represent the real and imaginary parts of $\cdot$, respectively, and the self-adjointness of the linear operator  ${\mathcal L}$ (i.e., $({\mathcal L}\psi,\phi)=(\psi,{\mathcal L}\phi)$) is used in the third equality.

Similarly, we can deduce
\begin{align*}
\frac{d}{dt}M=\frac{d}{dt}(\psi,\psi)=(\partial_t \psi,\psi)+(\psi,\partial_t \psi)=2{\textrm {Re}}(\partial_t \psi,\psi)=2{\textrm {Im}}\bigg({\mathcal L}\psi+\frac{\beta|\psi|^2\psi q}{\sqrt{(\psi^2,\psi^2)+C_0}} ,\psi\bigg)=0.
\end{align*}
This completes the proof.
\end{proof}

%
%


\section{High-order structure-preserving discretization in time}
In this section, we derive a class of high-order methods for the reformulated system $\eqref{eq2.5}$ by utilizing the collocation method. We show that the proposed schemes can exactly preserve the semi-discrete form of the modified energy \eqref{modified-energy-conservation-law} and mass \eqref{mass-conservation-law}, simultaneously, when the Gauss collocation methods are employed in time. We here focus
on developing time-discrete methods, and denote $t_n=n\tau,\ n = 0,1,2\cdots,N$, where $\tau$ is the
time step. The approximations of the function $\psi({\bf x}, t)$ and $q(t)$ at time $t_n$ are denoted by $\psi^n$ and $q^n$, respectively.

\indent Applying an $s$-stage collocation method to the system $\eqref{eq2.5}$, we can obtain:
\begin{Scheme}\label{sch3.1}  Let $c_1,c_2,\cdots,c_s$ be distinct real numbers $(0 \leqslant c_i \leqslant 1)$. For given $(\psi^n,q^n)$, the collocation polynomials $u(t)$ and $v(t)$ are two polynomials of degree $s$ satisfying
\begin{align*}
&u(t_n)=\psi^n, \ v(t_n)=q^n,\\
&\partial_t u(t_n^i)=-\text{\rm i}\Bigg({\mathcal L}u(t_n^i) +\frac{\beta |u(t_n^i)|^2u(t_n^i)v(t_n^i)}{\sqrt{\big(u^2(t_n^i),u^2(t_n^i)\big)+C_0}} \Bigg),\\
&\frac{d}{dt} v(t_n^i)=\frac{\Big(\partial_t u(t_n^i),|u(t_n^i)|^2u(t_n^i)\Big)+\Big(|u(t_n^i)|^2u(t_n^i),\partial_t u(t_n^i)\Big)}{\sqrt{\big(u^2(t_n^i),u^2(t_n^i)\big)+C_0}},
\end{align*}
where $t_n^i=t_n+c_i \tau, i=1,2,\cdots,s$. Then the numerical solution is defined by $\psi^{n+1}=u(t_n+\tau)$ and $q^{n+1}=v(t_n+\tau)$, respectively.
\end{Scheme}

As is shown by Theorem 1.4 in~\cite{ELW06} that the collocation method could derive a special RK method. Once the collocation points $c_1,c_2,\cdots,c_s$ are chosen as Gaussian quadrature nodes, i.e., the zeros of the s-th shifted Legendre polynomial $
\frac{d^s}{dx^s}\big(x^s(x-1)^s\big)$, the resulting {\bf Scheme~\ref{sch3.1}} is the so-called Gaussian collocation method. According to Ref. \cite{ELW06}, the collocation method shares the same order 2s as the underlying quadrature formula. In particular, the coefficients of fourth order and sixth order Gauss collocation methods have been given explicitly in Ref. \cite{ELW06} (see Table~\ref{tab1} for more details).

\begin{table}[!t]
\hspace{-1cm}
\begin{minipage}[c]{0.46\textwidth}
\begin{spacing}{1.8}
\begin{flushright}
\begin{tabular}{c|cc}
$\frac{1}{2}-\frac{\sqrt{3}}{6}$ &$\frac{1}{4}$ & $\frac{1}{4}- \frac{\sqrt{3}}{6}$\\
$\frac{1}{2}+\frac{\sqrt{3}}{6}$ &$\frac{1}{4}+ \frac{\sqrt{3}}{6}$ &$\frac{1}{4}$ \\
\hline
                                 &$\frac{1}{2}$&$\frac{1}{2}$
\end{tabular}
\end{flushright}
\end{spacing}
\end{minipage}
\hspace{6mm}
\begin{minipage}[c]{0.46\textwidth}
\begin{spacing}{1.8}
\begin{flushleft}
\begin{tabular}{c|ccc}
$\frac{1}{2}-\frac{\sqrt{15}}{10}$ &$\frac{5}{36}$ &  $\frac{2}{9}-\frac{\sqrt{15}}{15}$    &$\frac{5}{36}- \frac{\sqrt{15}}{30}$\\
$\frac{1}{2}$   &$\frac{5}{36}+ \frac{\sqrt{15}}{24}$  &  $\frac{2}{9}$  & $\frac{5}{36}- \frac{\sqrt{15}}{24}$ \\
$\frac{1}{2}+\frac{\sqrt{15}}{10}$ & $\frac{5}{36}+ \frac{\sqrt{15}}{30}$ & $\frac{2}{9}+\frac{\sqrt{15}}{15}$  & $\frac{5}{36}$ \\
\hline
& $\frac{5}{18}$ & $\frac{4}{9}$  & $\frac{5}{18}$
\end{tabular}
\end{flushleft}
\end{spacing}
\end{minipage}
\setlength{\abovecaptionskip}{-0.7cm}
 \caption{\footnotesize RK coefficients of Gaussian collocation methods of order 4 (left) and 6 (right).}
 \label{tab1}
\end{table}

\begin{Theorem} \label{the3.1}The $s$-stage Gaussian collocation {\bf Scheme~\ref{sch3.1}} preserves the following semi-discrete energy and mass conservation laws
\begin{align*}
E^{n}=E^0, \ M^{n}=M^0, \ n=1,2,\cdots,N,
\end{align*}
where
\begin{align}\label{semi-discrte-energy-mass}
E^{n}=(\mathcal L \psi^n,\psi^n)+\frac{\beta}{2}(q^n)^2-\frac{\beta}{2}C_0, \ M^n=(\psi^n,\psi^n).
\end{align}
\end{Theorem}
\begin{proof} It follows from $\psi^n=u(t_n),q^n=v(t_n)$ and $\psi^{n+1}=u(t_{n+1}),q^{n+1}=v(t_{n+1})$ that
\begin{align*}
E^{n+1}-E^{n}&=(\mathcal L \psi^{n+1},\psi^{n+1})-(\mathcal L \psi^n,\psi^n)+\frac{\beta}{2}\big[(q^{n+1})^2-(q^{n})^2\big] \\
&=\big(\mathcal L u(t_{n+1}),u(t_{n+1})\big)-\big(\mathcal L u(t_{n}),u(t_{n})\big)+\frac{\beta}{2}\big[(v(t_{n+1}))^2-(v(t_{n}))^2\big]\\
&=\int_{t_n}^{t_{n+1}} \bigg [\frac{d}{dt}\big(\mathcal Lu(t),u(t)\big)+\frac{\beta}{2}\frac{d}{dt}v^2(t) \bigg ] dt \\
&=\int_{t_n}^{t_{n+1}}\Big [\big(\dot{u}(t),\mathcal Lu(t)\big)+\big(\mathcal Lu(t),\dot{u}(t)\big) +\beta v(t) \dot{v}(t)\Big] dt.
\end{align*}
The integrands $(\dot{u}(t),\mathcal Lu(t))+(\mathcal L{u}(t),\dot{u}(t))$  and $v(t)\dot{v}(t)$ are real polynomials of degree $2s-1$,
which can be integrated without error by the $s$-stage Gaussian quadrature formula. Thus it follows from the collocation condition that
\begin{align*}
&\int_{t_n}^{t_{n+1}}\Big [ \big(\dot{u}(t),\mathcal Lu(t)\big)+\big(\mathcal Lu(t),\dot{u}(t)\big) +\beta v(t) \dot{v}(t)\Big] dt\\
=&\tau \sum_{i=1}^{s}b_i \Big [ \big(\dot{u}(t_n^i),\mathcal Lu(t_n^i)\big)+\big(\mathcal Lu(t_n^i),\dot{u}(t_n^i)\big) +\beta v(t_n^i) \dot{v}(t_n^i)\Big]\\
=&2\tau {\text Re}\sum_{i=1}^{s}b_i \Big (\mathcal Lu(t_n^i)+\frac{\beta |u(t_n^i)|^2u(t_n^i)v(t_n^i)}{\sqrt{\big(u^2(t_n^i),u^2(t_n^i)\big)+C_0}},\dot{u}(t_n^i)\Big)\\
=&-2\tau {\text Im} \sum_{i=1}^{s}b_i \Bigg (\mathcal Lu(t_n^i)+\frac{\beta |u(t_n^i)|^2u(t_n^i)v(t_n^i)}{\sqrt{\big(u^2(t_n^i),u^2(t_n^i)\big)+C_0}},\mathcal Lu(t_n^i)+\frac{\beta |u(t_n^i)|^2u(t_n^i)v(t_n^i)}{\sqrt{\big(u^2(t_n^i),u^2(t_n^i)\big)+C_0}}\Bigg)\\
=&0,
\end{align*}
which yields $E^{n+1}=E^n, n=0,1,\cdots,N-1$. Similarly, we can obtain
\begin{align*}
M^{n+1}-M^n&=\int_{t_n}^{t_{n+1}}\frac{d}{dt}\big(u(t),u(t)\big)dt=\int_{t_n}^{t_{n+1}}[(\dot{u}(t),u(t))+(u(t),\dot{u}(t))]dt\\
&=2\tau {\text Re} \sum_{i=1}^{s}b_i(\dot{u}(t_n^i),u(t_n^i))\\
&=2\tau {\text Im}\sum_{i=1}^{s} b_i \bigg ( {\mathcal L}u(t_n^i) +\frac{\beta |u(t_n^i)|^2u(t_n^i)v(t_n^i)}{\sqrt{\big(u^2(t_n^i),u^2(t_n^i)\big)+C_0}} ,u(t_n^i)) \bigg )\\
&=0,
\end{align*}
which leads to $M^{n+1}=M^n, n=0,1,\cdots,N-1$. This completes the proof.
\end{proof}
\begin{Remark} It is remarked that any other symplectic Runge-Kutta method can preserve the quadratic invariant~\cite{cooper87,Sanzs88,Sanzs94}, thus, other arbitrarily high-order schemes which preserve the modified energy and the mass in \eqref{semi-discrte-energy-mass} can be easily obtained.
\end{Remark}

\section{Structure-preserving spatial discretization}
As sated as above, the semi-discrete {\bf Scheme~\ref{sch3.1}} can reach arbitrarily high-order in time, and exactly preserve the semi-discrete modified energy and mass, respectively.  In general, the numerical schemes are called structure-preserving if they can preserve the corresponding physical/geometric properties exactly after temporal and spatial full-discretizations. Thus, the structure-preserving spatial discretization is a major concern. In this paper, the standard Fourier pseudo-spectral method is chosen for spatial discretizations because of the high-order accuracy as well as the application of FFT technique \cite{ST06}. We show that the resulting fully-discrete schemes can preserve the energy and mass conservation laws in the fully discrete level.

To make the remaining part self-explanatory, we briefly reintroduce the following notations (see~\cite{cuijincpc} for more details).  For given even integers $N_x, N_y$ and $N_z$, the spatial domain $\mathcal D=[0,l_x]\times[0,l_y]\times[0,l_z]$ is uniformly partitioned with step sizes $h_x=l_x/N_x, h_y=l_y/N_y, h_z=l_z/N_z$, and the spatial grid points are denoted as follows:
\begin{align*}
 \Omega_h=\{(x_j,y_k,z_l)|x_j=jh_x, y_k=kh_y, z_l=lh_z, \ (j,k,l) \in \mathcal{T}_h    \},
\end{align*}
where the index set $\mathcal{T}_h$ is defined as
\begin{align*}
\mathcal{T}_h=\{ {\bm j}:=(j,k,l)| 0 \leqslant j \leqslant N_x-1, 0 \leqslant k \leqslant N_y-1, 0 \leqslant l \leqslant  N_z-1 \}.
\end{align*}
Let
\begin{align*}
{\mathbb V_h}=\{ U  | U =(&U_{0,0,0}, U_{1,0,0}, \cdots, U_{N_x-1,0,0}, \ \
                                  U_{0,1,0}, U_{1,1,0}, \cdots, U_{N_x-1,1,0}, \\
                                 &U_{0,N_y-1,0}, U_{1,N_y-1,0}, \cdots, U_{N_x-1,N_y-1,0}, \ U_{0,0,1}, U_{1,0,1}, \cdots, U_{N_x-1,0,1},\ \cdots, \\
                                 &U_{0,N_y-1,N_z-1}, U_{1,N_y-1,N_z-1}, \cdots, U_{N_x-1,N_y-1,N_z-1}  )^{\top}   \}
\end{align*}
be a vector space of grid functions defined on $\Omega_h$.
Note that the bold $\bm j \in \mathcal{T}_h $ refer to an index, while $j$ means the first component of $\bm j$.
For any two grid functions $u,v \in \mathbb V_h$, we define the discrete inner product and norm, respectively, as follows:
\begin{align*}
\langle u,v  \rangle _h:=h_1h_2h_3 \sum\limits_{{\bm j}\in \mathcal{T}_h}u_{{\bm j}} {\bar v}_{{\bm j}}, \ \ \lVert v \rVert_h=\sqrt{\langle v,v\rangle_h}, \ \ \lVert v \rVert_{\infty,h}= \mathop{\mathrm{max}}\limits_{{{\bm j}\in \mathcal{T}_h} }|v_{\bm j}|, \ \ \Vert v \Vert_{p,h}= \sqrt[p]{h_1h_2h_3 \sum\limits_{{\bm j}\in \mathcal{T}_h}| { v}_{{\bm j}}|^p },
\end{align*}
where ${\bar v}_{{\bm j}}$ refers to the conjugate of ${ v}_{\bm j}$. In addition, we denote $`\cdot$' as the componentwise product of the vectors, that is,
\begin{align*}
{u}\cdot {v}=&\big(u_{0,0,0}v_{0,0,0},\cdots,u_{N_x-1,0,0}v_{N_x-1,0,0},\cdots,u_{0,N_y-1,N_z-1}v_{0,N_y-1,N_z-1},\\
&\cdots,u_{N_x-1,N_2y-1,N_z-1}v_{N_x-1,N_y-1,N_z-1}\big)^{\top}.
\end{align*}
For brevity, we denote $u\cdot u$ as $u^2$.

Denote
\begin{align*}
S_N={\rm span}\{g_j(x)g_k(y)g_l(z), \ \ (j,k,l) \in  \mathcal{T}_h \}
\end{align*}
as the interpolation space, where $g_j(x), g_k(y)$ and $g_l(z)$ are trigonometric polynomials of degree $N_x/2, N_y/2$ and $N_z/2$, given respectively by
\begin{align*}
&g_j(x)=\frac{1}{N_x}\sum\limits_{p=-N_x/2}^{N_x/2}\frac{1}{a_p}e^{{ i} p \mu_x ( x-x_j ) }, \ \ \
g_k(y)=\frac{1}{N_y}\sum\limits_{q=-N_y/2}^{N_y/2}\frac{1}{b_q}e^{{ i} q \mu_y ( y-y_k ) }, \\
&g_l(z)=\frac{1}{N_z}\sum\limits_{r=-N_z/2}^{N_z/2}\frac{1}{c_r}e^{{ i} r \mu_z ( z-z_l ) },
\end{align*}
where $\mu_w=2\pi/l_w$, $w=x,y,z$ and
\begin{align*}
       a_p=\begin{cases} 1, \text{ $|p|<\frac{N_x}{2}$},\\
                         2, \text{ $|p|=\frac{N_x}{2}$},
           \end{cases}
      b_q=\begin{cases} 1, \text{ $|q|<\frac{N_y}{2}$},\\
                         2, \text{ $|q|=\frac{N_y}{2}$},
           \end{cases}
      c_r=\begin{cases} 1, \text{ $|r|<\frac{N_z}{2}$},\\
                         2, \text{ $|r|=\frac{N_z}{2}$}.
           \end{cases}
\end{align*}
We define the interpolation operator $I_N:  C(\mathcal{D}) \to S_N$ as follows \cite{chen01msf}:
\begin{align} \label{eq2.2}
I_N \psi(x,y,z)=\sum\limits_{j=0}^{N_x-1}\sum\limits_{k=0}^{N_y-1}\sum\limits_{l=0}^{N_z-1} \psi_{j,k,l} \ g_j(x)g_k(y) g_l(z),
\end{align}
where $\psi_{j,k,l}=\psi(x_j,y_k,z_l), \ g_j(x_m)=\delta_m^j, g_k(x_n)=\delta_n^k, g_l(x_s)=\delta_s^l.$ \\

 To derive $\partial_x^{s_1}\partial_y^{s_2}\partial_z^{s_3} I_N \psi(x_j,y_k,z_l)$ at the collocation points $(x_j,y_k,z_l)$, one can differentiate \eqref{eq2.2} to arrive that
 \begin{align*}
 \partial_x^{s_1}\partial_y^{s_2}\partial_z^{s_3} I_N \psi(x_j,y_k,z_l)
 =&\sum\limits_{m=0}^{N_x-1}\sum\limits_{n=0}^{N_y-1}\sum\limits_{s=0}^{N_z-1}\psi_{m,n,s}
 \frac{d^{s_1}g_m(x_j)}{dx^{s_1}}  \frac{d^{s_2}g_n(y_k)}{dy^{s_2}}  \frac{d^{s_3}g_s(z_l)}{dz^{s_3}}\\
 =&\bigg (\big(D^z_{s_3} \otimes D^y_{s_2} \otimes D^x_{s_1}\big){\bm \psi} \bigg)_{\bm j},
 \end{align*}
where $\otimes$ denotes the Kronecker product and $D^x_{s_1}$ is an $N_x\times N_x$ matrix, $D^y_{s_2}$ is an $N_y \times N_y$ matrix, and $D^z_{s_3}$ is an $N_z\times N_z $ matrix, with elements given by
 \begin{align*}
 (D_{s_1}^x)_{j,m}= \frac{d^{s_1}g_m(x_j)}{dx^{s_1}}, \ \ (D^y_{s_2})_{k,n}=\frac{d^{s_2}g_n(y_k)}{dy^{s_2}}, \ \
 ( D^z_{s_3})_{l,s}=\frac{d^{s_3}g_s(z_l)}{dz^{s_3}}.
 \end{align*}
 Note that $\big ((D^z_{s_3} \otimes D^y_{s_2} \otimes D^x_{s_1}) {\bm  \psi} \big)_{\bm j}$ refers to the $(N_xN_y(l-1)+N_x(k-1)+j)$-th component of the vector $(D^z_{s_3} \otimes D^y_{s_2} \otimes D^x_{s_1}) {\bm  \psi} , \ {\bm  \psi} \in {\mathbb V_h} $. For brevity, we use similar notations hereafter.
\\
\indent In particular, for first and second derivatives, we obtain
 \begin{align*}
 \begin{split}
 &\partial_{x}I_N \psi(x_j,y_k,z_l)=\big ((I_{N_z}\otimes I_{N_y} \otimes D_1^x) {\bm  \psi} \big)_{\bm j}, \
  \partial_{y}I_N \psi(x_j,y_k,z_l)=\big ((I_{N_z}\otimes D_1^y   \otimes I_{N_x}) {\bm  \psi} \big)_{\bm j},  \\
 &\partial_{x}^2I_N \psi(x_j,y_k,z_l)=\big ((I_{N_z}\otimes I_{N_y} \otimes D_2^x) {\bm  \psi} \big)_{\bm j}, \
  \partial_{y}^2I_N \psi(x_j,y_k,z_l)=\big ((I_{N_z}\otimes D_2^y\otimes I_{N_x}) {\bm  \psi} \big)_{\bm j},   \\
 &\partial_{z}^2I_N \psi(x_j,y_k,z_l)=\big ((D_2^z \otimes I_{N_y}  \otimes I_{N_x}) {\bm  \psi} \big)_{\bm j},  \
 \end{split}
 \end{align*}
 where $D_1^x, D_1^y$ are skew-symmetric matrices, $D_2^x, D_2^y, D_2^z$ are symmetric matrices, and we further have \cite{ST06}
 \begin{align} \label{4.0}
 \left \{
 \aligned
 &D_1^w=F_N^H \Lambda_1^w F_N, \  \Lambda_1^w=\text{i} \mu_w {\text {diag}}\Big(0,1,\cdots,\frac{N_w}{2}-1,0,-\frac{N_w}{2}+1,\cdots,-1\Big), \\
 &D_2^w=F_N^H \Lambda_2^w F_N, \  \Lambda_2^w=\Big [ \text{i} \mu_w {\text {diag}}\big(0,1,\cdots,\frac{N_w}{2}-1,\frac{N_w}{2},-\frac{N_w}{2}+1,\cdots,-1\big) \Big ]^2,
 \endaligned
 \right.
 \end{align}
 where $F_{N_w}$ is the discrete Fourier transform matrix with elements
$(F_{N_w})_{j,k}=\frac{1}{\sqrt{N_w}}e^{-i\frac{2\pi}{N_w}jk}$, $F_{N_w}^H$ is the conjugate transpose matrix of $F_{N_w}$, where $w=x,y,z$.

 For any $u \in {\mathbb V_h}$, we introduce the following spectral operators in the vector form
 \begin{align*}
 &\Delta_h u=(I_{N_z}\otimes I_{N_y} \otimes D_2^x)u+(I_{N_z}\otimes D_2^y\otimes I_{N_x})u+ (D_2^z \otimes I_{N_y}  \otimes I_{N_x})u, \\
  &{\mathbb D}_1^x u=(I_{N_z}\otimes I_{N_y} \otimes D_1^x)u, \  \
 {\mathbb D}_1^y  u=(I_{N_z}\otimes D_1^y \otimes I_{N_x})u, \ \
 I_x u=\big(I_{N_z} \otimes I_{N_y} \otimes X \big)u,  \\
  & I_y u=\big(I_{N_z} \otimes Y \otimes I_{N_x} \big)u, \ \
  L_z^hu=-{ \rm i}(I_x {\mathbb D}_1^y -I_y {\mathbb D}_1^x)u, \ \
 \mathcal L_h u=-\frac{1}{2}\Delta_h u +V\cdot u-\Omega  L_z^hu,
 \end{align*}
where $X={\rm diag}( x_0,x_1,\cdots,x_{N_x-1} )$ and $Y={\rm diag}( y_0,y_1,\cdots,y_{N_y-1} )$. It follows from {\bf Lemma 2.2} in~\cite{cuijincpc} that
\begin{align*}
\langle \Delta_h u,v \rangle_h=\langle u,\Delta_h v\rangle_h,  \ \langle L_z^h u,v \rangle_h=\langle u,L_z^h  v \rangle _h, \ u,v \in {\mathbb V_h},
\end{align*}
which implies that
 \begin{align}\label{4.2}
\langle \mathcal L_h u,v \rangle_h=\langle u,\mathcal L_h v \rangle_h, \ \ \langle \mathcal L_h u,u \rangle _h \in {\mathbb R}.
 \end{align}

\indent Applying the Fourier pseudo-spectral method in space for {\bf Scheme~\ref{sch3.1}}, we then obtain the following full discrete scheme.
\begin{Scheme} \label{sch4.1}
Let $c_1,\cdots,c_s$ be distinct real numbers $(0 \leqslant c_i \leqslant 1)$. For given $\Psi^n \in V_h$ and ${ q}^n \in {\mathbb R}$, we assume that ${ u}(t)$ is a ${N_x}\times {N_y}\times {N_z}$ dimensional vector polynomial of degree $s$ and ${ v}(t)$ is a polynomial of degree $s$ satisfying
\begin{align*}
&{ u}(t_n)=\Psi^n, \ { v}(t_n)={ q}^n, \\
&{\dot { u}}(t_n^i)=-\text{\rm i}\Bigg({\mathcal L_h}{ u}(t_n^i) +\frac{\beta |{ u}(t_n^i)|^2{ u}(t_n^i){ v}(t_n^i)}{\sqrt{\big \langle { u}^2(t_n^i),{  u}^2(t_n^i) \big \rangle _h+C_0}} \Bigg),\\
&{\dot  { v}}(t_n^i)=\frac{\big \langle {\dot { u}}(t_n^i),|{ u}(t_n^i)|^2{ u}(t_n^i)\big \rangle _h+\big \langle |{ u}(t_n^i)|^2{ u}(t_n^i),{\dot { u}}(t_n^i)\big \rangle _h}{\sqrt{\big \langle { u}^2(t_n^i),{ u}^2(t_n^i)\big \rangle _h+C_0}},
\end{align*}
where $t_n^i=t_n+c_i\tau, i=1,\cdots,s.$ Then the numerical solution is defined by $\Psi^{n+1}={u}(t_{n+1})$ and ${ q}^{n+1}={ v}(t_{n+1})$.
\end{Scheme}
Analogous to arguments in the semi-discrete scheme, we can derive the following result.
\begin{Theorem} \label{th4.1}  The fully discrete {\bf Scheme \ref{sch4.1}} can preserve the fully-discrete modified energy and mass, that is,
\vspace{-0.2cm}
\begin{align*}
E_h^{n}=E_h^0, \ M_h^{n}=M_h^0, \ n=1,\cdots,N,
\end{align*}
\vspace{-0.35cm}
where
\begin{align}\label{4.04}
M_h^n=\langle \Psi^n,\Psi^n \rangle _h,  \ E_h^{n}=\langle \mathcal L_h \Psi^n,\Psi^n \rangle _h+\frac{\beta}{2}( {q}^n)^2-\frac{\beta}{2}C_0.
\end{align}
\end{Theorem}
\begin{proof} The proof is similar to the {\bf Theorem \ref{semi-discrte-energy-mass}}. For brevity, we omit the details.
\end{proof}
\begin{Remark} If the standard Fourier pseudo-spectral method is applied to the system \eqref{2.1} for spatial discretizations, the discrete Hamiltonian energy at time level $t_n$  is given by
\begin{align} \label{4.05}
H_h^n=\langle \mathcal L_h \Psi^n,\Psi^n \rangle _h+\frac{\beta}{2}\| \Psi^n\|_{4,h}^4.
\end{align}
 However, we should note that the modified energy \eqref{modified-energy-conservation-law} is only equivalent to the Hamiltonian energy \eqref{energy-conservation-law} in the continuous sense, but not for the discrete sense. Thus, the proposed schemes cannot preserve such discrete Hamiltonian energy exactly.

\end{Remark}


\section{A fast solver for the proposed high-order schemes}

In this section, we develop a fast solver to implement {\bf Scheme~\ref{sch4.1}} efficiently. For brevity, we take the 2-stage Gauss method (i.e., $s=2$) for an example where the corresponding RK coefficients $a_{ij}, b_j, i,j=1,2$ are given in Table~\ref{tab1}. 

For given $\psi^n,q^n$, the 2-stage Gauss method can be rewritten as
\begin{align}\label{eq:5.1}
\left \{
 \aligned
&\Psi_1=\psi^n+\tau a_{11}k_1+\tau a_{12}k_2,\ \Psi_2=\psi^n+\tau a_{21}k_1+\tau a_{22}k_2,\\
&l_i=2\text{Re}\langle k_i,\Phi_i\rangle_h,\ \Phi_i= \frac{|\Psi_i|^2\cdot\Psi_i}{\sqrt{\|\Psi_i\|_{4,h}^4+C_0}},\  i=1,2, \\
&Q_1 = q^n+\tau a_{11}l_1+\tau a_{12}l_2,\ Q_2 = q^n+\tau a_{21}l_1+\tau a_{22}l_2,
\endaligned
 \right.
\end{align}
and
\begin{align}\label{eq:5.2}
&k_1=-\text{i}\mathcal{L}_h\psi^n-\text{i}\tau a_{11}\mathcal{L}_hk_1-\text{i}\tau a_{12}\mathcal{L}_hk_2-\text{i}\beta\Phi_1Q_1,\\\label{eq:5.3}
&k_2=-\text{i}\mathcal{L}_h\psi^n-\text{i}\tau a_{21}\mathcal{L}_hk_1-\text{i}\tau a_{22}\mathcal{L}_hk_2-\text{i}\beta\Phi_2Q_2,
\end{align}
where $\psi^{n+1}$ and $q^{n+1}$ are updated by
\begin{align}\label{eq:5.4}
&\psi^{n+1}=\psi^n+\tau \sum_{i=1}^2b_ik_i,\ q^{n+1}=q^n+\tau \sum_{i=1}^2b_il_i.
\end{align}

Recalling the linear operator $\mathcal{L}_hu:=-\frac{1}{2}\Delta_hu+  \mathcal{L}_{ h}^2 u$, where   $\mathcal{L}_h^2 u =V\cdot u-\Omega L_z^hu$, one can reformulate the equations \eqref{eq:5.2} and \eqref{eq:5.3} respectively as
\begin{align*}
&k_1-\frac{\text{i}\tau a_{11}}{2}\Delta_hk_1-\frac{\text{i}\tau a_{12}}{2}\Delta_hk_2=f_1(\psi^n,q^n,k_1,k_2),\\
&k_2-\frac{\text{i}\tau a_{21}}{2}\Delta_hk_1-\frac{\text{i}\tau a_{22}}{2}\Delta_hk_2=f_2(\psi^n,q^n,k_1,k_2),
\end{align*}
where
\begin{align*}
&f_i(\psi^n,q^{n},k_1,k_2)=-\text{i}\mathcal{L}_h\psi^n-\text{i}\tau a_{i1}\mathcal{L}_h^2k_1-\text{i}\tau a_{i2}\mathcal{L}_h^2k_2-\text{i}\beta\Phi_iQ_i  , \ i=1,2.
\end{align*}

Then, we apply the fixed-point iteration method to solve the nonlinear algebraic equations as above. For iteration step $s$, we have
\begin{align}\label{eq:5.5}
&k_1^{s+1}-\frac{\text{i}\tau a_{11}}{2}\Delta_hk_1^{s+1}-\frac{\text{i}\tau a_{12}}{2}\Delta_hk_2^{s+1}=f_1(\psi^n,q^n,k_1^{s},k_2^{s}),\\\label{eq:5.6}
&-\frac{\text{i}\tau a_{21}}{2}\Delta_hk_1^{s+1}+k_2^{s+1}-\frac{\text{i}\tau a_{22}}{2}\Delta_hk_2^{s+1}=f_2(\psi^n,q^n,k_1^{s},k_2^{s}).
\end{align}
For brevity, we denote
\begin{align*}
&\widetilde{k}_i=(F_{N_z}\otimes F_{N_y}\otimes F_{N_x} ) {k}_i, \ \ \widetilde{f}_i=(F_{N_z}\otimes F_{N_y}\otimes F_{N_x}) {f}_i, \ i=1,2, \\
& \widetilde{\Delta}_h=\Lambda_{2}^{z}\otimes I_{N_y}\otimes I_{N_x}+I_{N_z}\otimes \Lambda_{2}^{y}\otimes I_{N_x}+I_{N_z}\otimes I_{N_y}\otimes \Lambda_{2}^{x}.
\end{align*}
Multiplying both sides of~\eqref{eq:5.5} and \eqref{eq:5.6} by the matrix $F_{N_z} \otimes F_{N_y} \otimes F_{N_x} $, respectively, we then obtain from \eqref{4.0}, together with the definitions of $\Delta_h$ and $\otimes$,  that
\begin{align*}
&\widetilde{k}_1^{s+1}-\frac{\text{i}\tau a_{11}}{2}\widetilde{\Delta}_h\widetilde{k}_1^{s+1}-\frac{\text{i}\tau a_{12}}{2}\widetilde{\Delta}_h\widetilde{k}_2^{s+1}=\widetilde{f}_1(\psi^n,q^n,k_1^{s},k_2^{s}),\\
&-\frac{\text{i}\tau a_{21}}{2}\widetilde{\Delta}_h\widetilde{k_1}^{s+1} + \widetilde{k}_2^{s+1}-\frac{\text{i}\tau a_{22}}{2}\widetilde{\Delta}_h\widetilde{k}_2^{s+1}=\widetilde{f}_2(\psi^n,q^n,k_1^{s},k_2^{s}),
\end{align*}
which implies the following relation
\begin{align*}
&\Big[1-\frac{\text{i}\tau a_{11}}{2}\big(\lambda_{2,j}^x+\lambda_{2,k}^y+\lambda_{2,l}^z\big)\Big](\widetilde{k}_1)_{\bm j}^{s+1}
-\frac{\text{i}\tau a_{12}}{2}\big(\lambda_{2,j}^x+\lambda_{2,k}^y+\lambda_{2,l}^z\big)(\widetilde{k}_2)_{\bm j}^{s+1}=(\widetilde{f}_1)_{\bm j}(\psi^n,q^n,k_1^{s},k_2^{s}),\\
&-\frac{\text{i}\tau a_{21}}{2}\big(\lambda_{2,i}^x+\lambda_{2,j}^y+\lambda_{2,k}^z\big)(\widetilde{k}_1)_{\bm j}^{s+1}+
\Big[1-\frac{\text{i} \tau a_{22}}{2}\big(\lambda_{2,j}^x+\lambda_{2,k}^y+\lambda_{2,l}^z\big)\Big](\widetilde{k}_2)_{\bm j}^{s+1}=(\widetilde{f}_2)_{\bm j}(\psi^n,q^n,k_1^{s},k_2^{s}),
\end{align*}
where ${\bm j}\in {\mathcal T}_h$ and $\lambda_{2,j}^w$ represents the $j$-th eigenvalues of the spectral differential matrix  $D_2^w,\ w=x,y,z$ (see \eqref{4.0}).
For a given ${\bm j}\in\mathcal{T}$, the above equations derive  a $2\times 2$ linear system for the unknowns $((\widetilde{k}_1)_{\bm j}^{s+1},(\widetilde{k}_2)_{\bm j}^{s+1})^T$.

Solving above linear system for all ${\bm j}\in\mathcal{T}$, we can obtain $\widetilde{k}_1^{s+1}$ and $\widetilde{k}_2^{s+1}$,  then the relation {$k_i^{s+1}=\big(F_{N_z}^H \otimes F_{N_y}^H \otimes F_{N_x}^H \big) \widetilde  k_i^{s+1}$} further gives ${k}_i^{s+1}, i=1,2$. In practical computation, the iteration terminates if the infinity norm of the error between two adjacent iterative steps is less than $10^{-14}$, that is,
\begin{align*}
\mathop{\rm max}\limits_{l\leqslant i\leqslant 2} \big \{  \lVert k_i^{s+1}-k_i^s \rVert_{\infty,h} \big \} < 10^{-14}.
\end{align*}
Subsequently, $l_i,\ i=1,2$ is calculated by \eqref{eq:5.1}. Finally, $\psi^{n+1}$ and $q^{n+1}$ are updated from \eqref{eq:5.4}.

\begin{Remark} We should note the following facts: (i) the SAV approach needs to introduce an
auxiliary variable, but it can be eliminated in practical computation;
 (ii) the related data, in our practical computation, are stored in three-dimensional arrays instead of vectors, thus the fast Fourier transform (FFT) algorithms can be employed to speed up the process;
 (iii) small modifications would allow us to efficiently implement arbitrary stage RK methods.
\end{Remark}


\section{Numerical examples}
In this section, some numerical examples are carried to investigate the accuracy, CPU time and invariants-preservation of the proposed schemes.
As shown above, the newly proposed scheme~\ref{sch4.1}, which preserves the discrete mass and modified energy precisely, could reach arbitrarily high-order accuracy in time. Next, we take for example the 4th- and 6th-order Gaussian collocation methods, denoted by 4th-order HSAV and 6th-order HSAV, respectively. The numerical results would be compared with the  Crank-Nicolson finite difference (CNFD) method~\cite{baoc}, semi-implicit finite difference (SIFD) method~\cite{baoc}, and the linearly implicit Fourier pseudo-spectral (LIFP) method~\cite{cuijincpc}. In addition, the convergent rate is obtained by the following formula
\begin{align*}
\quad {\rm  Rate}=\frac{{\rm ln} \big(error_1/error_2)}{{\rm ln} (\delta_1/\delta_2)},
\end{align*}
where $\delta_l,error_l \ (l=1,2)$ are step sizes and errors with step size $\delta_l$, respectively.\\

\noindent \textbf{Example 6.1} In this example, we mainly investigate the temporal accuracy and computational efficiency of the proposed schemes in 3D, and take the initial condition $\psi_0$ and the external trap potential $V(\bf x)$ in~\eqref{eq1.1}-\eqref{eq1.3} as
\begin{align*}
& \psi_0(x,y,z)=\frac{(\gamma_x\gamma_y\gamma_z)^{1/4}}{ 2{\pi}^{3/4}}e^{-V(x,y,z)}, \\
& V(x,y,z)=(\gamma_x^2 x^2+ \gamma_y ^2 y^2+\gamma_z^2 z^2)/2,
\end{align*}
and choose $\mathcal{D}=[-8,8]^3, \Omega=0.7$ and $\gamma_x=\gamma_y=\gamma_z=1.0$. For comparison, the numerical ``exact" solution $\psi_e$ is obtained by the 6th-order HSAV method with $\tau=10^{-3}$ and $h=1/8$.
Let $e(\tau, h)$ be the error of numerical solution with mesh size $h$ and time step $\tau$.
We compute the discrete $L^{\infty}$ errors between the numerical ``exact" solution and the numerical solution by the 4th- and 6th-order HSAV methods, respectively. Moreover, the mass error, Hamiltonian energy error and quadratic energy error on time level $t_n$ will be calculated by the following formulas:
\begin{align*}
 e(M^n):=|M_h^n- M_h^0|, \  e(H^n):=|H_h^n- H_h^0|, \ e(E^n):=|E_h^n- E_h^0|, \ \ n=1,2,\cdots,N,
\end{align*}
respectively.

\indent Numerical results are shown in Table~\ref{Tab2} with different values of $\beta$. As is illustrated that the 4th- and 6th-order HSAV methods arrive at fourth-order and sixth-order convergence rates in time, respectively. Moreover, for a given  time step and mesh size, the numerical errors are observed to increase along with the growth of $\beta$. In fact, the increase in $\beta$ can cause more
vortices, and the lattice will thereby becomes much dense. In this case, the high-order numerical algorithms, such as the 6th-order HSAV, can show their obvious advantages in practical computations to obtain a given high accuracy.

\vspace{-3mm}

\begin{table}[H]
\footnotesize
\caption{\footnotesize Temporal errors of the numerical solutions with $t=3,N=32$. \label{Tab2}}
\centering
\begin{tabular}{@{}  c  @{}  c  @{}  c  @{}  c  @{} c  @{} c  @{}  }
\toprule
 \qquad{ } &\qquad{ } \qquad{ }  & \quad{ } $\tau=0.02$ { }& \quad{ }  $\tau=0.015$  { }& \quad{ } $ \tau=0.01$ \quad{ } & { } $ \tau=0.005$    \quad{ } \\
\hline
 \hspace{2.9cm} $\beta=20$         &$\lVert e \rVert_{\infty}$  & 6.0575e-007 &  1.9225e-007 & 3.8033e-008 & 2.3615e-009        \\
                          &{\small Rate}           &* & 3.99 &3.99 & 4.00      \\
 4th-order HSAV \hspace{0.55cm} $\beta=50$        &$\lVert e \rVert_{\infty}$  &2.6576e-006 & 8.4483e-007 & 1.6728e-007 &1.0391e-008    \\
                          &{\small Rate}            & *&3.99 & 3.99& 4.00  \\
 \hspace{2.9cm} $\beta=100$       &$\lVert e \rVert_{\infty}$  & 1.0359e-005 & 3.3007e-006 & 6.5436e-007 &4.0667e-008    \\
                          &{\small Rate}            & *&3.98 & 3.99& 4.00  \\
 \hline
  \hspace{2.9cm}$\beta=20$ &      $\lVert e \rVert_{\infty}$  & 5.8626e-010  & 1.0460e-010  & 9.2088e-012   & 1.4132e-013  \\
                          &{\small Rate}           & *  & 5.99 &   5.99    &  6.02   \\
  6th-order HSAV  \hspace{0.55cm} $\beta=50$       &$\lVert e \rVert_{\infty}$  & 3.4924e-009 & 6.2332e-010   & 5.4876e-011  & 8.5756e-013\\
                          &{\small Rate}            &* &5.99 & 5.99 & 6.00  \\
  \hspace{2.9cm}$\beta\!=\!100$       &$\lVert e \rVert_{\infty}$  & 1.6989e-008 & 3.0331e-009  & 2.6685e-010  & 4.1656e-012\\
                          &{\small Rate}            &* &5.99  & 5.99  & 6.00  \\
\toprule
\end{tabular}
\end{table}

\vspace{-3mm}

Subsequently, some comparisons are made with other algorithms in the literature. In Figure~\ref{fig5.-1} (a), we present the $L^{\infty}$-norm solution error versus the execution time for different schemes. As is shown that the high-order HSAV methods are more effective than other second-order ones. Furthermore, Figure~\ref{fig5.-1} (b)-(d) investigate the errors of invariants of different methods in the long-time behaviour, where we choose $\tau=0.01, T = 20$. As demonstrated in Figure~\ref{fig5.-1} (b) that all the numerical methods can preserve the discrete mass exactly except the SIFD method. Subsequently, we study the conservation of discrete energy during the evolution, and it can be observed from Figure~\ref{fig5.-1} (c) that the HSAV schemes preserve the Hamiltonian energy much better than the SIFD and LIFP methods except the CNFD method, which possess the precise energy conservation law. In particular, Figure~\ref{fig5.-1} (d) demonstrates that the proposed schemes preserve the quadratic energy exactly, which conforms the preceding theoretical analysis.

\begin{figure}[H]
\begin{minipage}{0.48\linewidth}
  \centerline{\includegraphics[height=5.4cm,width=0.88\textwidth]{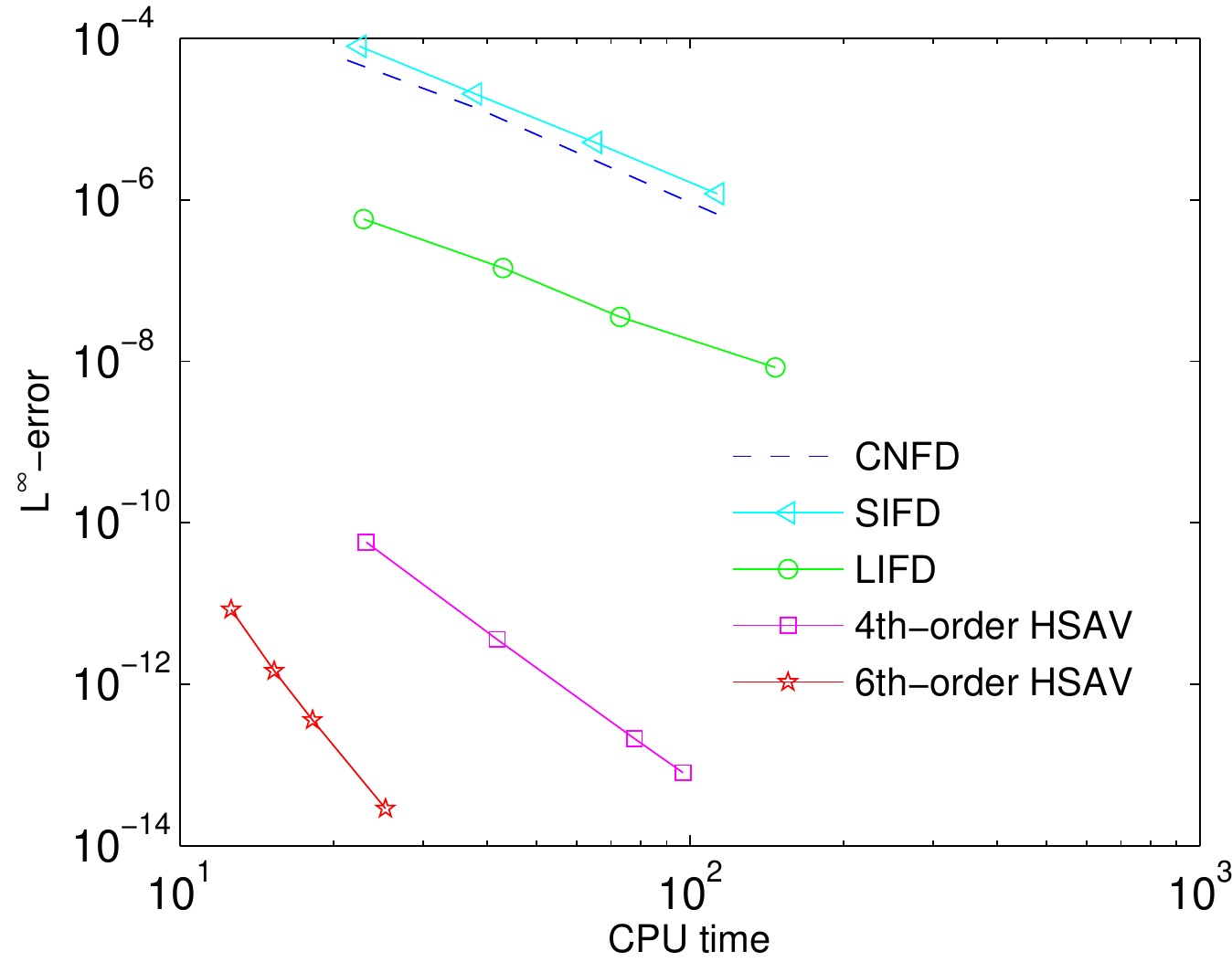}}
  \centerline{\footnotesize  (a) Solution error vs. execution time}
  \label{fig1.-33}
\end{minipage}
\hfill
\begin{minipage}{0.48\linewidth}
  \centerline{\includegraphics[height=5.4cm,width=0.88\textwidth]{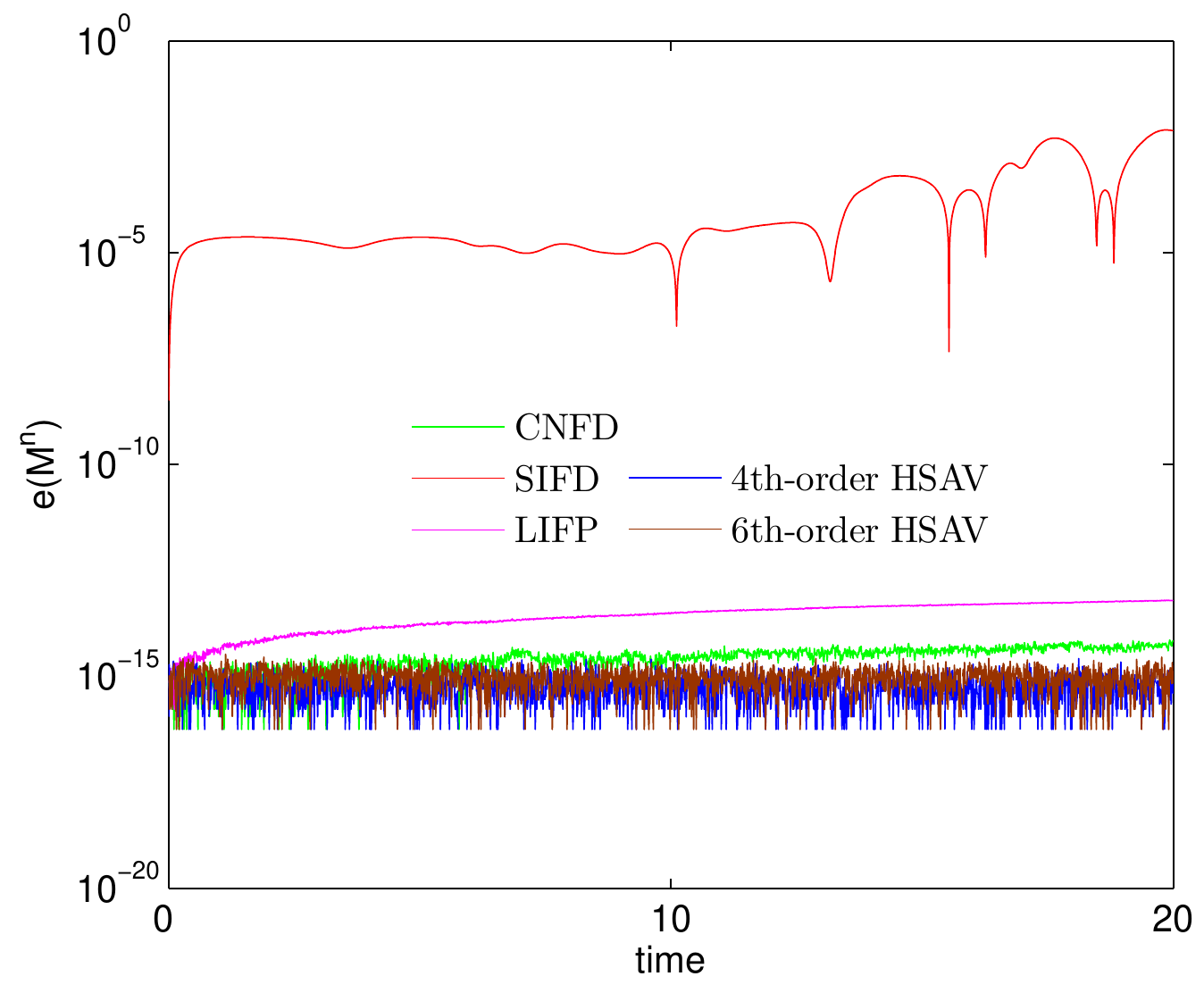}}
  \centerline{  \footnotesize(b) Mass error}
   \label{fig1.-33}
\end{minipage}
\end{figure}
\vspace{-2mm}
\begin{figure}[H]
\begin{minipage}{0.48\linewidth}
  \centerline{\includegraphics[height=5.4cm,width=0.87\textwidth]{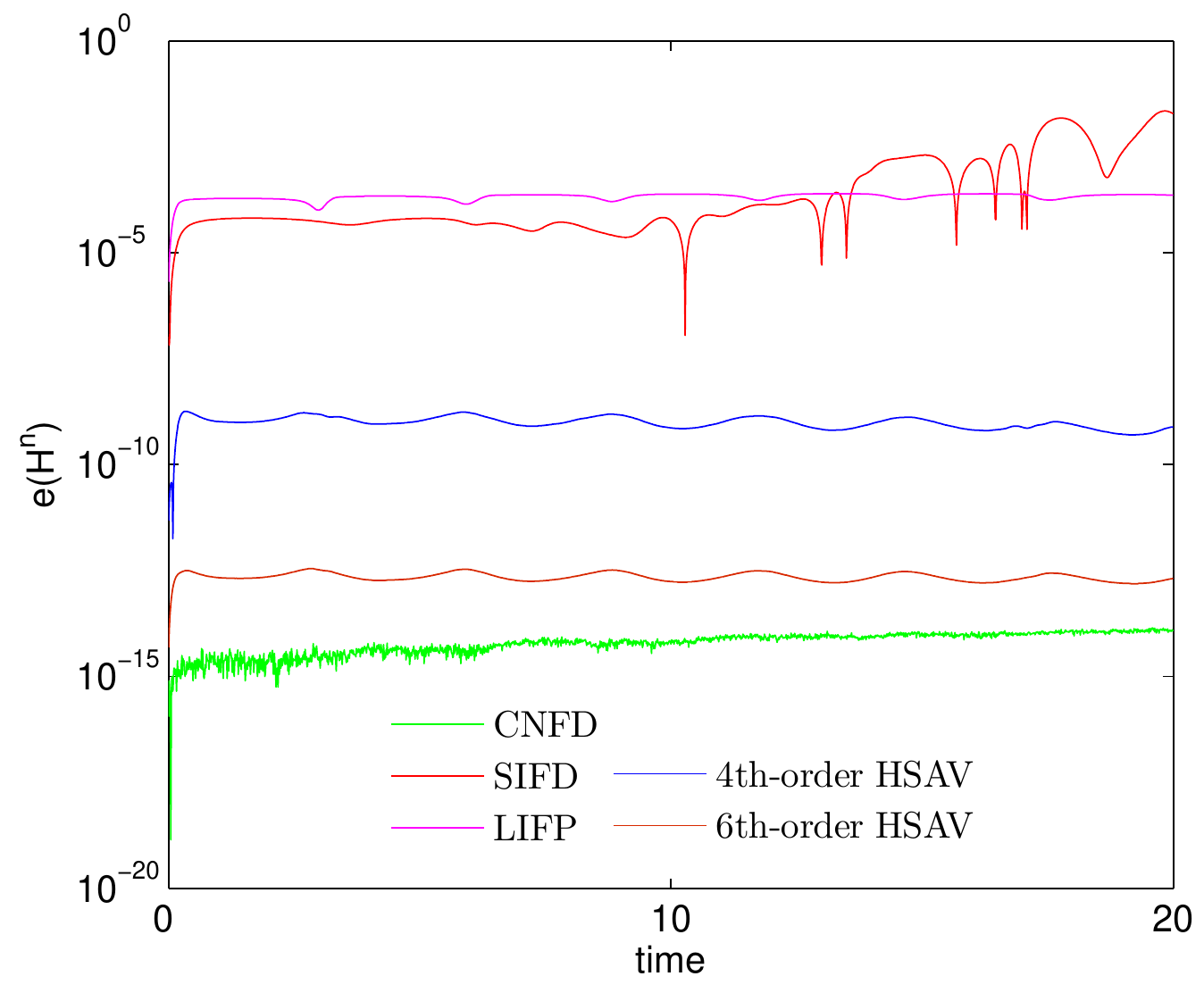}}
  \centerline{  \footnotesize(c) Hamiltonian energy error}
  \label{fig1.-33}
\end{minipage}
\hfill
\begin{minipage}{0.48\linewidth}
  \centerline{\includegraphics[height=5.2cm,width=0.9\textwidth]{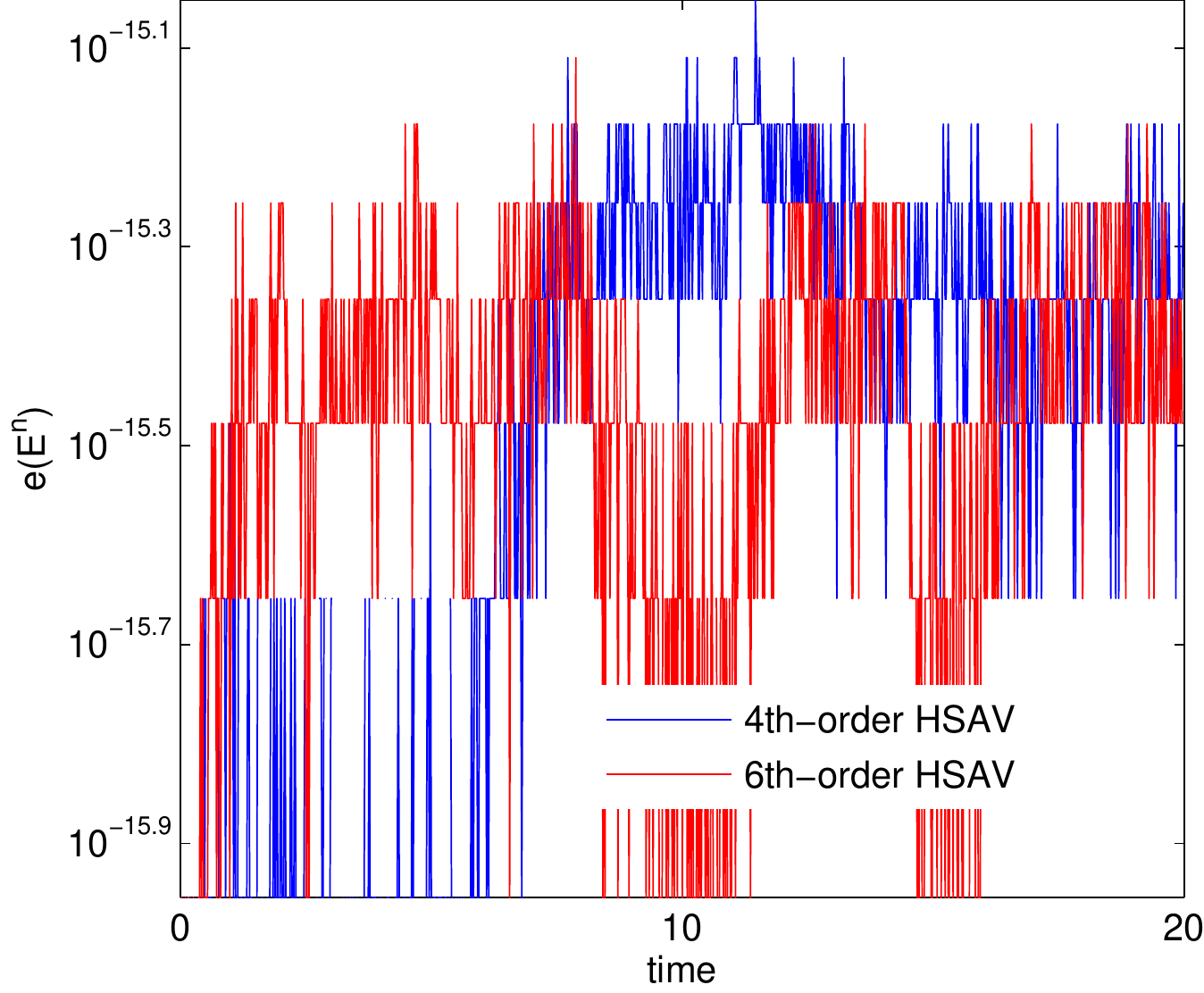}}
  \centerline{\footnotesize  (d) Quadratic energy error}
   \label{fig1.-33}
\end{minipage}
\caption{\footnotesize Comparisons of different numerical schemes }
\label{fig5.-1}
\end{figure}

\noindent \textbf{Example 6.2} In this example, we further consider the dynamics of the 2D GP equation starting from a quantized vortex lattice for rotating BECs~\cite{Baow}, i.e. we here choose $\Omega=0.9, \beta=1000$, $V(x,y)=\frac{1}{2}(x^2+y^2)$ and the spatial domain $\mathcal D=[-16,16]^2$ with mesh size $h=32/128$. The ground state is computed numerically by the backward Euler pseudo-spectral (BESP) method,  provided in the GPELab program~\cite{antoine14}, with the same parameter values and $\gamma_x=1,\gamma_y=1$.

Then, the dynamics of the vortex lattice is investigated numerically by perturbing the harmonic potential $V(x,y)=\frac{1}{2}(\gamma_x^2x^2+\gamma_y^2y^2)$ with different parameters: (i) case $\rm \MyRoman{1}$: $\gamma_x=\gamma_y=1.4$, (ii) case $\rm \MyRoman{2}$: $\gamma_x=1.1$, and $\gamma_y=0.9$, respectively. For brevity, we only present the contour plots of the density
function $|\psi|^2$ for the dynamics of vortex lattices computed by 4th-order HSAV method, and that of 6th-order counterpart is similar apparently. From Figure~\ref{fig5.-21}, we can observe massive quantized vortices in the ground state at $t=0$. During the time evolution,  the lattice structures are all conserved due to the high accuracy and efficiency of the proposed methods, and the lattice shrinks or expands on account of the changing of the trapping frequencies. Meanwhile, the vortex lattice is clearly observed to rotate clockwise around the center. On the other hand, because of the increase and decrease of $\gamma_x$ and $\gamma_y$ in case $\rm \MyRoman{2}$,  the condensate in Figure~\ref{fig5.-22} is observed to contract and expand in $x$- and $y$-directions, respectively.

Moreover, we inspect the long-time behaviour by carrying out a large time period. As shown in Figure~\ref{fig5.-23} that the proposed schemes can preserve the discrete mass and quadratic energy precisely, and 6th-order HSAV scheme performs more accurate than 4th-order counterpart in terms of the Hamiltonian energy, which conforms the preceding theoretical analysis again.

\begin{figure}[H]
    \centering
             \includegraphics[height=3.5cm,width=0.23\linewidth]{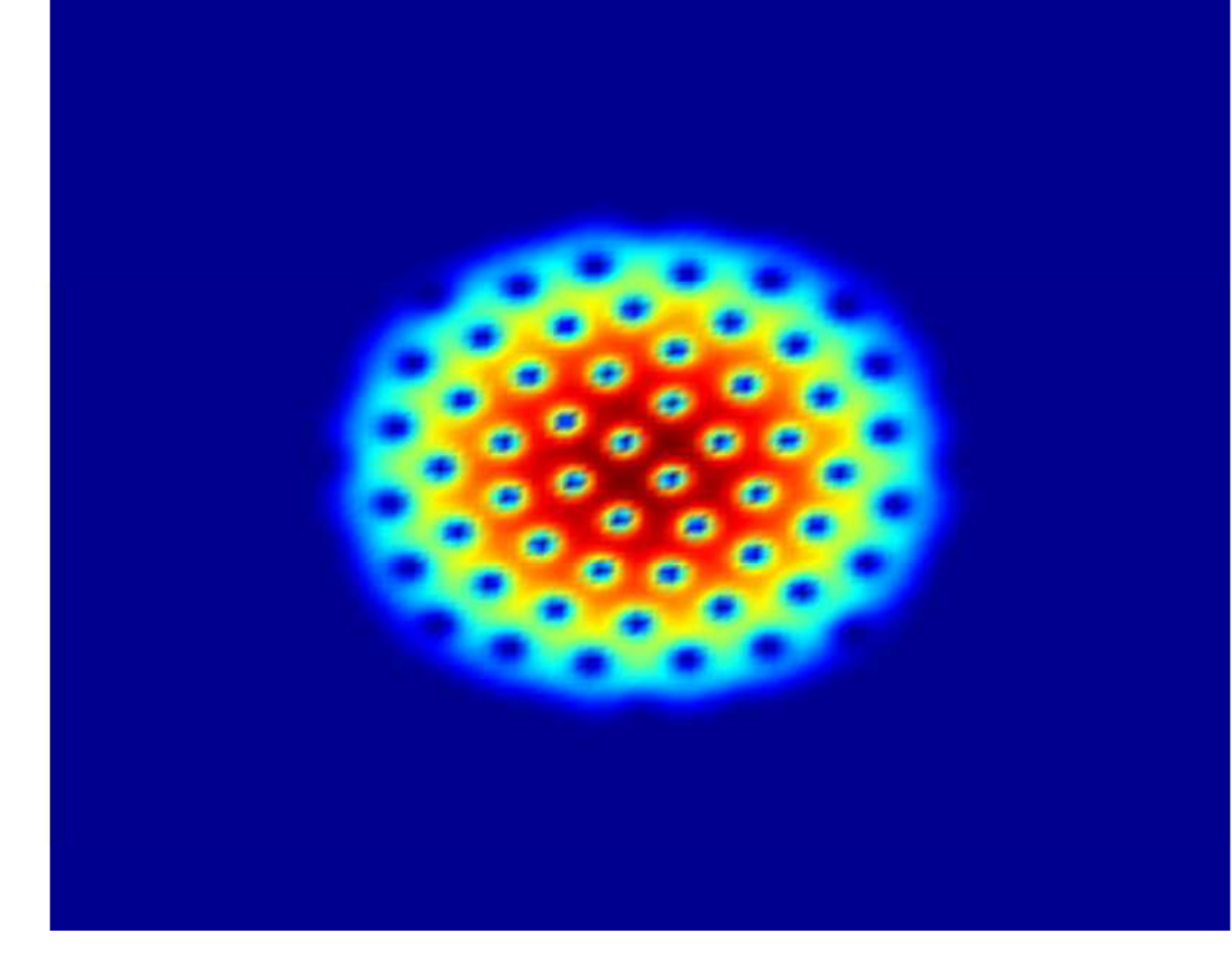}
             \includegraphics[height=3.5cm,width=0.23\linewidth]{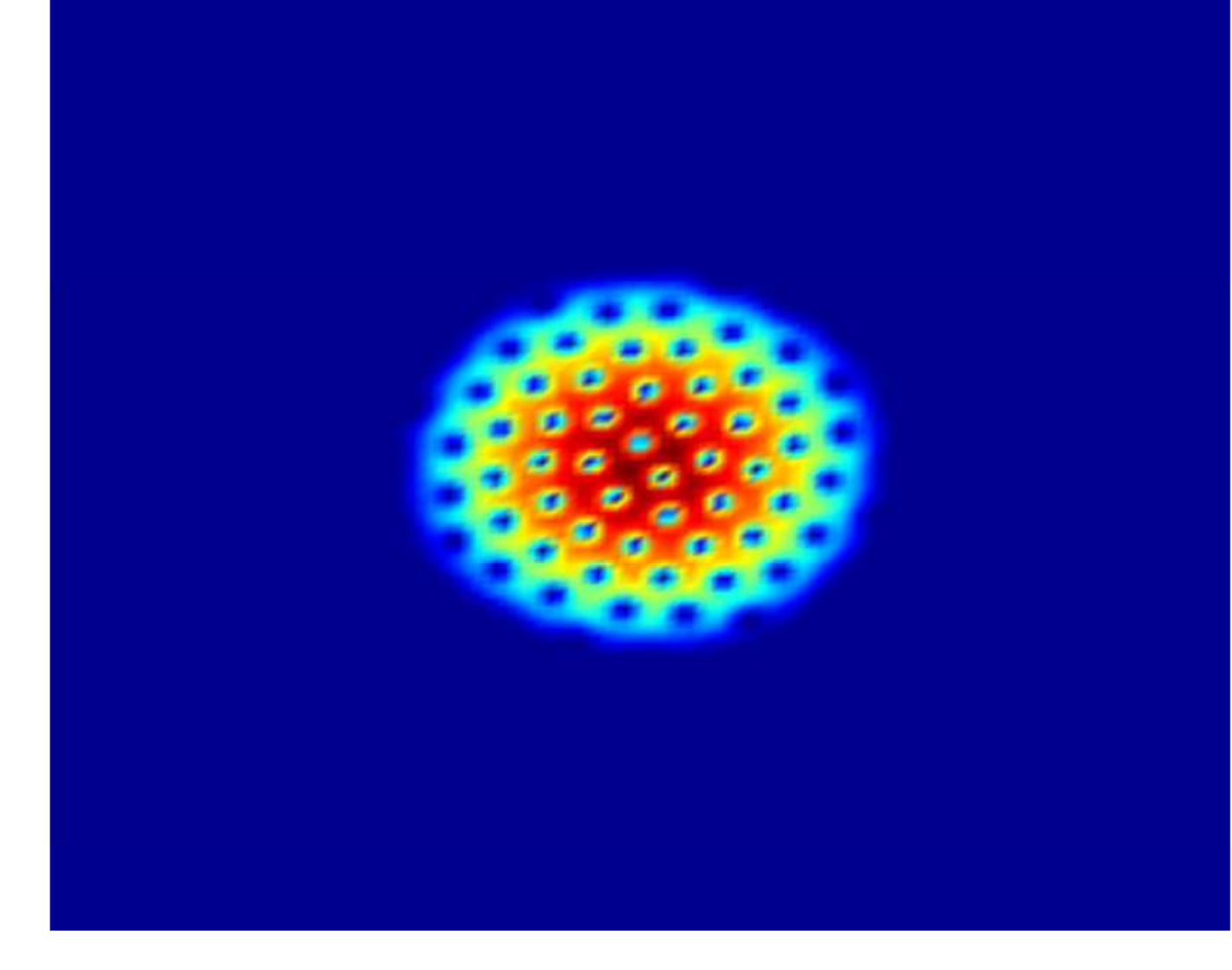}
	         \includegraphics[height=3.5cm,width=0.23\linewidth]{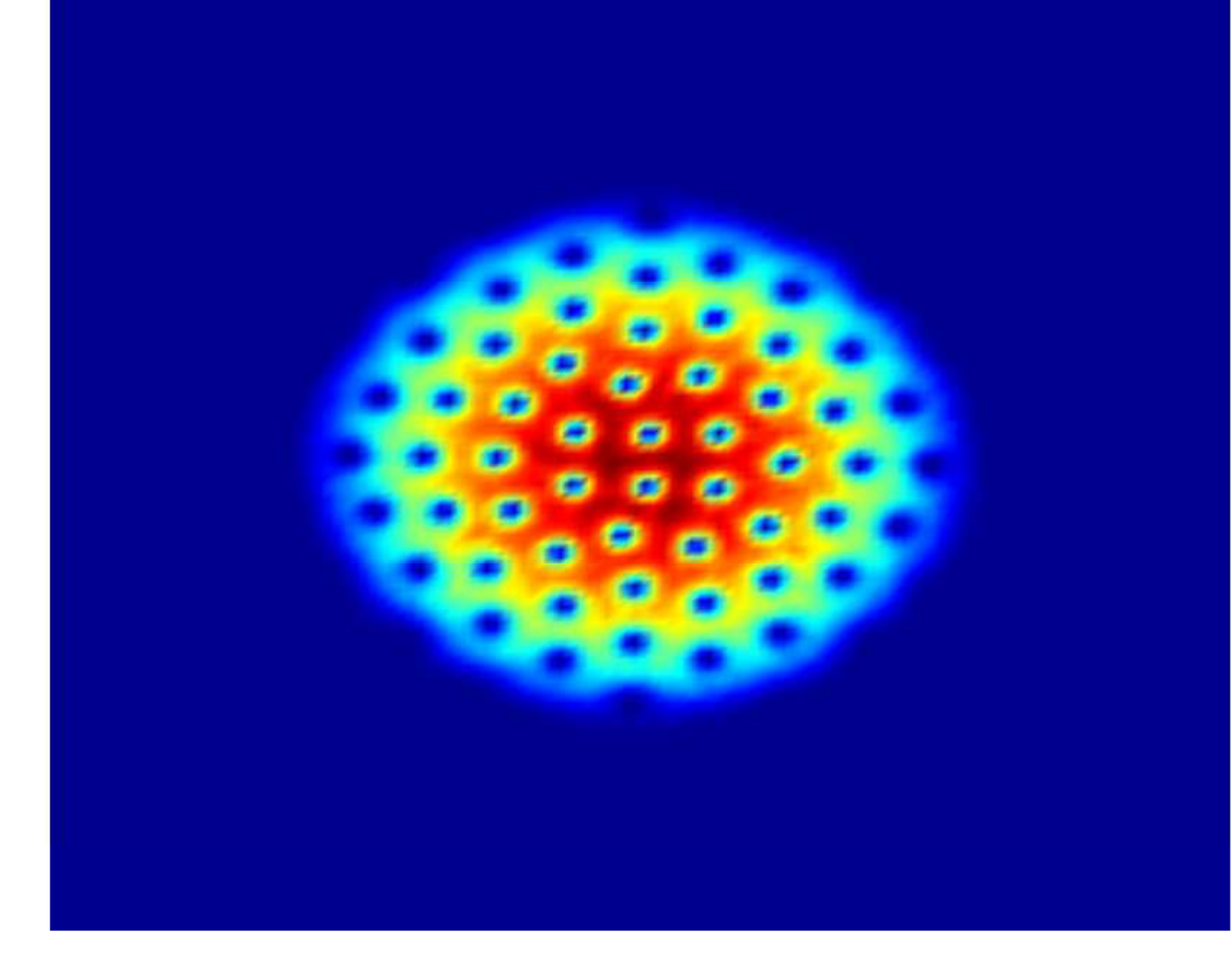}
             \includegraphics[height=3.5cm,width=0.23\linewidth]{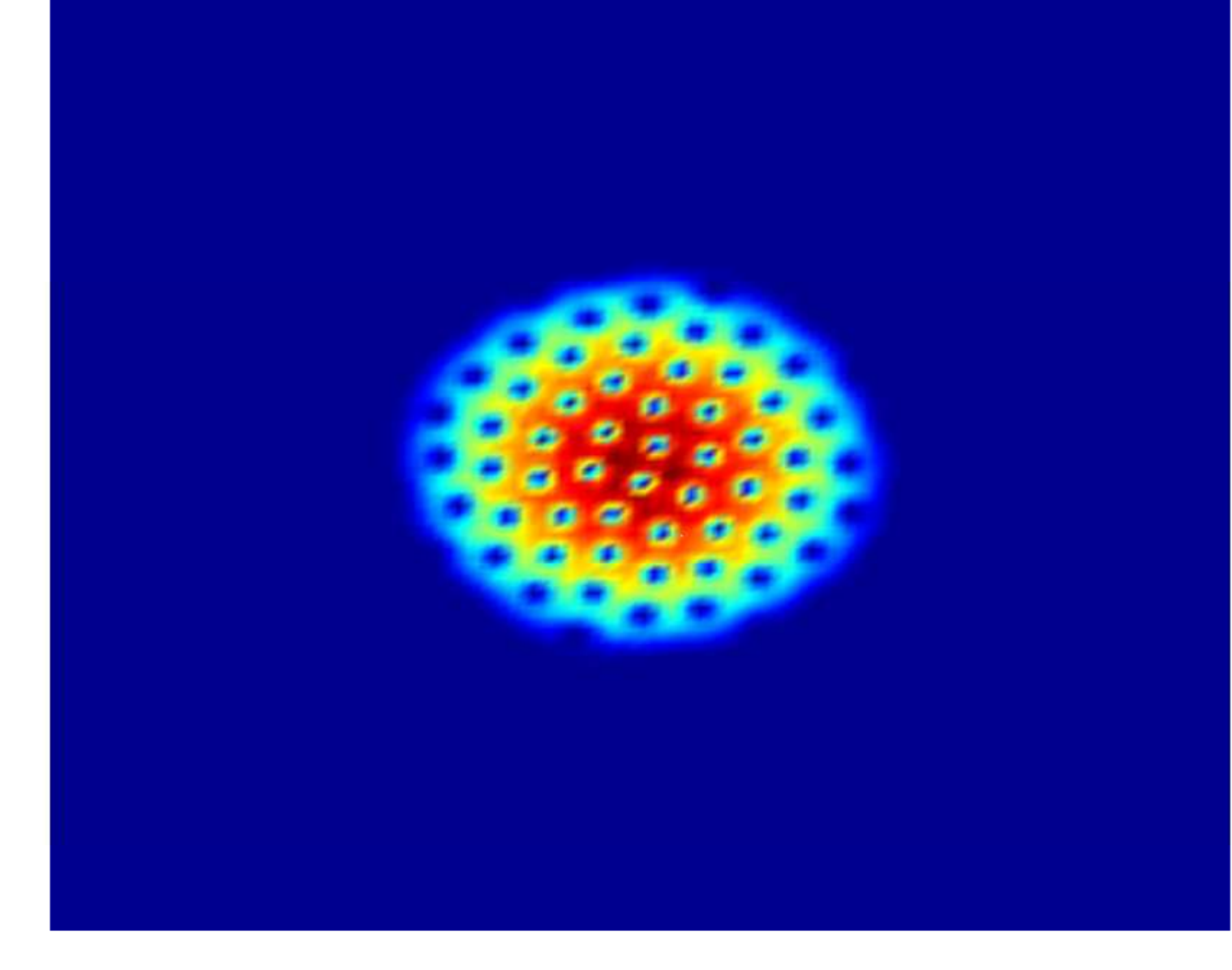}\\
	         ~\includegraphics[height=3.5cm,width=0.23\linewidth]{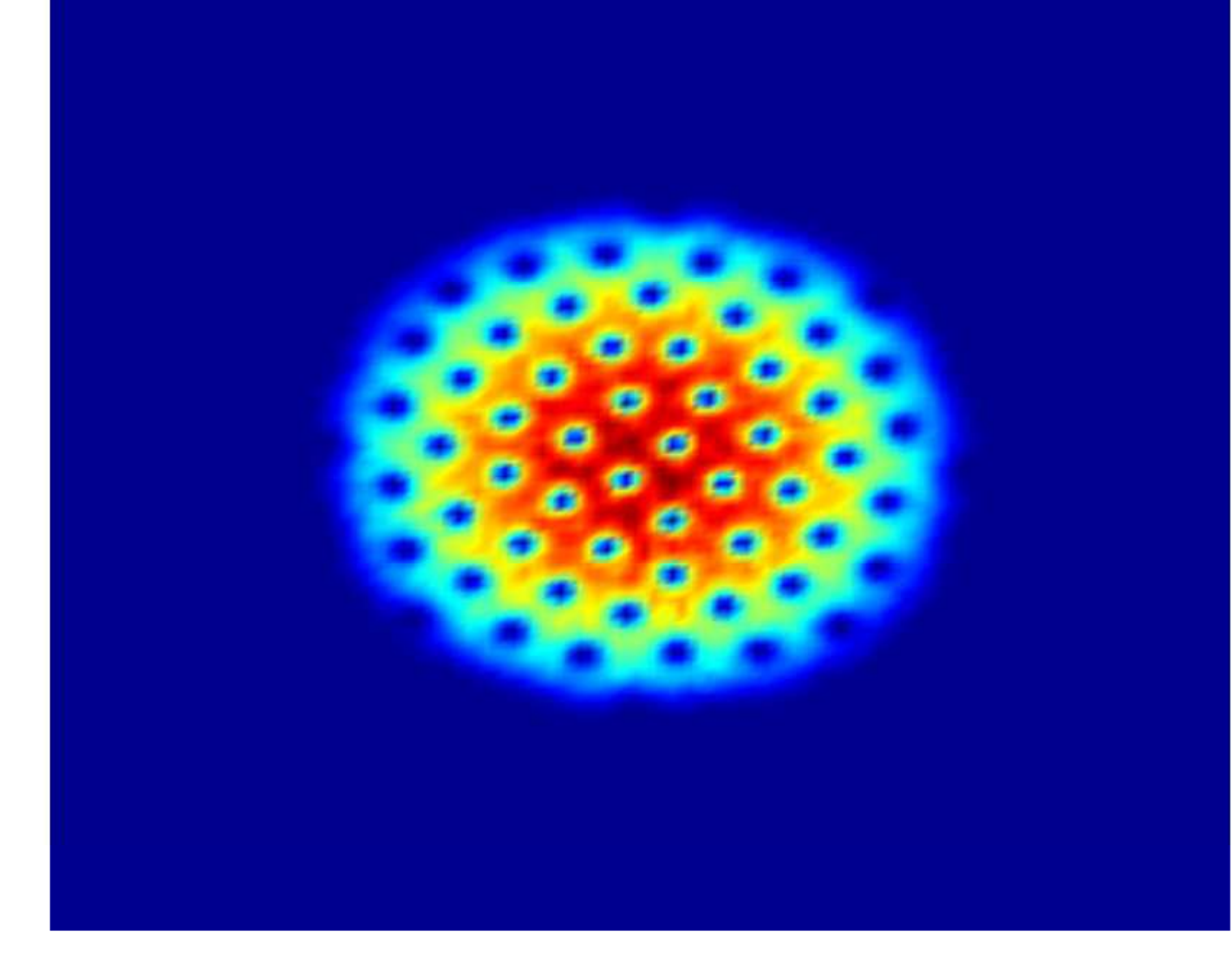}
             \includegraphics[height=3.5cm,width=0.23\linewidth]{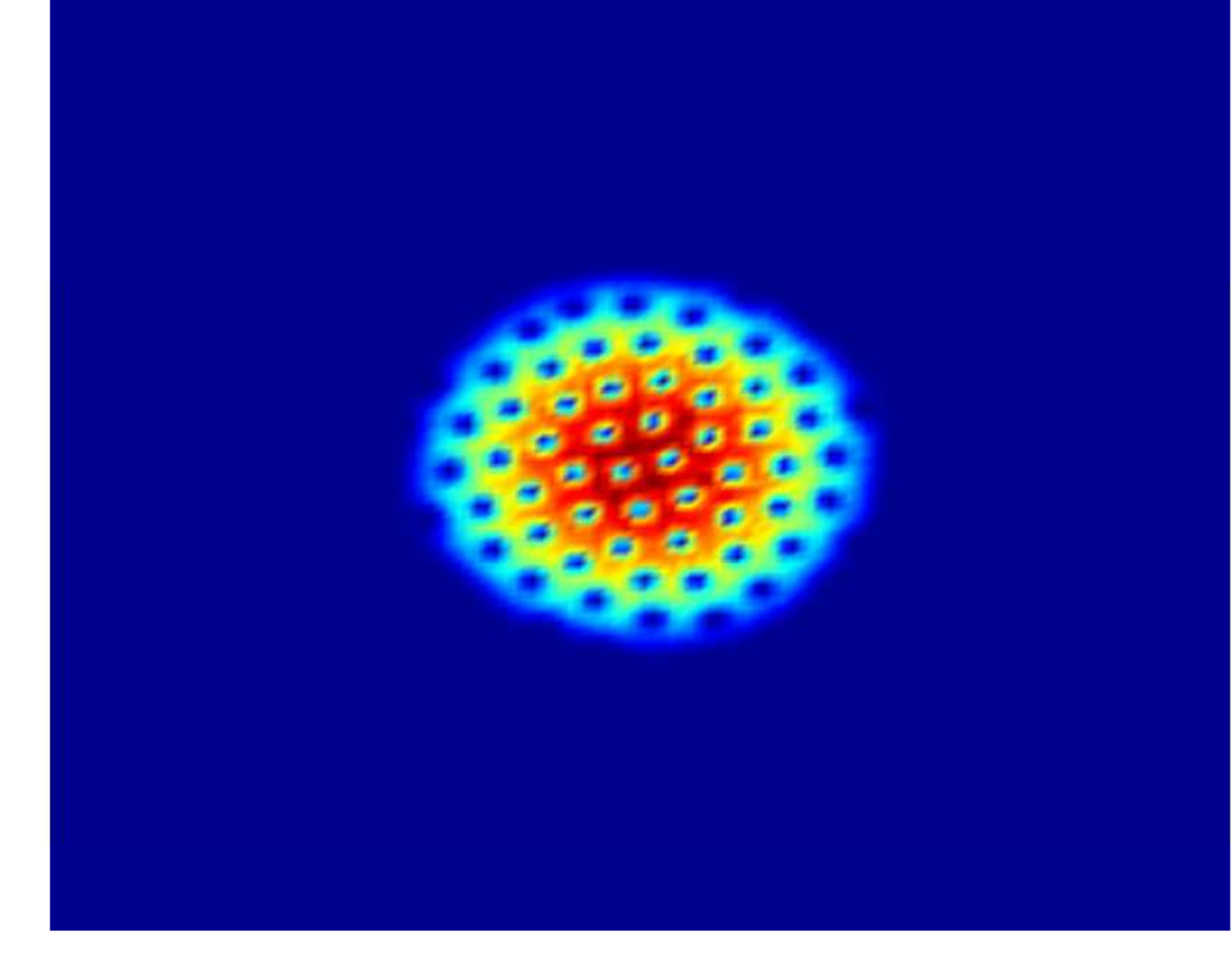}
	         \includegraphics[height=3.5cm,width=0.23\linewidth]{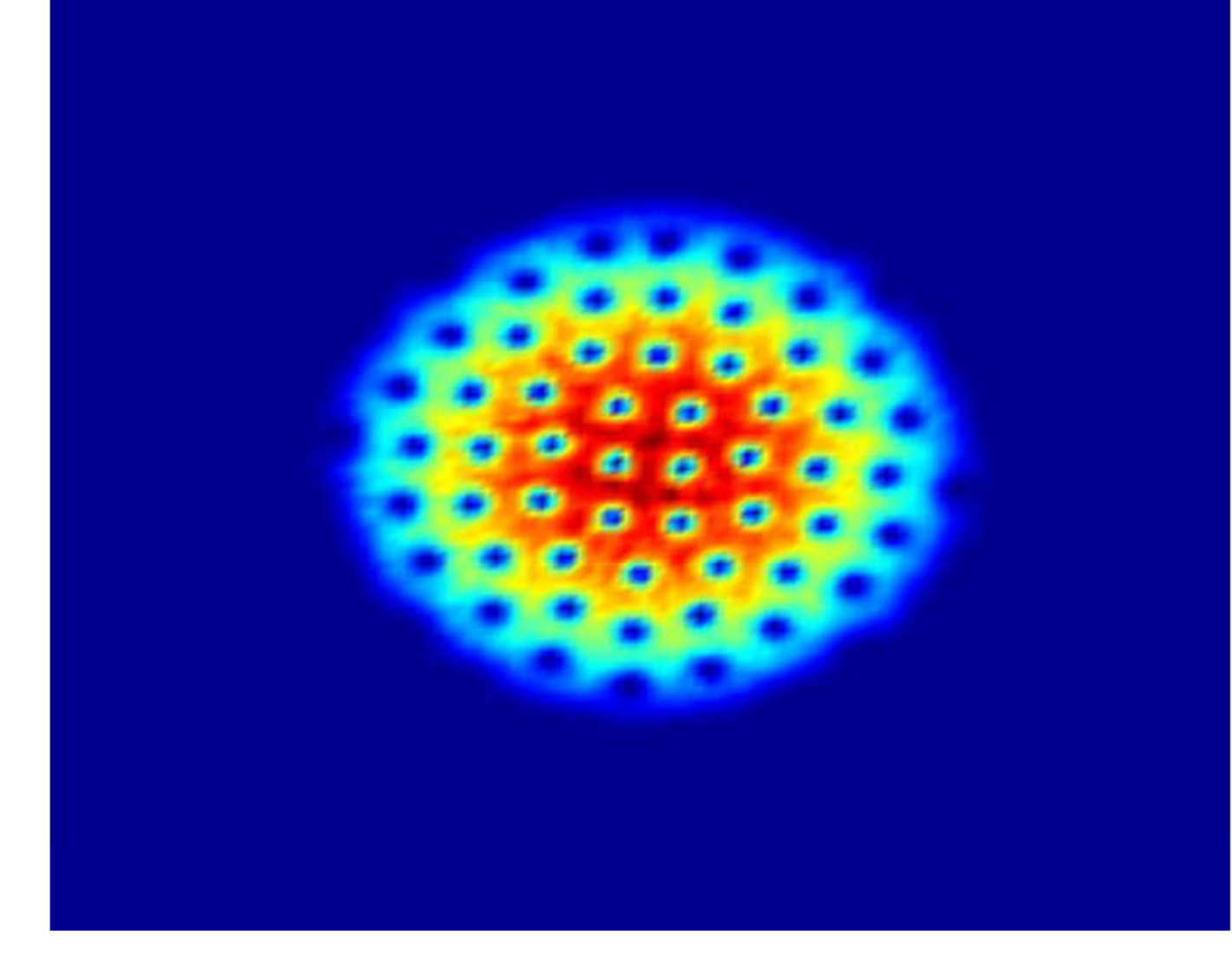}
             \includegraphics[height=3.5cm,width=0.23\linewidth]{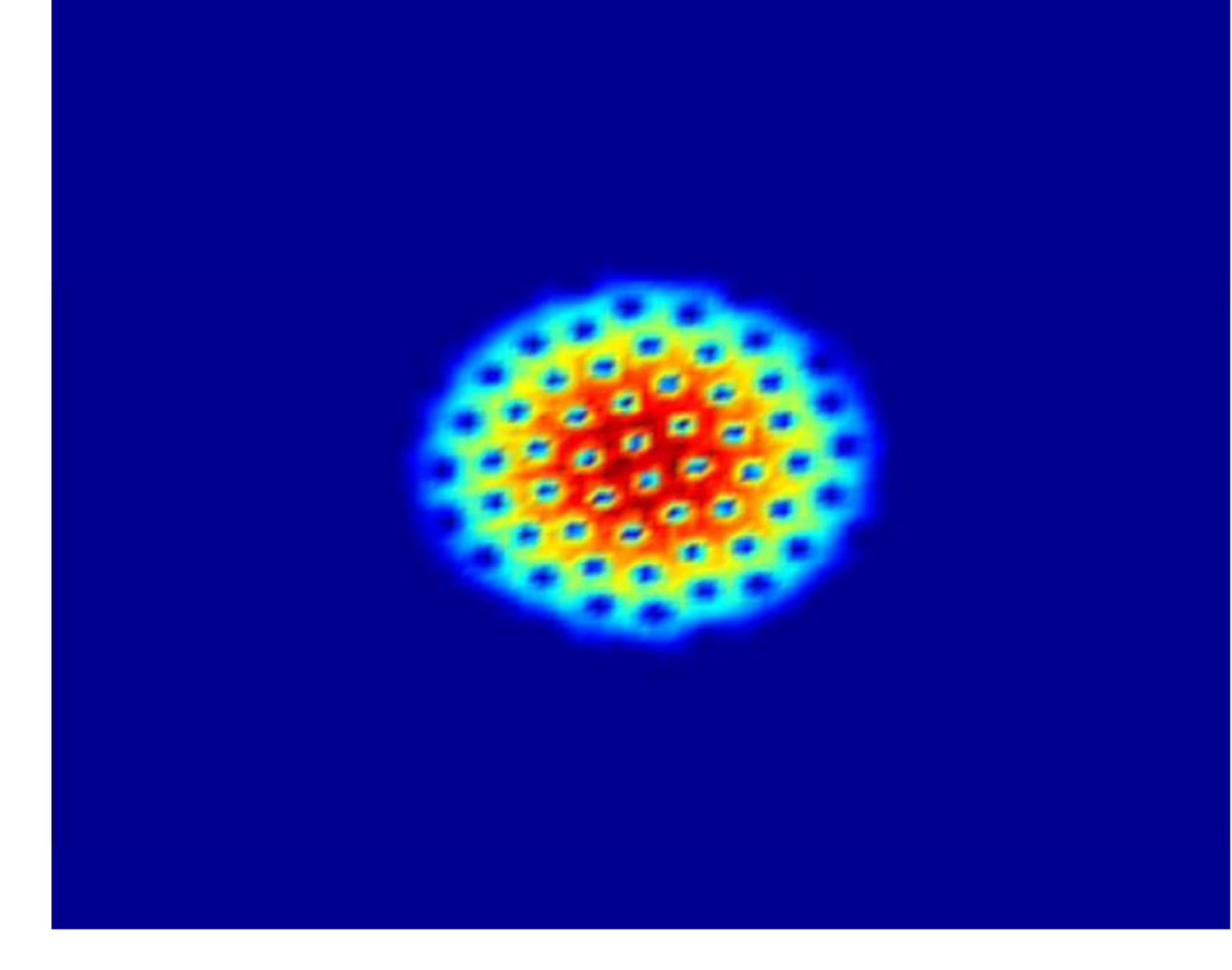}
\setlength{\abovecaptionskip}{-0.5mm}
	\caption{\footnotesize Contour plots of the density function $|\psi|^2$ for the dynamics of  vortex lattices in a 2D rotating BEC at different times $t=0,1,2.2,3.2,4.4,5.6,6.6,10$ with $\gamma_x=\gamma_y=1.4 $ (in order from left to right and from top to bottom).}
\label{fig5.-21}
 \end{figure}

\vspace{-5mm}

\begin{figure}[H]
\centering
    \includegraphics[height=3.5cm,width=0.23\linewidth]{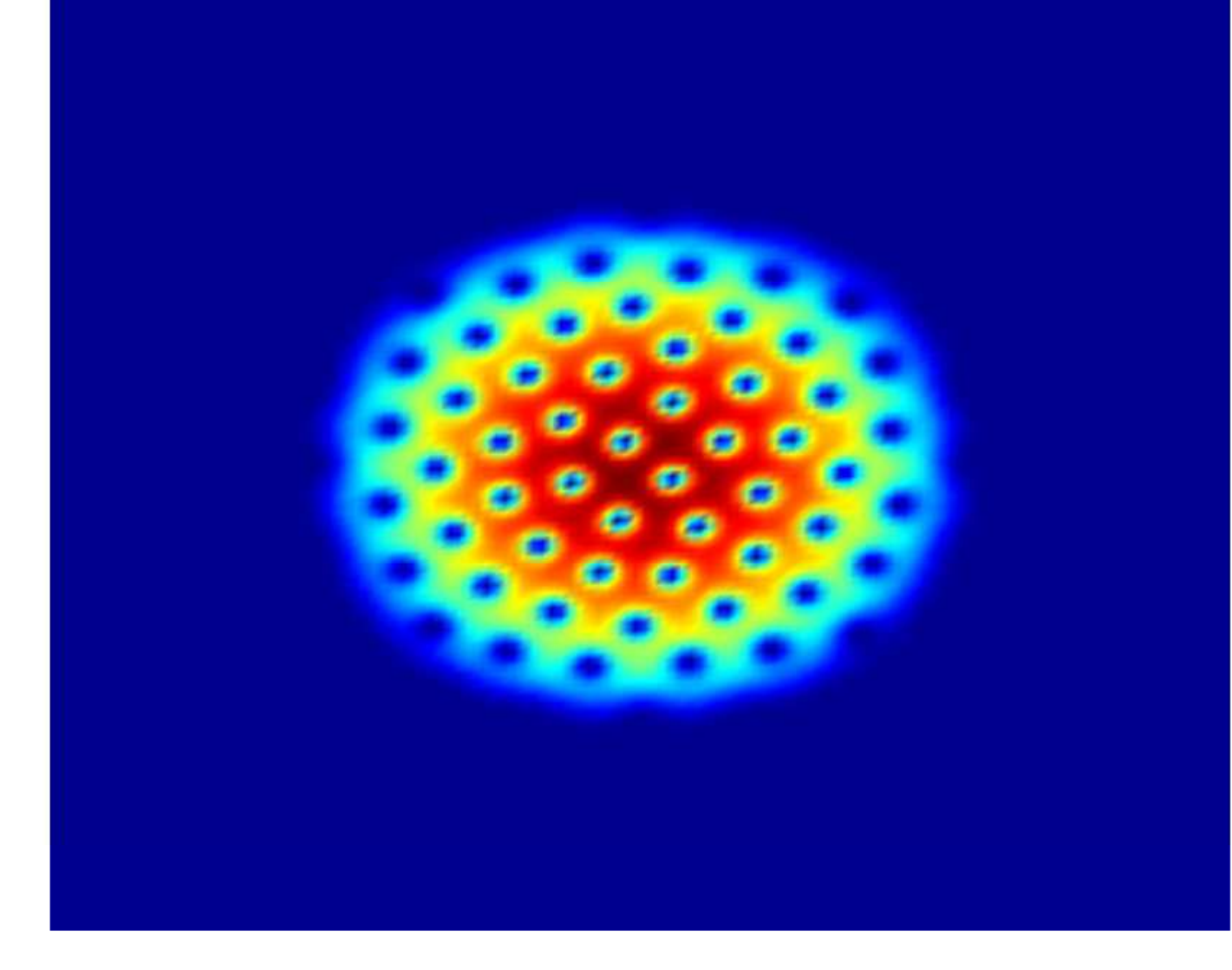}
    \includegraphics[height=3.5cm,width=0.23\linewidth]{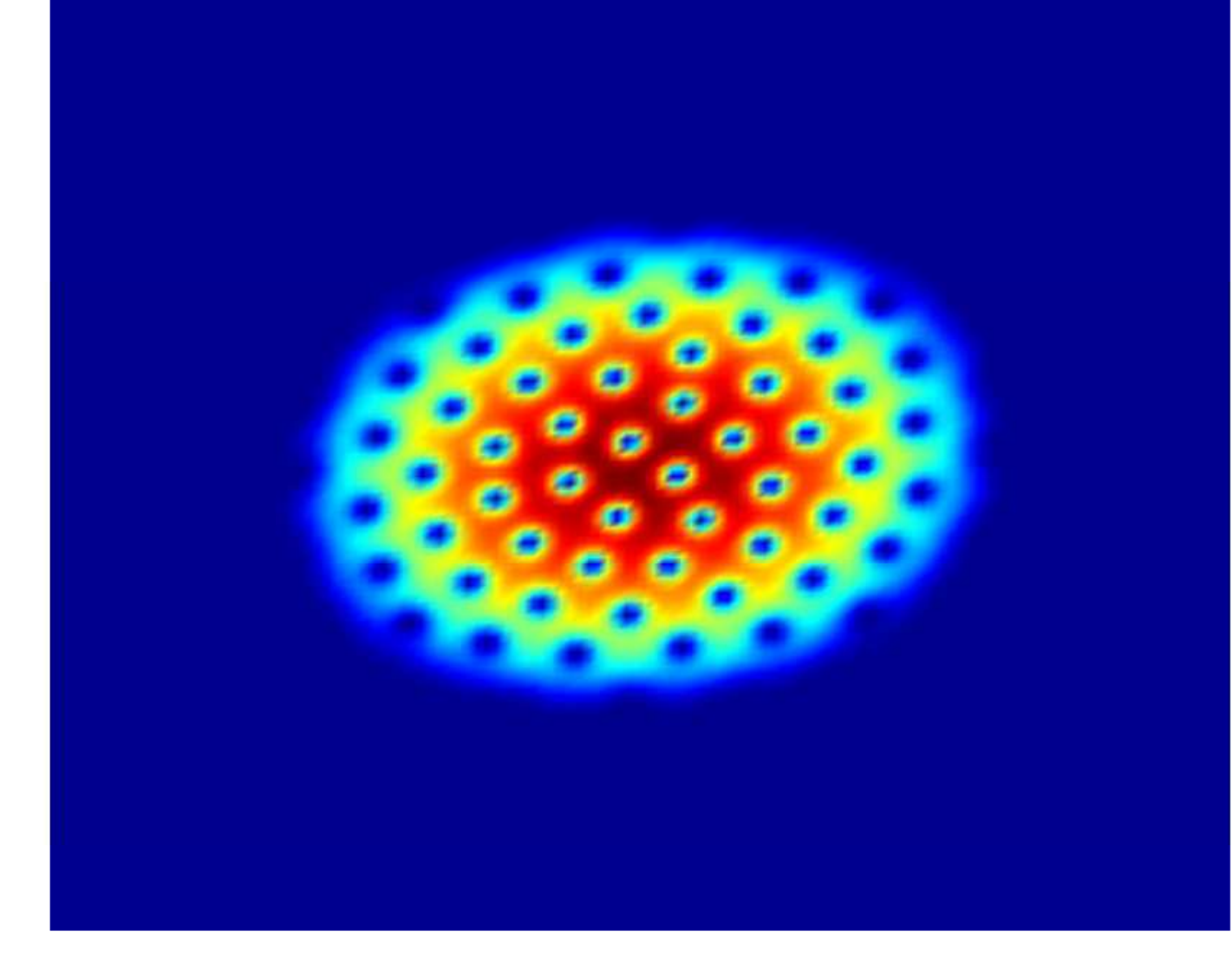}
    \includegraphics[height=3.5cm,width=0.23\linewidth]{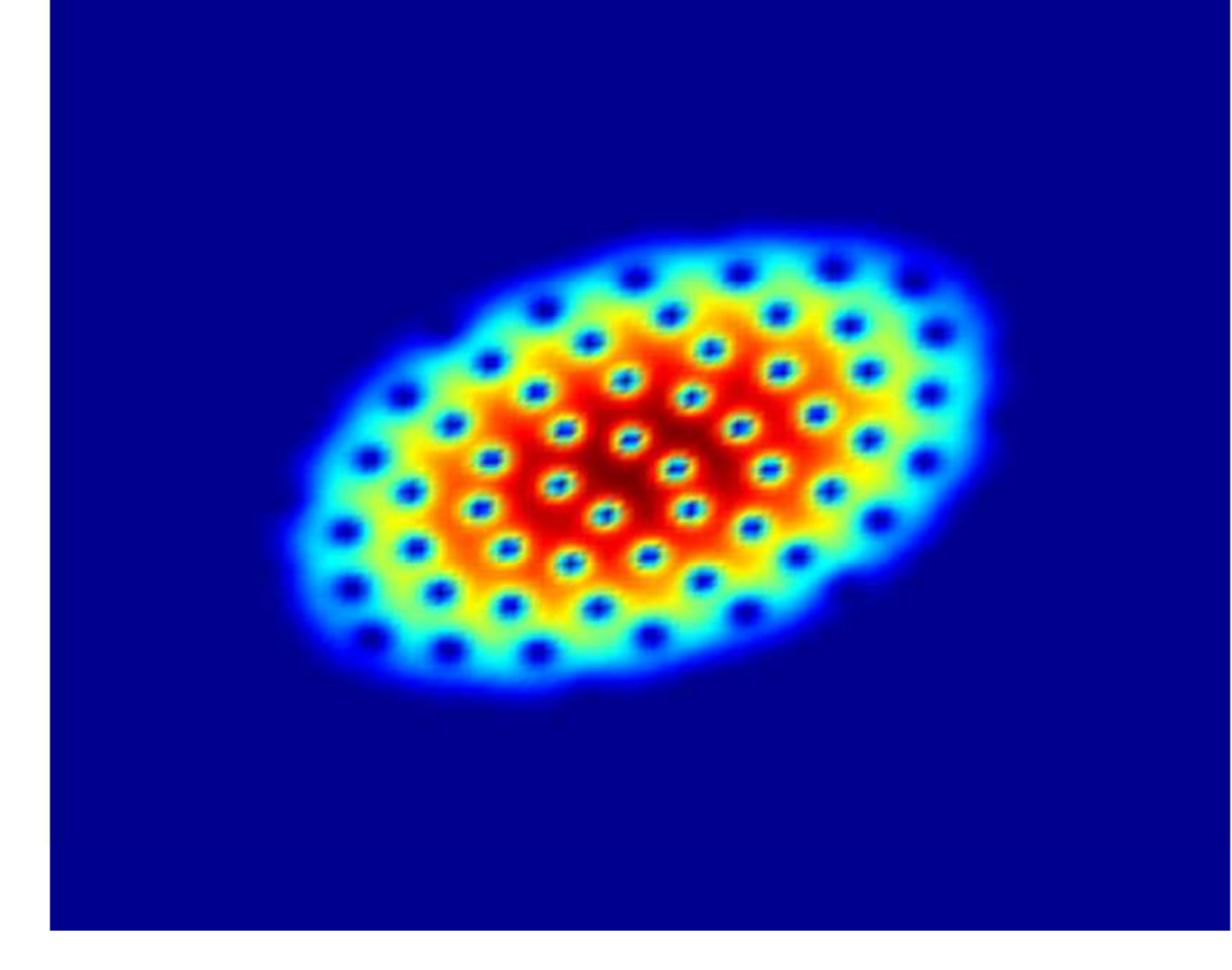}
	\includegraphics[height=3.5cm,width=0.23\linewidth]{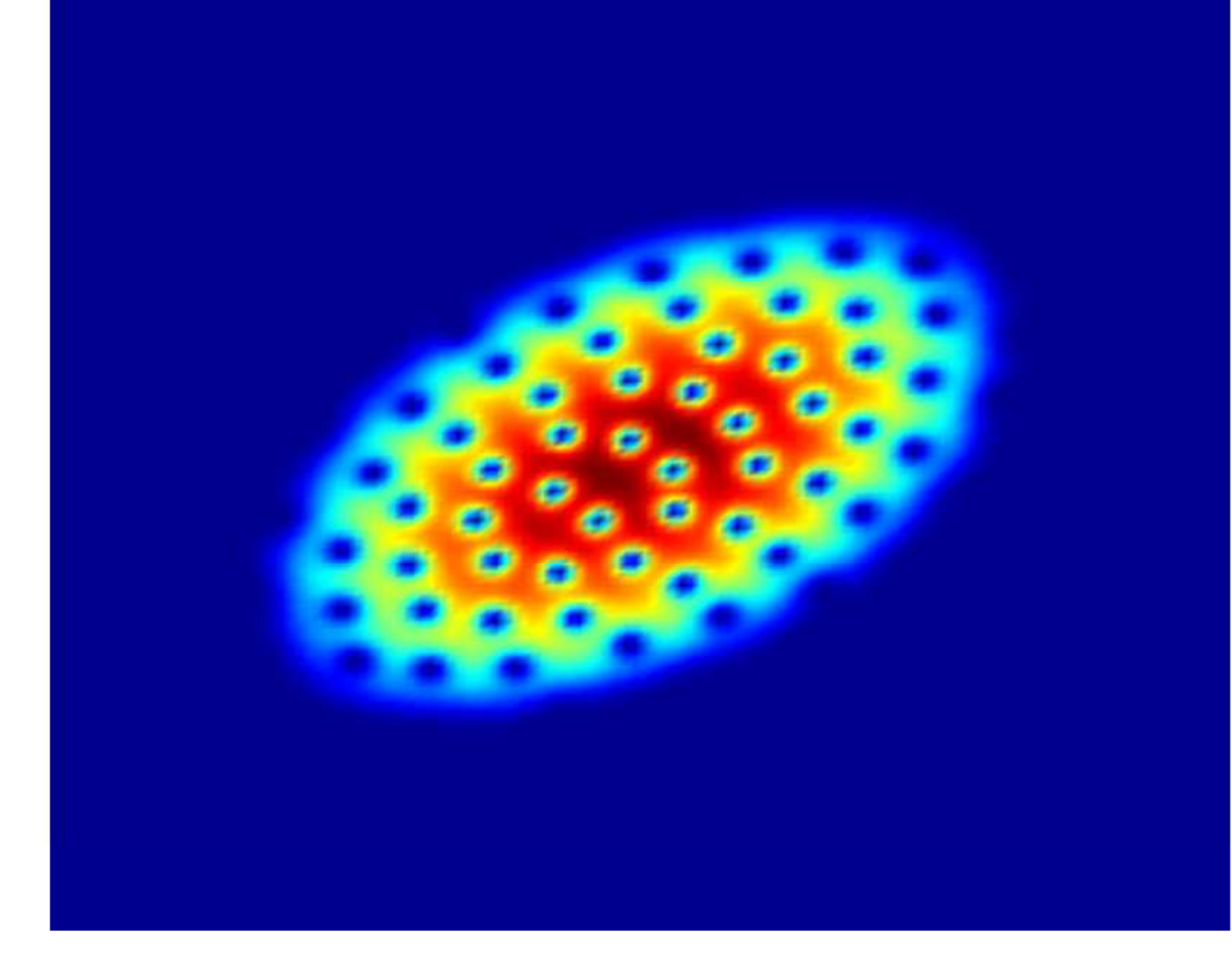}\\
	~\includegraphics[height=3.5cm,width=0.23\linewidth]{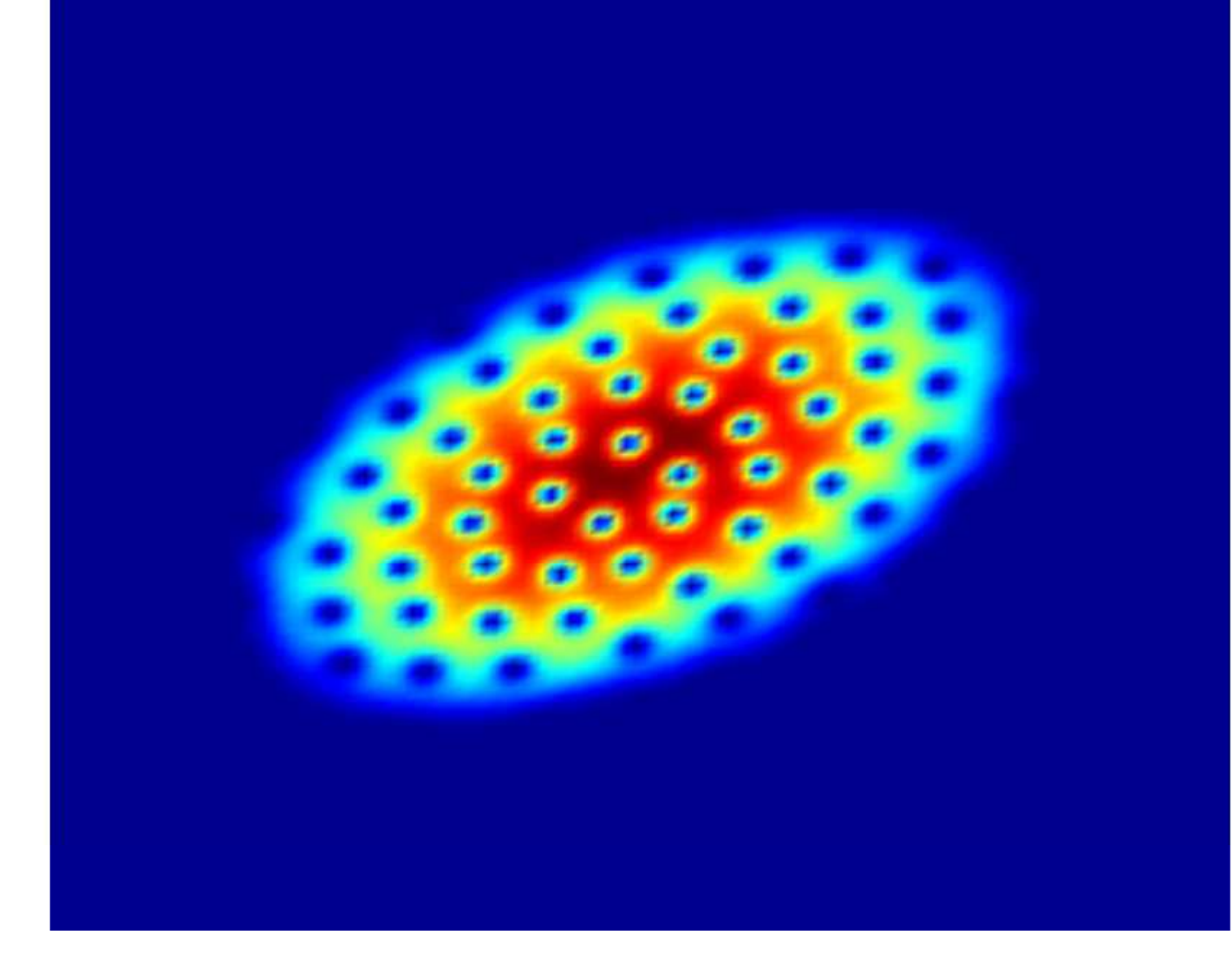}
    \includegraphics[height=3.5cm,width=0.23\linewidth]{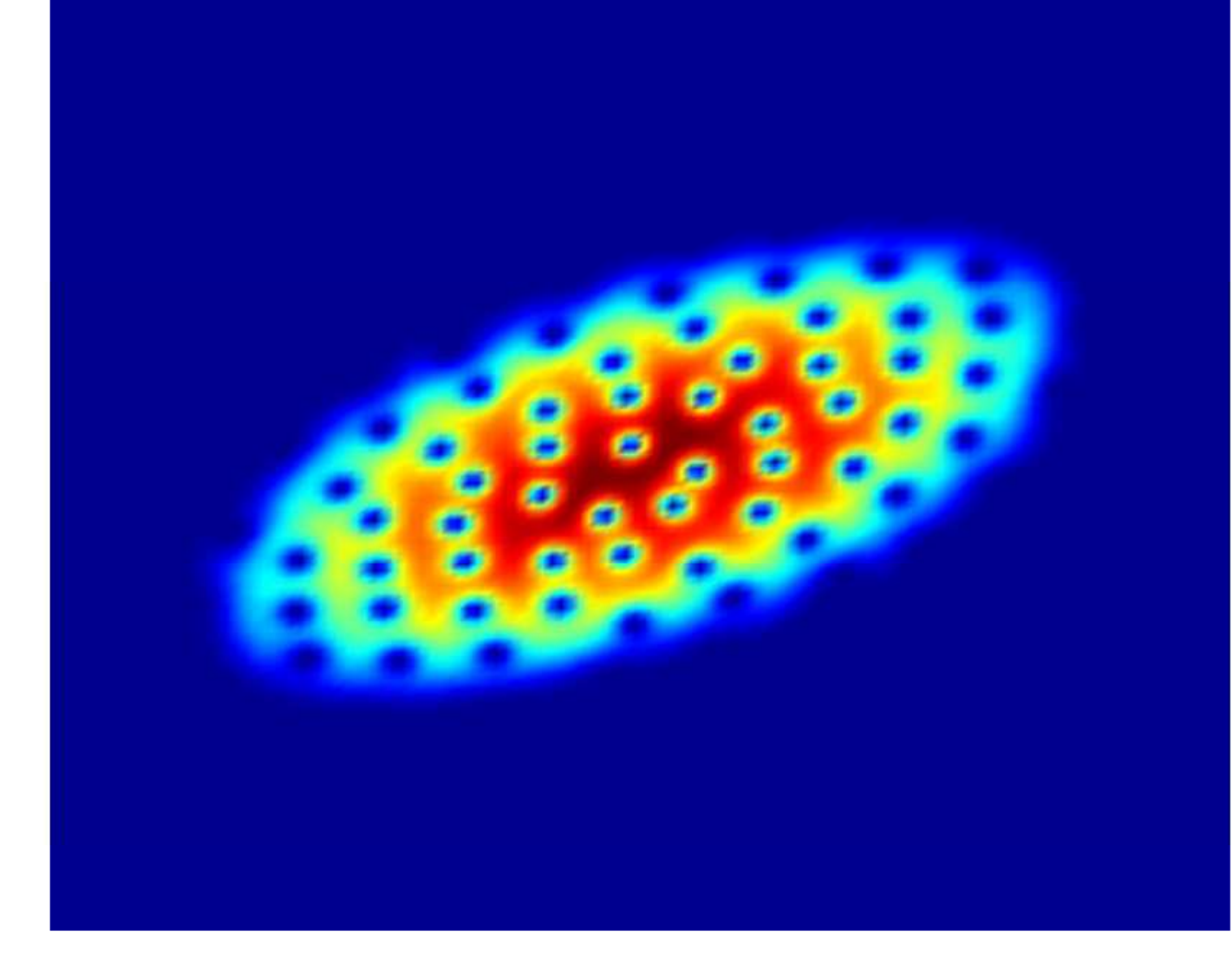}
	\includegraphics[height=3.5cm,width=0.23\linewidth]{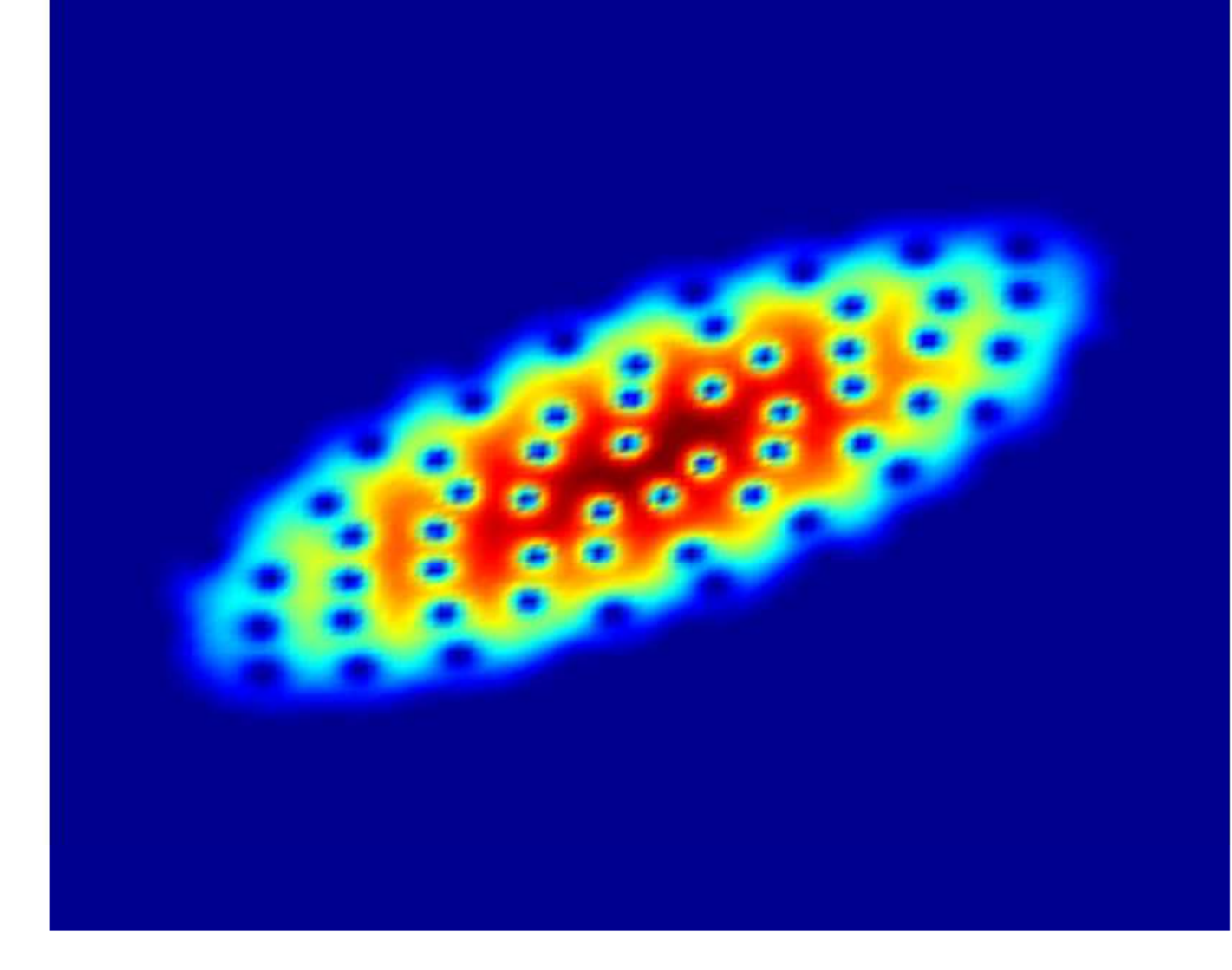}
    \includegraphics[height=3.5cm,width=0.23\linewidth]{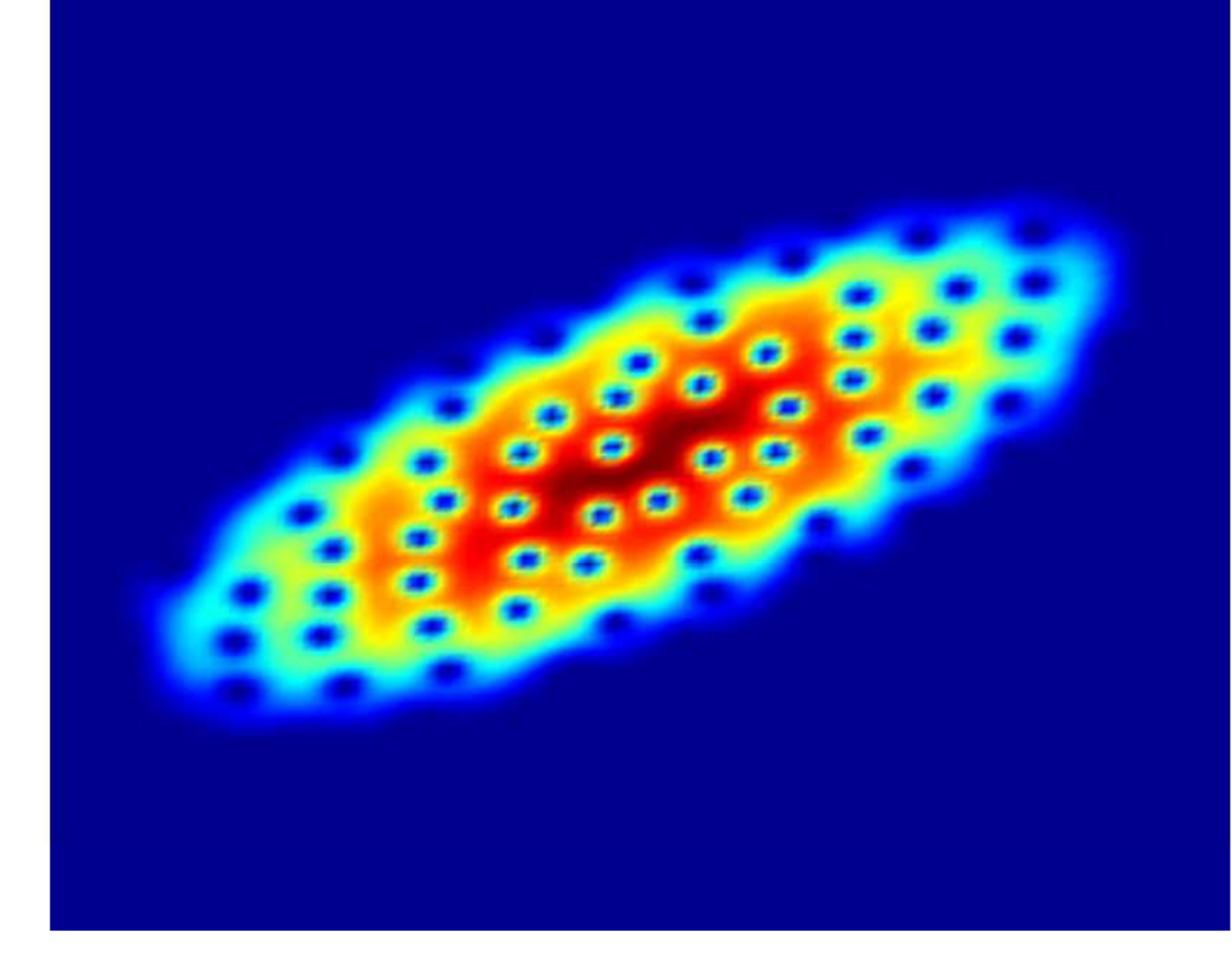}
\setlength{\abovecaptionskip}{-0.5mm}
	\caption{\footnotesize Contour plots of the density function $|\psi|^2$ for the dynamics of  vortex lattices in a 2D rotating BEC at different times $t=0,1,1.8,2.6,3.4,4.2,5,6$ with $\gamma_x=1.1$ and $\gamma_y=0.9 $ (in order from left to right and from top to bottom).}
	\label{fig5.-22}
\label{fig5.-2}
\end{figure}

\begin{figure}[H]
\begin{minipage}{0.48\linewidth}
  ~\centerline{\includegraphics[height=5.5cm,width=0.88\textwidth]{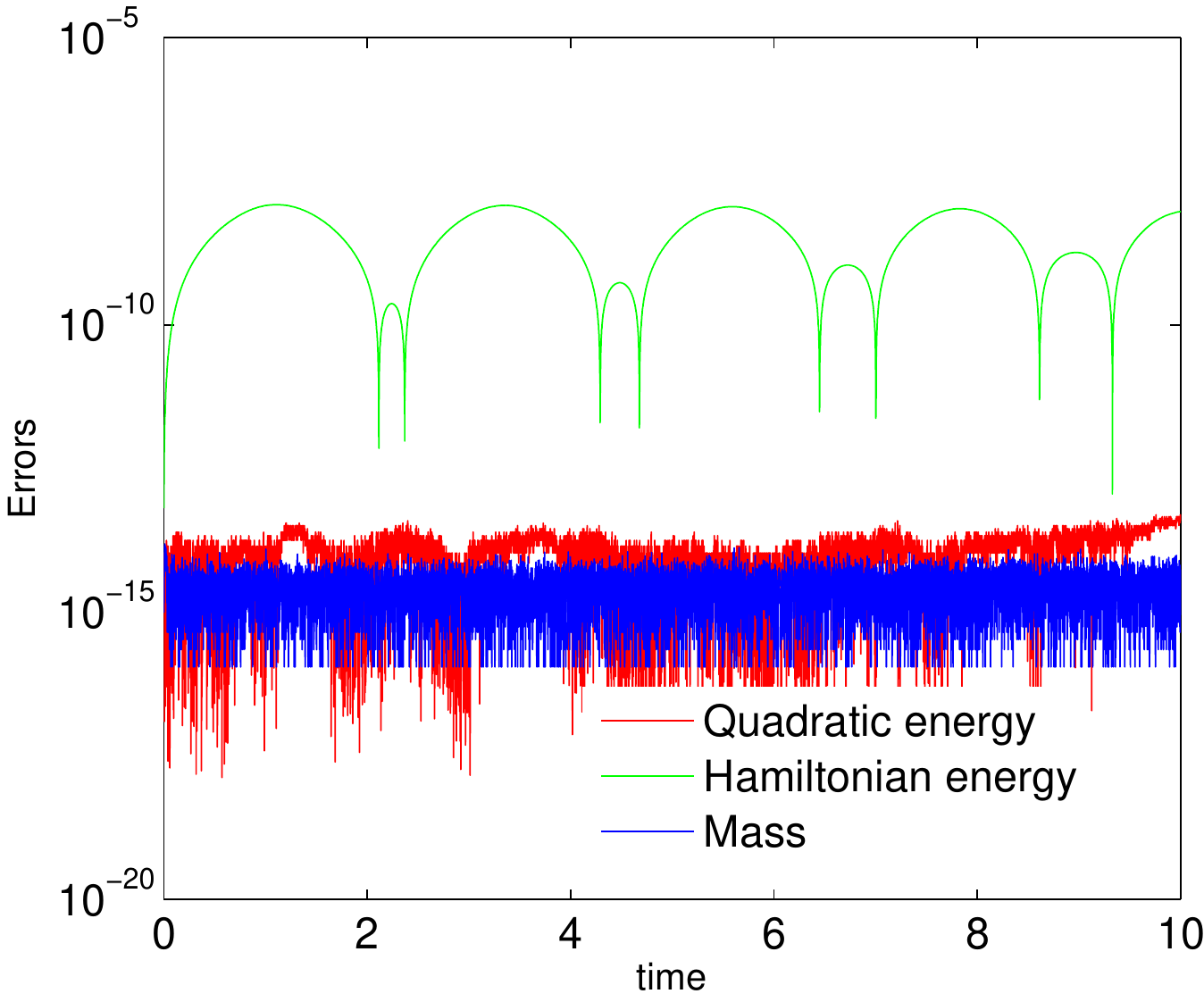}}
  \centerline{\footnotesize   (a) 4th-order HSAV for case $\rm \MyRoman{1}$}
\end{minipage}
\hfill
\begin{minipage}{0.48\linewidth}
  \centerline{\includegraphics[height=5.5cm,width=0.88\textwidth]{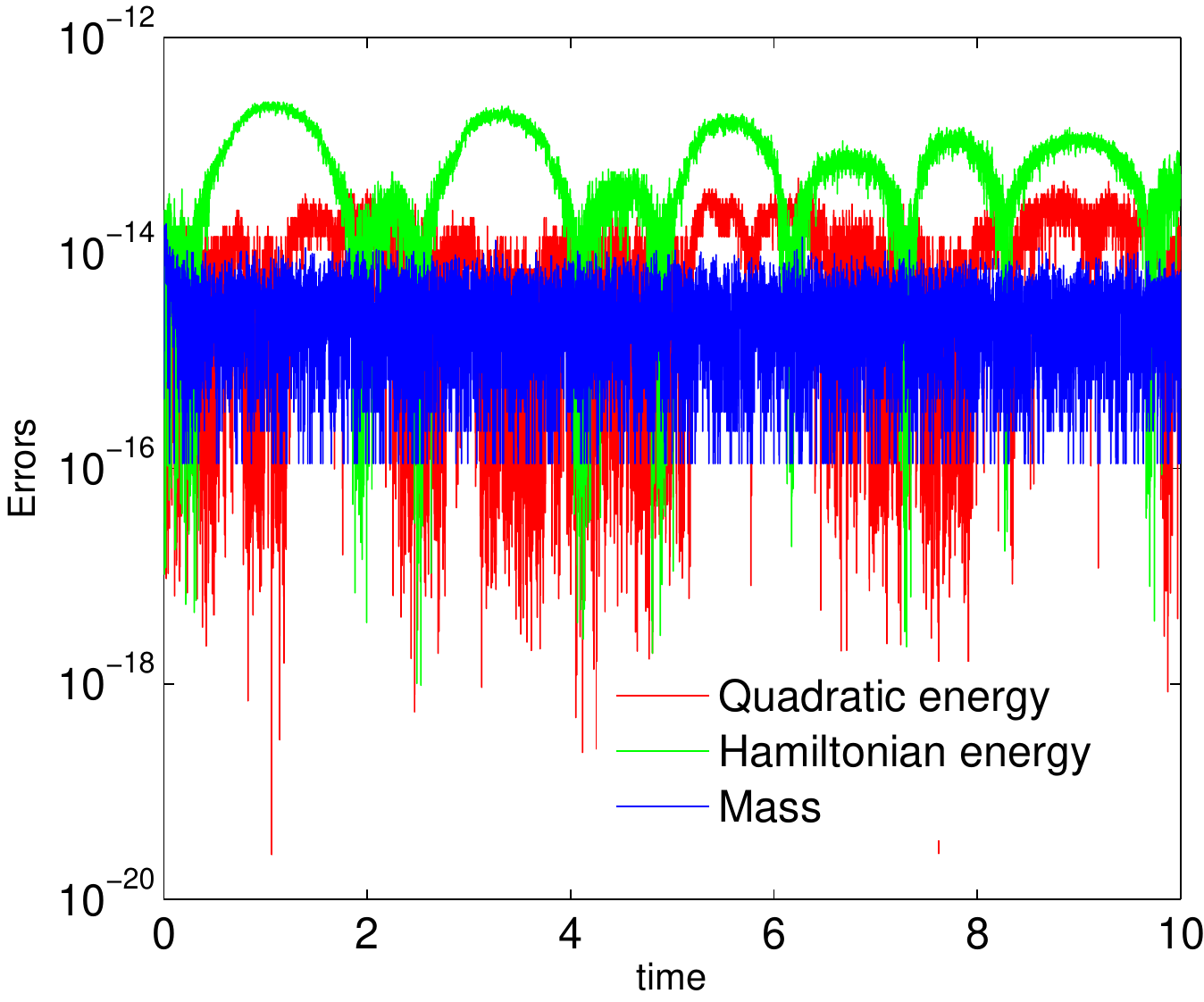}}
  \centerline{\footnotesize   (b) 6th-order HSAV for case $\rm \MyRoman{1}$}
\end{minipage}
\end{figure}

\begin{figure}[H]
\begin{minipage}{0.48\linewidth}
  \centerline{~~\includegraphics[height=5.5cm,width=0.88\textwidth]{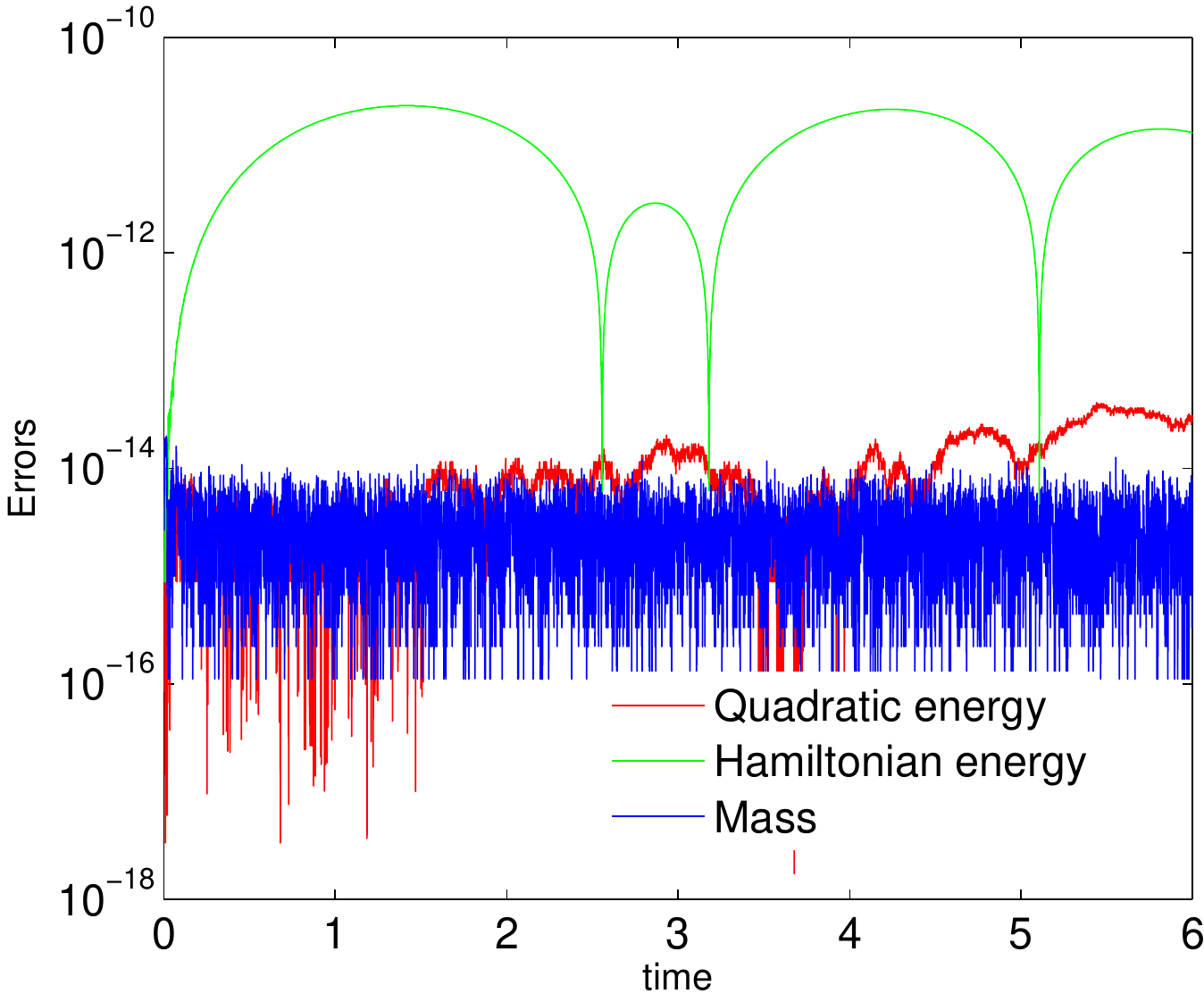}}
  \vspace{-0.4mm}
  \centerline{\footnotesize  (c) 4th-order HSAV for case $\rm \MyRoman{2}$}
\end{minipage}
\hfill
\begin{minipage}{0.48\linewidth}
  \centerline{\includegraphics[height=5.5cm,width=0.88\textwidth]{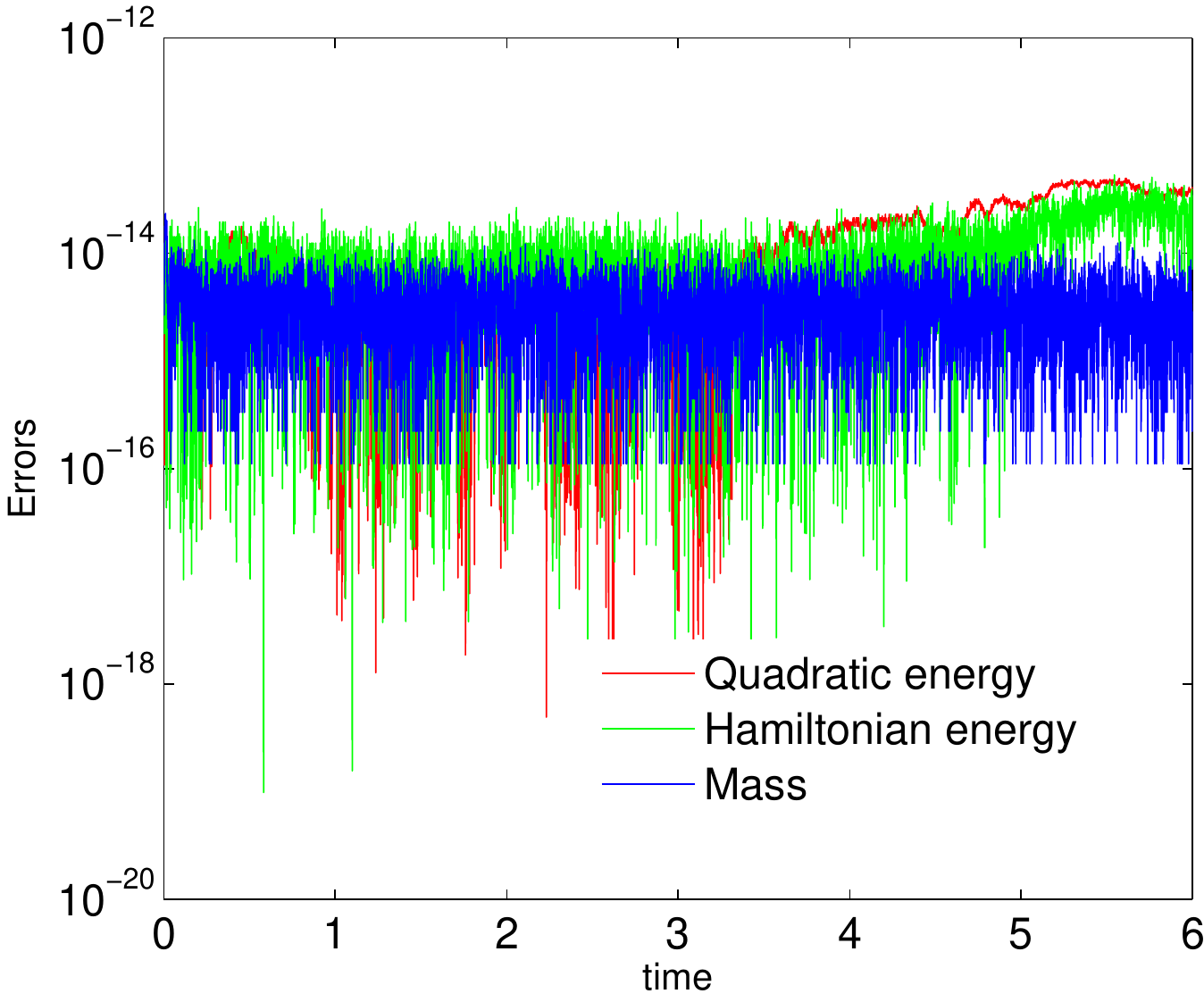}}
   \vspace{-0.4mm}
  \centerline{\footnotesize  (d) 6th-order HSAV for case $\rm \MyRoman{2}$}
\end{minipage}
\vspace{0.9mm}
  \setlength{\abovecaptionskip}{-0.00mm}
\caption{\small Evolution of discrete mass and energy with $\tau=0.001, \beta=1000$ and $\Omega=0.9$.}
\label{fig5.-23}
\end{figure}

\vspace{-0.35cm}

\noindent \textbf{Example 6.3} In this example, we further simulate dynamics of vortex lines in a 3D rotating BEC, and then choose the ground state computed by the BESP method~\cite{antoine14}. For the GP equation~\eqref{eq1.1}, we take $\beta=400, \Omega=0.8, V=x^2+y^2+\frac{z^2}{2}$, and the spatial domain $\mathcal {D}=[-10,10]^3$ with mesh size $h=20/64$.

From Figure~\ref{fig5.-31}, we can clearly observe the initial stationary vortex profiles from different angles, as well as the phase of the ground state in the $(x,y)$-plane. In the subsequent simulations, we only show the contour plots of the density
function $|\psi|^2$ for the dynamics of vortex lines computed by 4th-order HSAV method, and 6th-order counterpart performs similarly.
Figures~\ref{fig5.-32} and~\ref{fig5.-33} illustrate the contour plots from the angle of Figure~\ref{fig5.-31} (a) for the dynamics at different time levels. It is demonstrated that the proposed methods can resolve the 3D GP problem very well because of their high accuracy and efficiency, and the vortex structure is also conserved during the dynamics. It is interesting to find from Figure~\ref{fig5.-32} that the lattice shrinks or expands periodically in the vertical direction. In particular, eight vortex lines of the lattice shown in Figure~\ref{fig5.-33} are clearly observed to rotate counterclockwise around the $z$-axis from above.
Moreover, we study the long-time behaviour by carrying out a large time period $T=20$. As is shown in Figure~\ref{fig5.-34} that the discrete mass and quadratic energy are conserved precisely. What's more, the quadratic energy of the 6th-order HSAV method is preserved up to machine accuracy, as Figure~\ref{fig5.-34} (b) demonstrates.

\vspace{-0.5cm}

\begin{figure}[H]
\begin{minipage}{0.48\linewidth}
  \centerline{\includegraphics[height=5.4cm,width=0.53\textwidth]{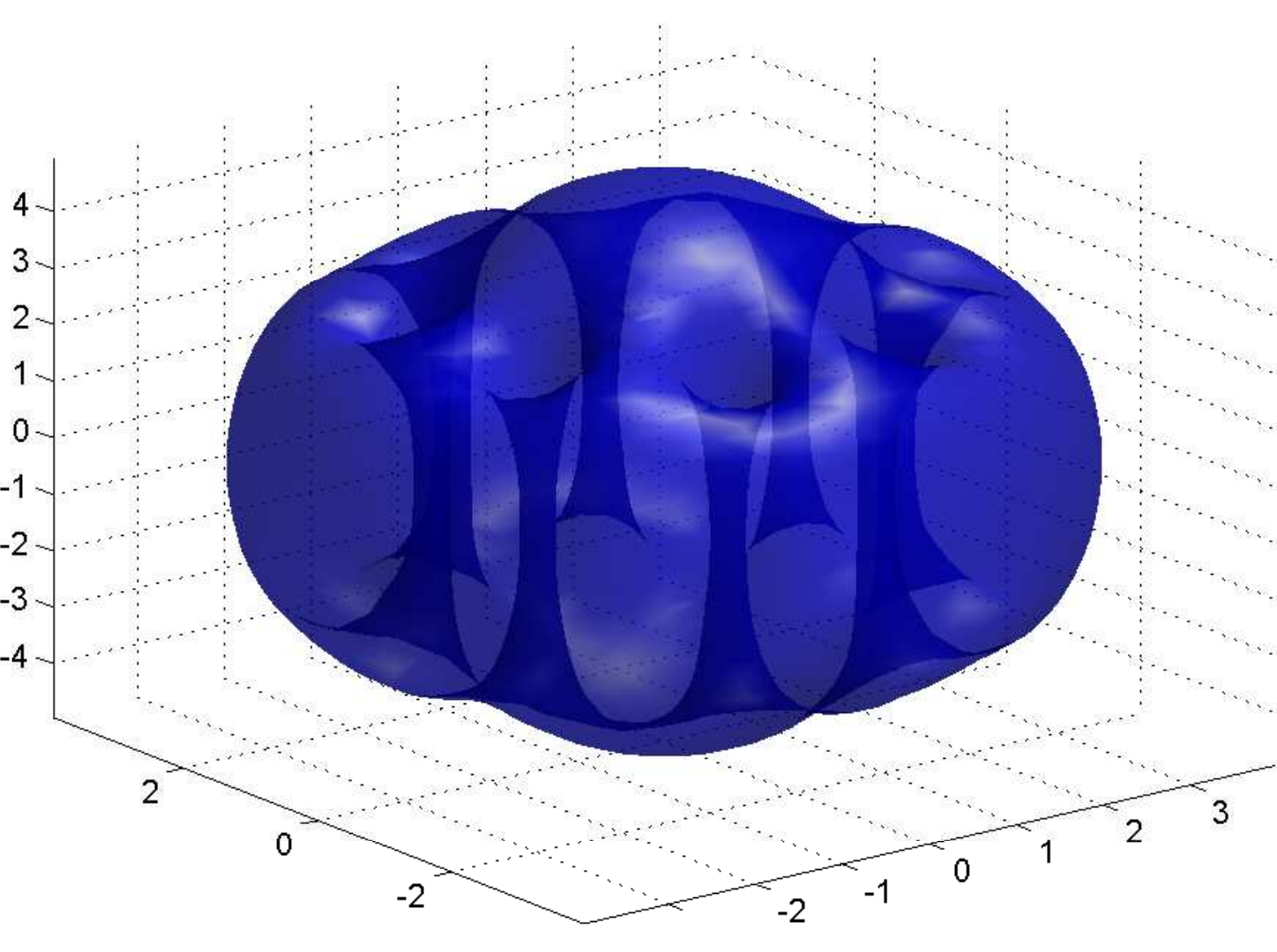}}
  \centerline{\small  (a) Side view image}
\end{minipage}
\hfill
\begin{minipage}{0.48\linewidth}
  \centerline{\includegraphics[height=5.4cm,width=0.53\textwidth]{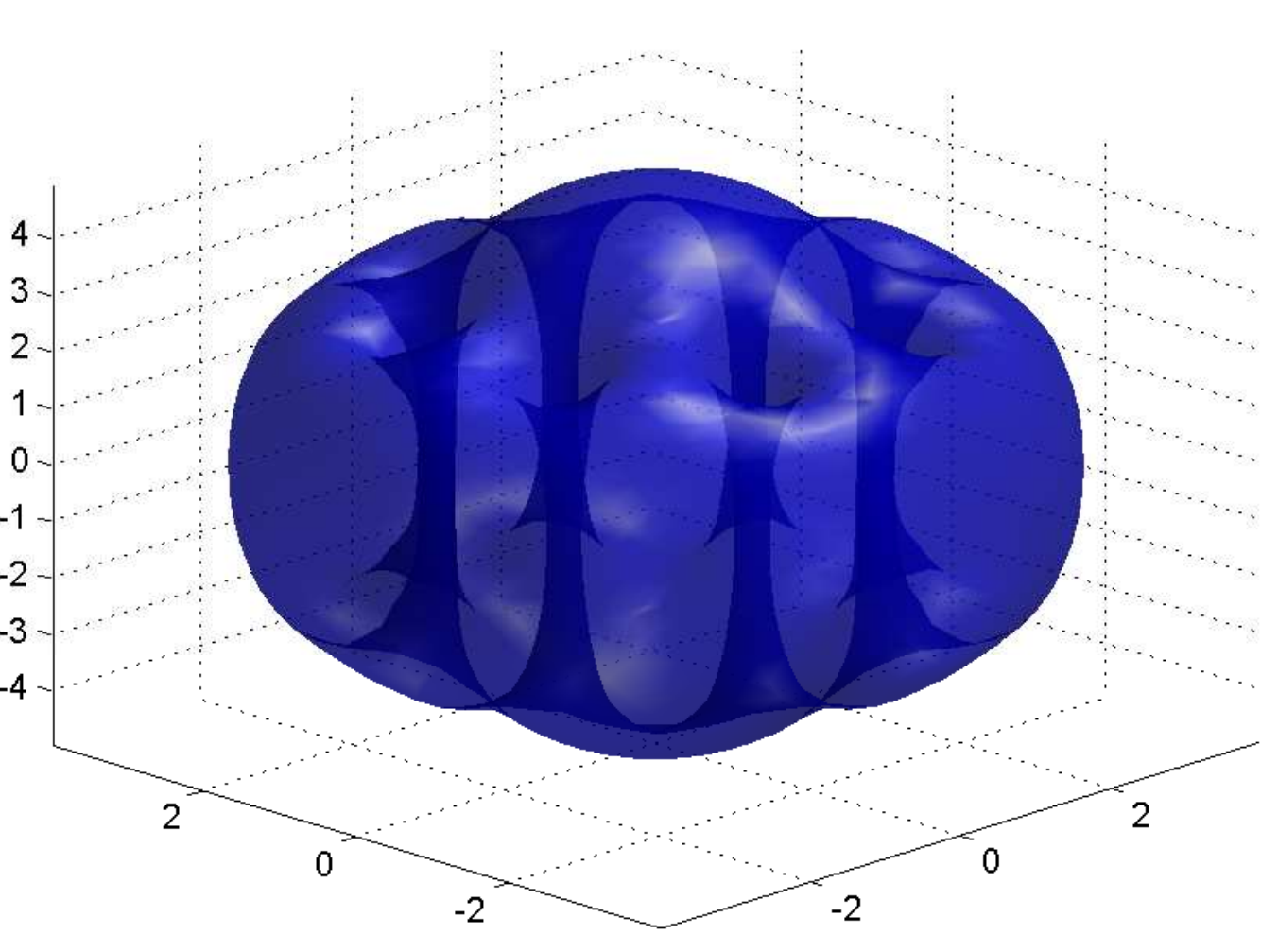}}
  \centerline{\small  (b) Side view image}
\end{minipage}
\vfill
\begin{minipage}{0.48\linewidth}
  \centerline{~~\includegraphics[height=4.0cm,width=0.69\textwidth]{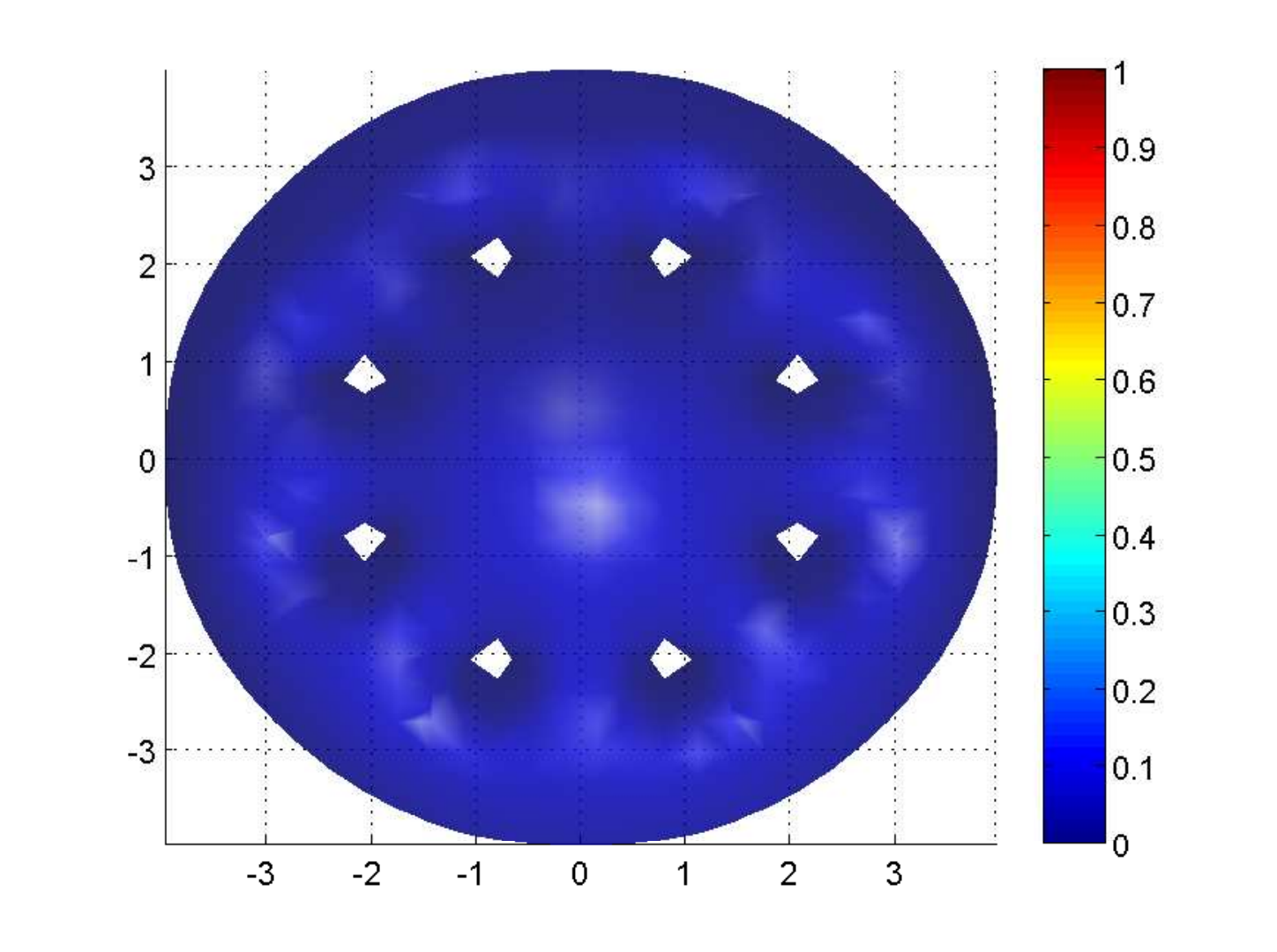}}
  \centerline{\small  (c) Vertical view image}
\end{minipage}
\hfill
\begin{minipage}{0.48\linewidth}
  \centerline{\includegraphics[height=4.45cm,width=0.76\textwidth]{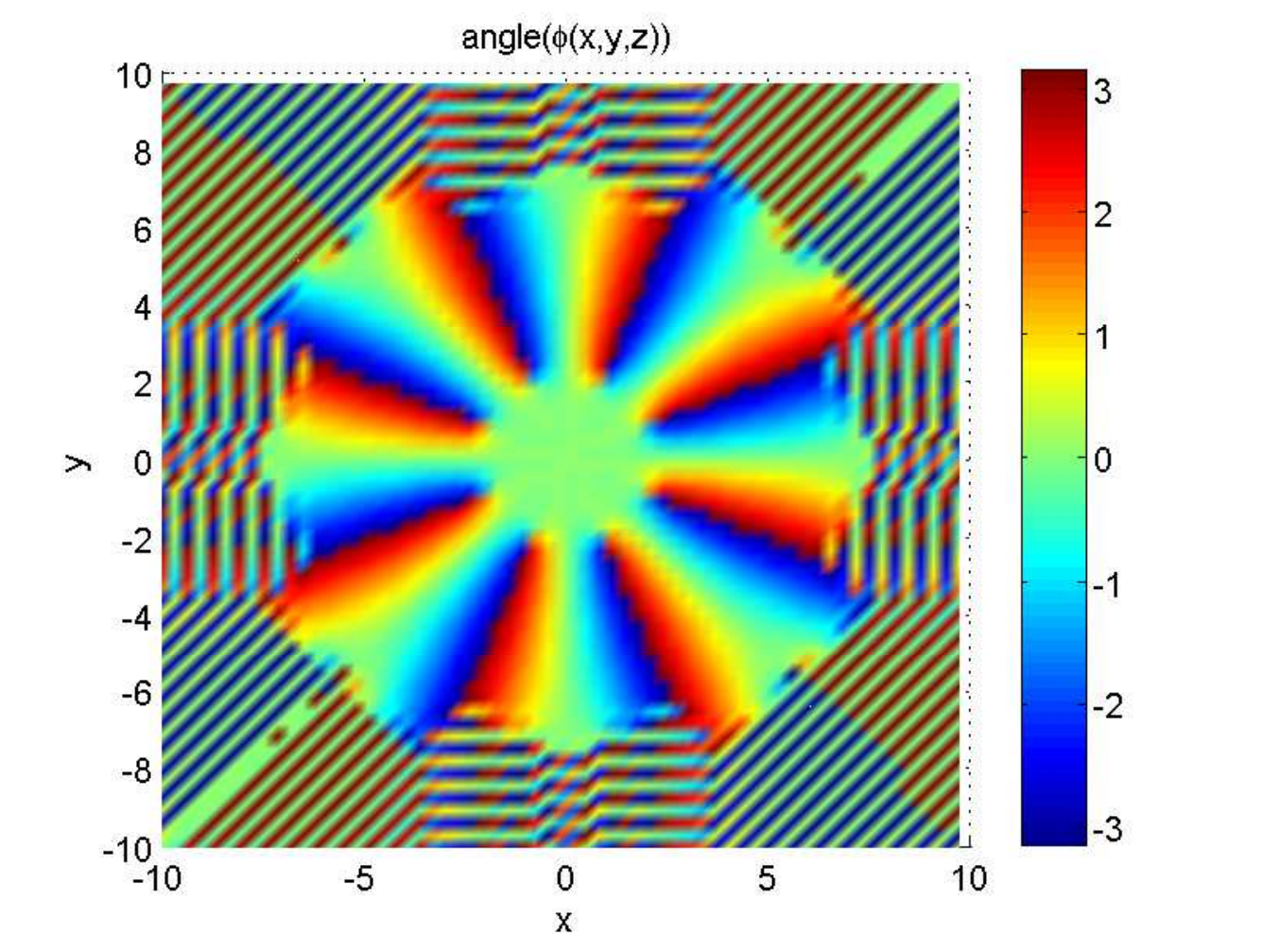}}
  \centerline{\small  (d) Slice in the $(x,y)$-plane of the phase}
\end{minipage}
\vspace{0.9mm}
  \setlength{\abovecaptionskip}{-0.1mm}
\caption{\small Condensate ground state in a 3D rotating BEC with $\beta=400$, and $\Omega=0.8$.}
\label{fig5.-31}
\end{figure}

\begin{figure}[H]
\centering
   ~~ ~~\includegraphics[height=5.6cm,width=0.26\linewidth]{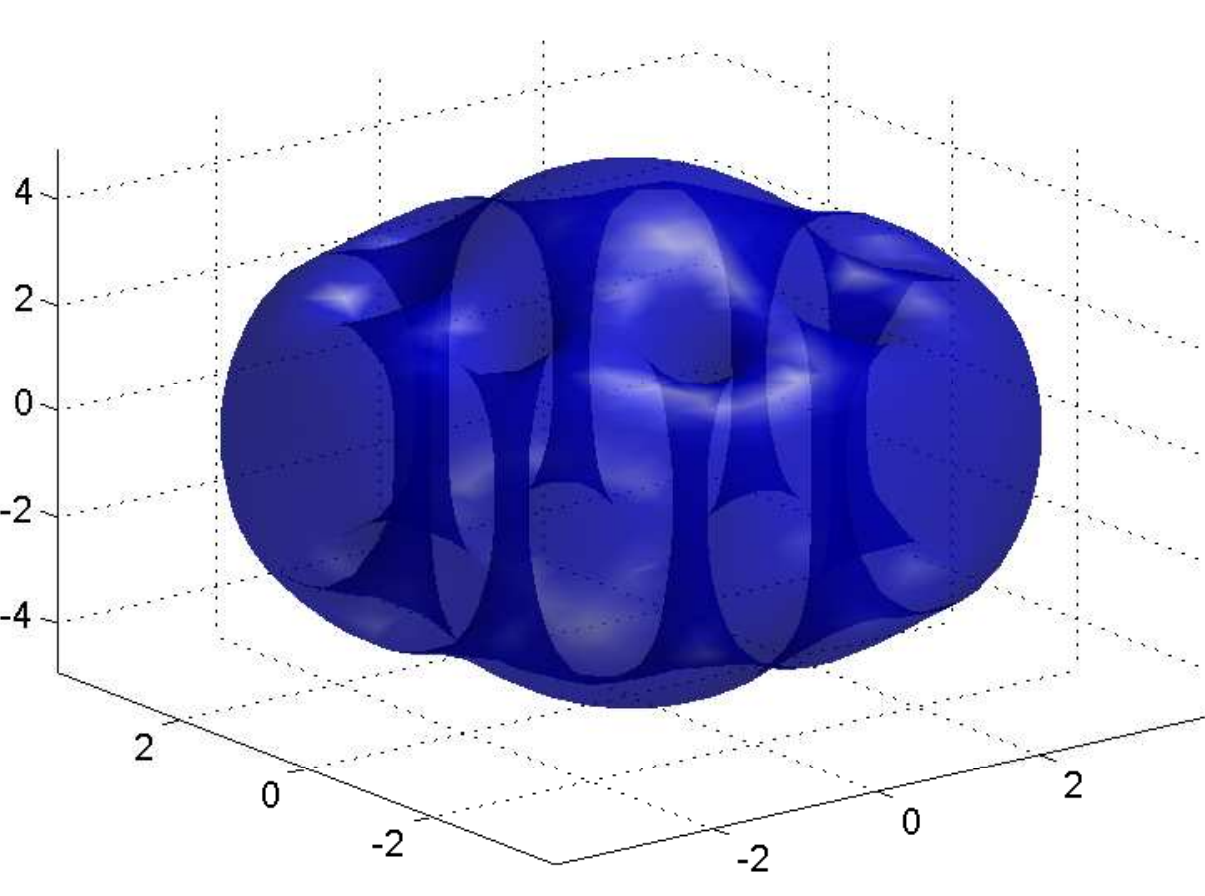}    
   ~ \includegraphics[height=5.5cm,width=0.30\linewidth]{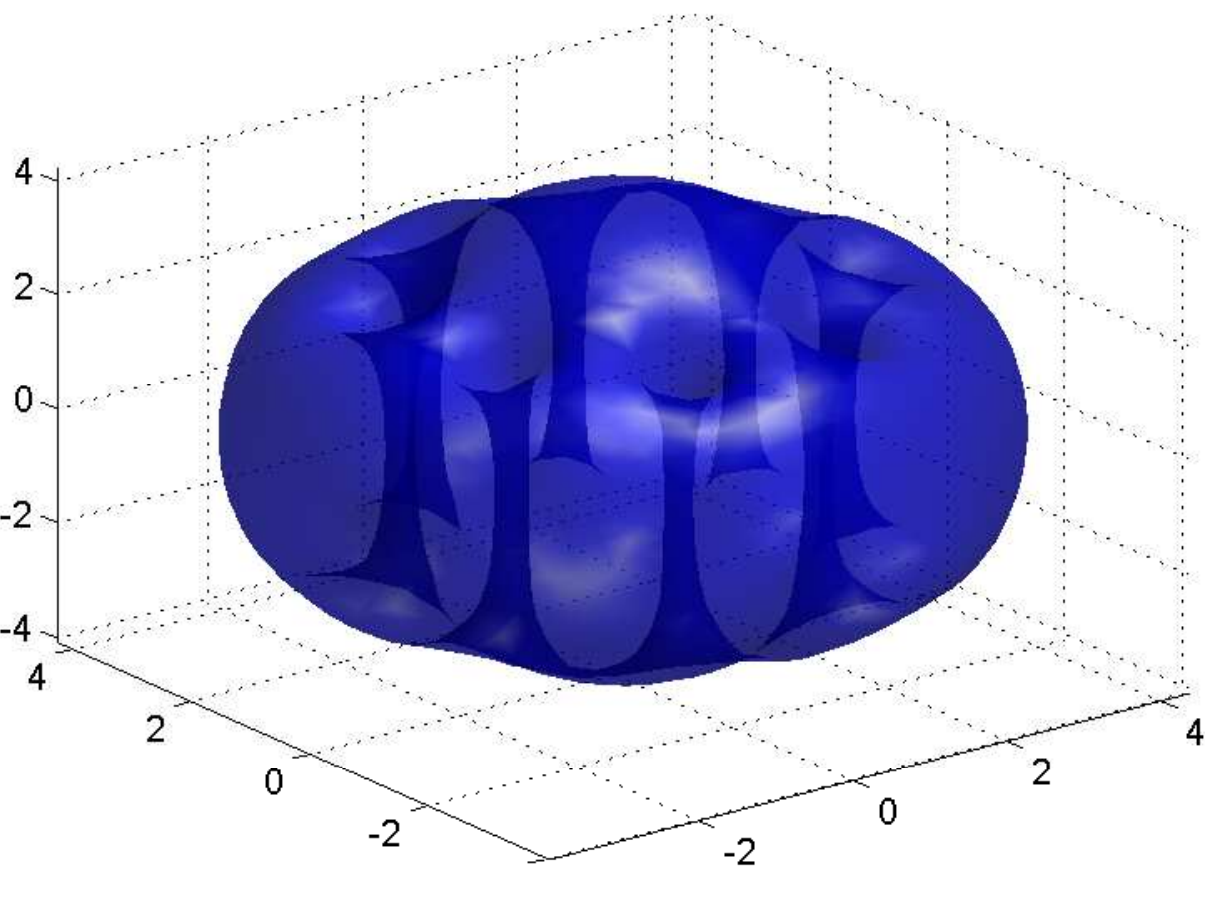}    
	\includegraphics[height=4.9cm,width=0.35\linewidth]{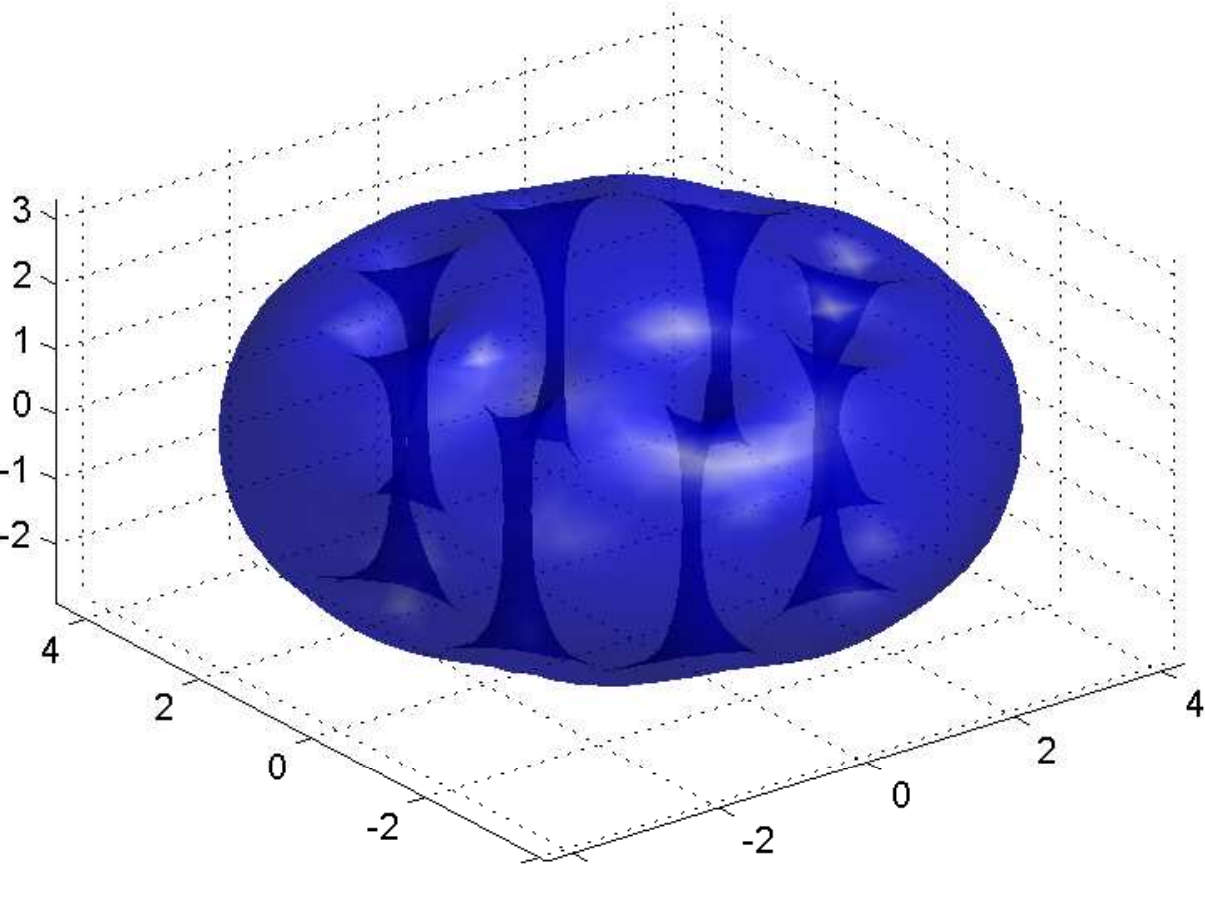}\\   
	\includegraphics[height=4.9cm,width=0.35\linewidth]{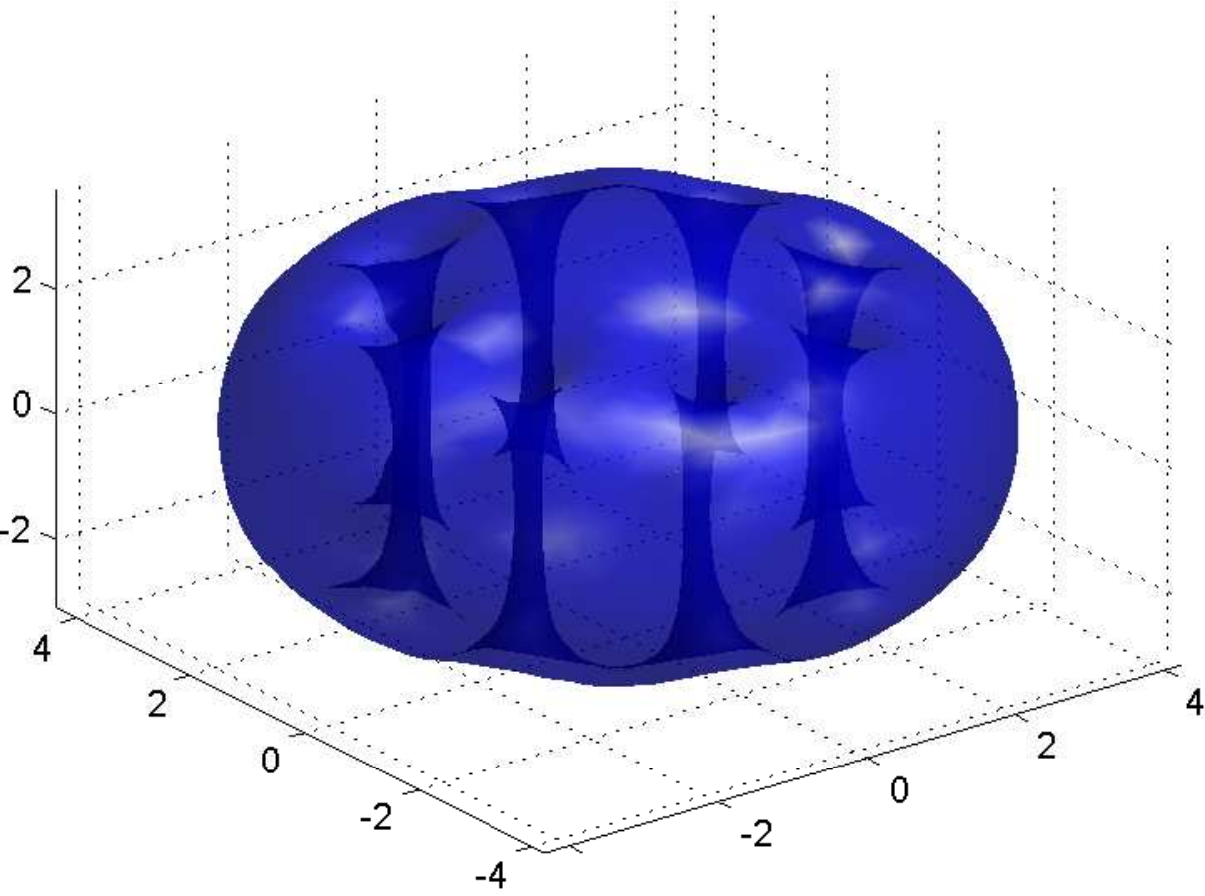}   
    \includegraphics[height=5.5cm,width=0.30\linewidth]{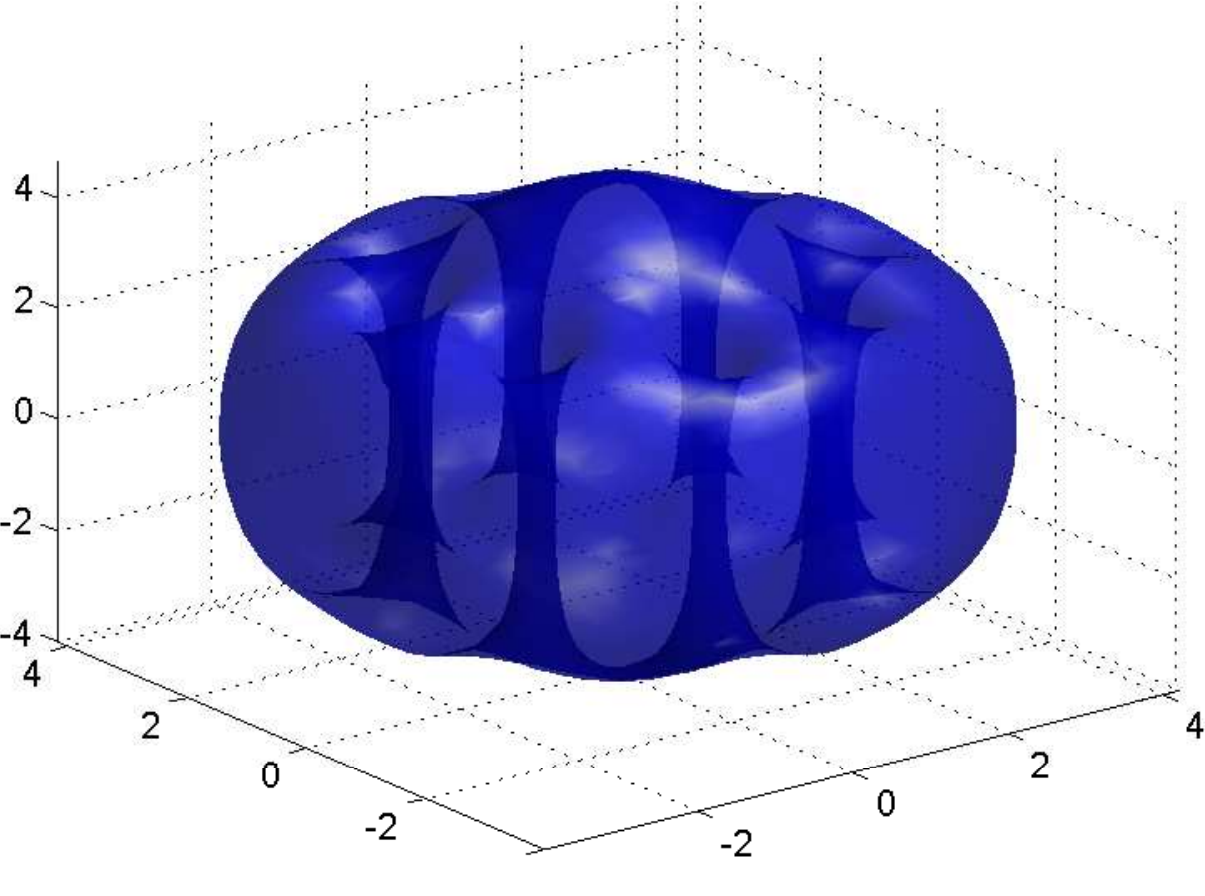}    
	~\includegraphics[height=5.6cm,width=0.27\linewidth]{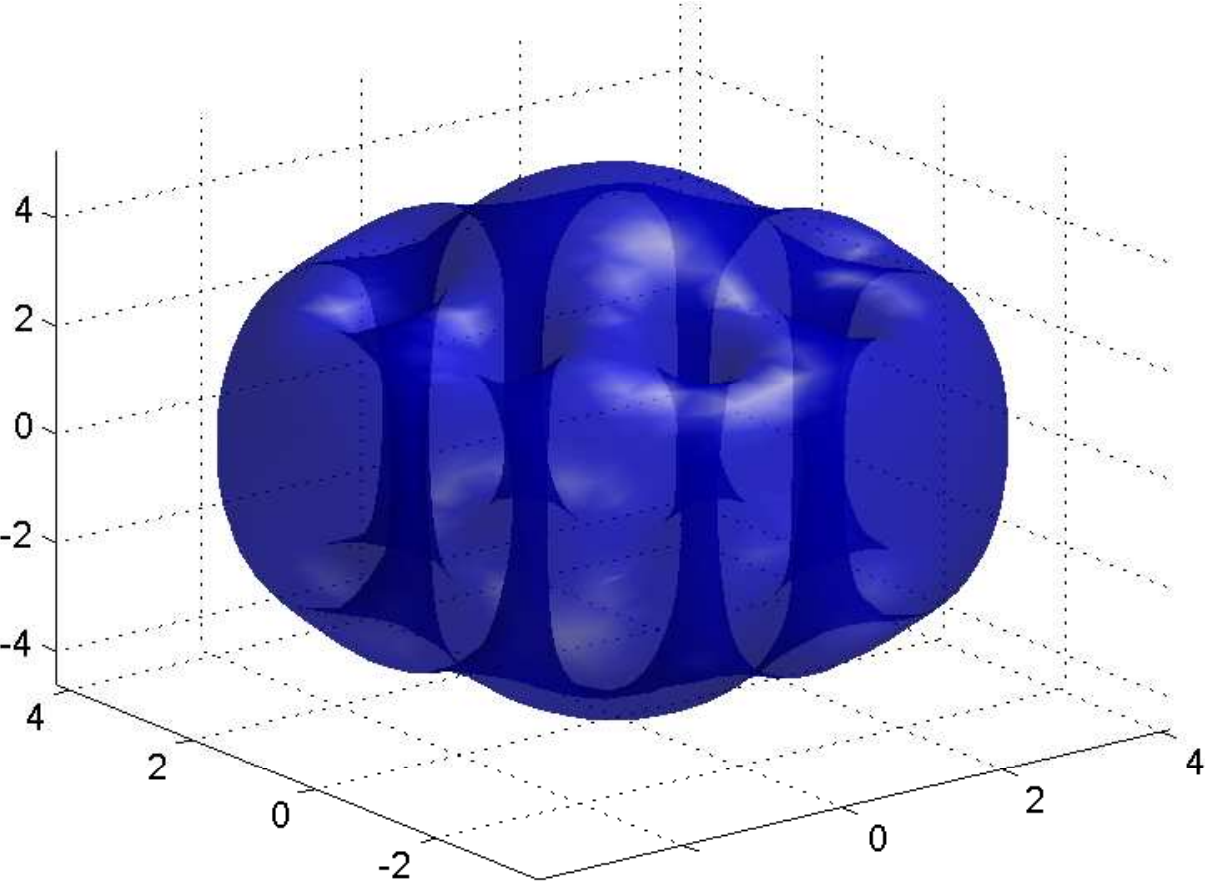}   
\setlength{\abovecaptionskip}{-0.1mm}
	\caption{\footnotesize Contour plots of the density function $|\psi|^2$ for the dynamics of vortex lines in a 3D rotating BEC with $\beta=400, \Omega=0.8$ at different times $t=1.0,1.3,2.7,3.3,4.2,5.2$ (in order from left to right and from top to bottom).}
	\label{fig5.-32}
\end{figure}

\vspace{-0.3cm}

\begin{figure}[H]
\centering
    \includegraphics[height=5.6cm,width=0.32\linewidth]{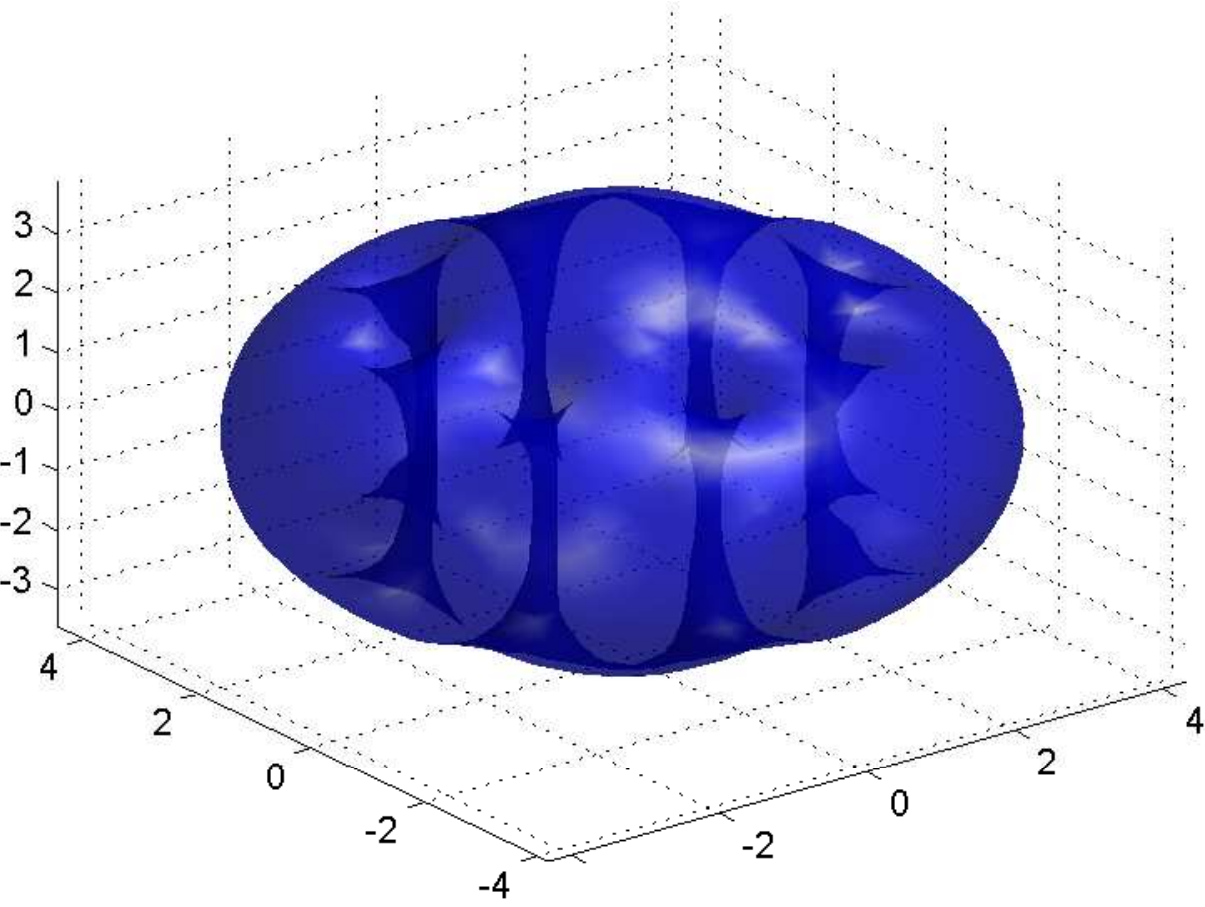}
    \includegraphics[height=5.5cm,width=0.32\linewidth]{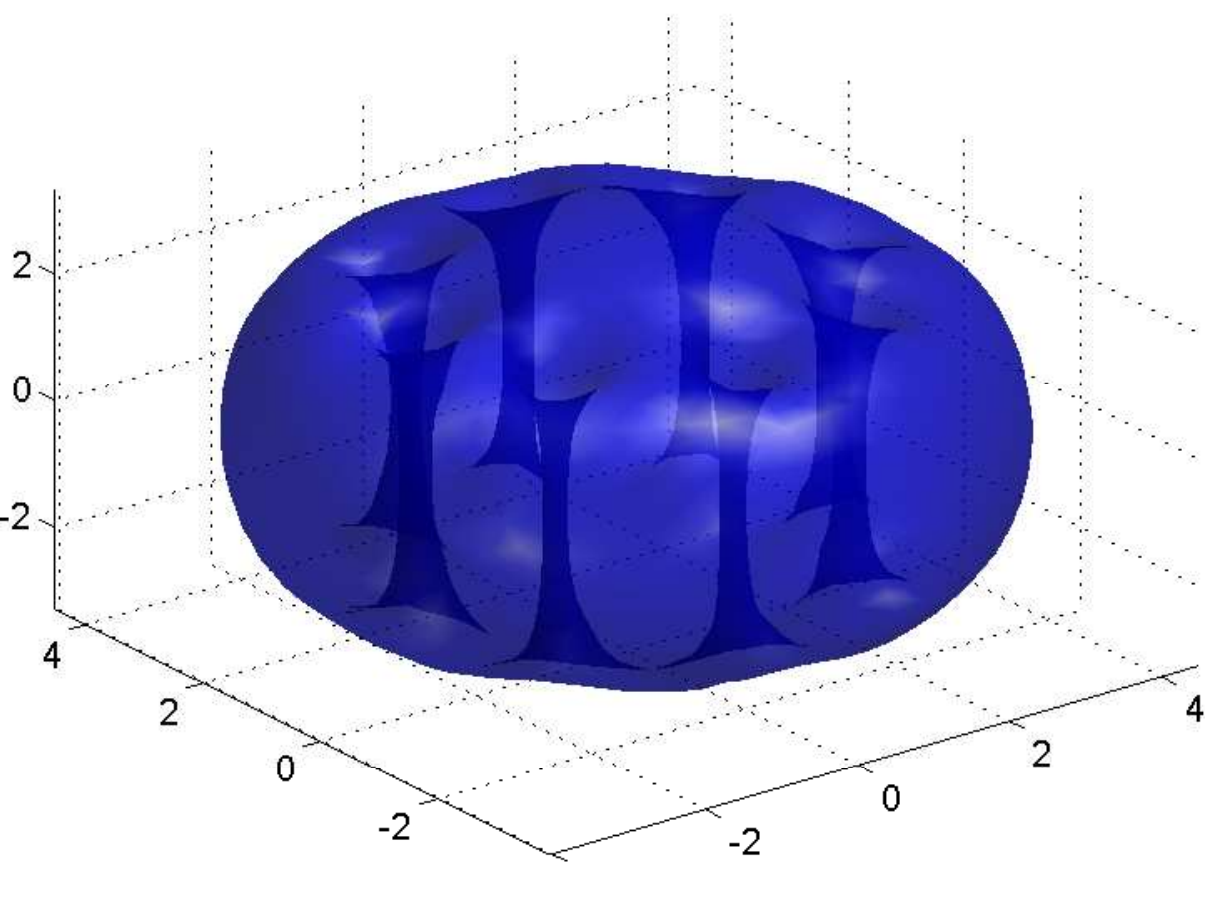}
	\includegraphics[height=5.7cm,width=0.3\linewidth]{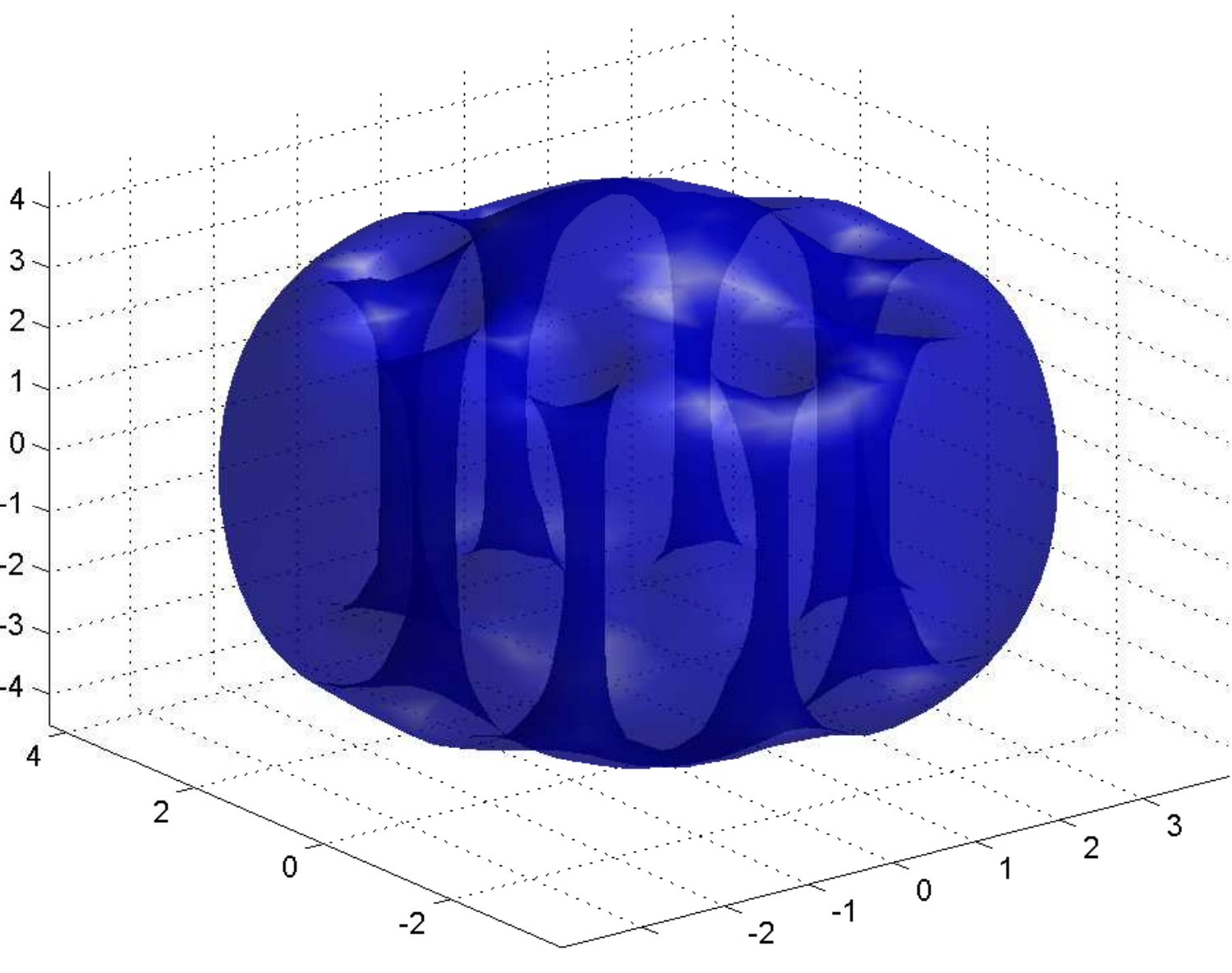} \\
	~~\includegraphics[height=5.2cm,width=0.32\linewidth]{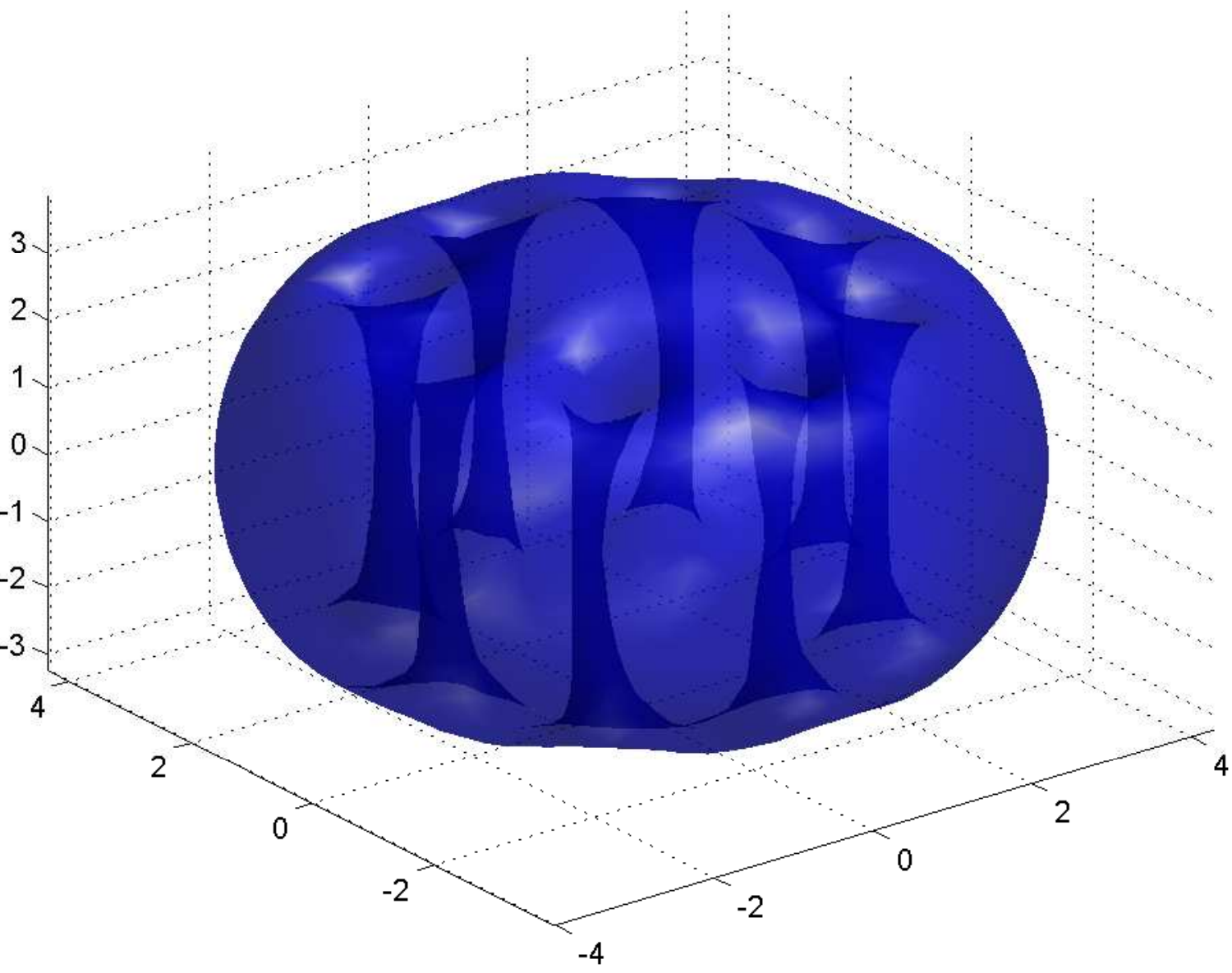}
    ~~\includegraphics[height=5.7cm,width=0.30\linewidth]{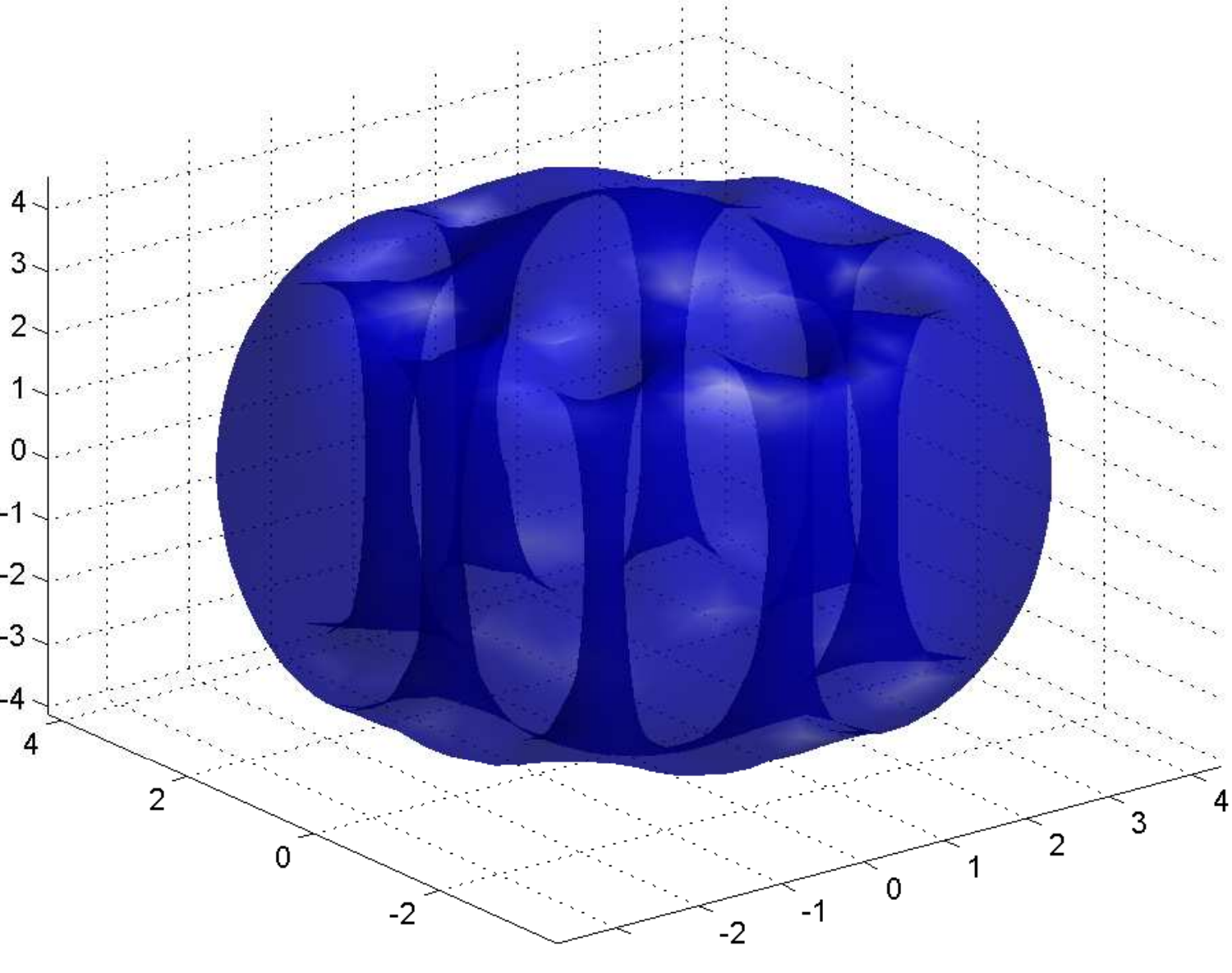}
	~\includegraphics[height=5.7cm,width=0.30\linewidth]{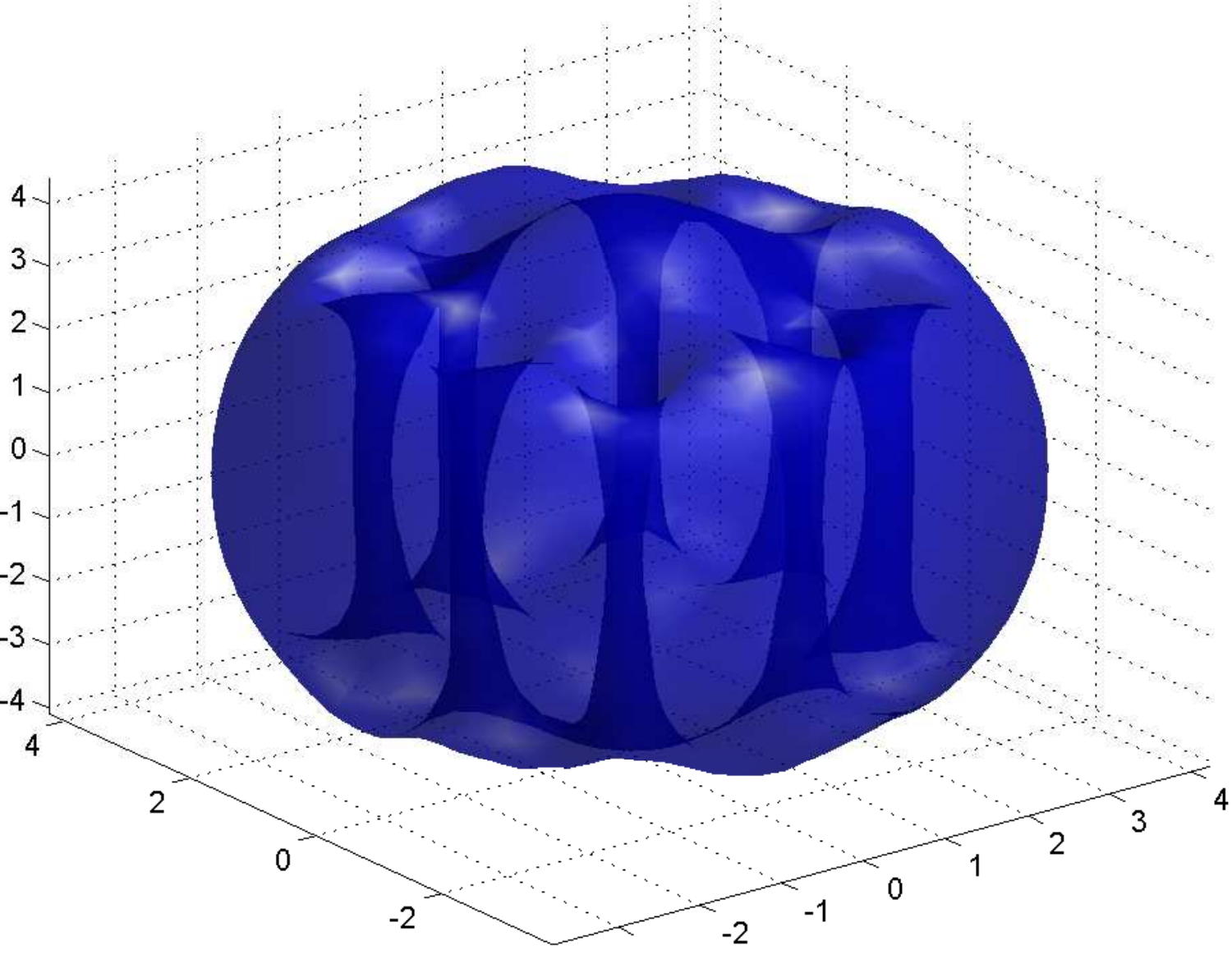}
\setlength{\abovecaptionskip}{-0.1mm}
	\caption{\footnotesize Contour plots of the density function $|\psi|^2$ for the dynamics of vortex lines in a 3D rotating BEC with $\beta=400, \Omega=0.8$ at different times $t=7.0,8.5,10,14,15,20$ (in order from left to right and from top to bottom).}
	\label{fig5.-33}
\end{figure}

\vspace{-0.3cm}

\begin{figure}[H]
\begin{minipage}{0.48\linewidth}
  \centerline{\includegraphics[height=5.5cm,width=0.88\textwidth]{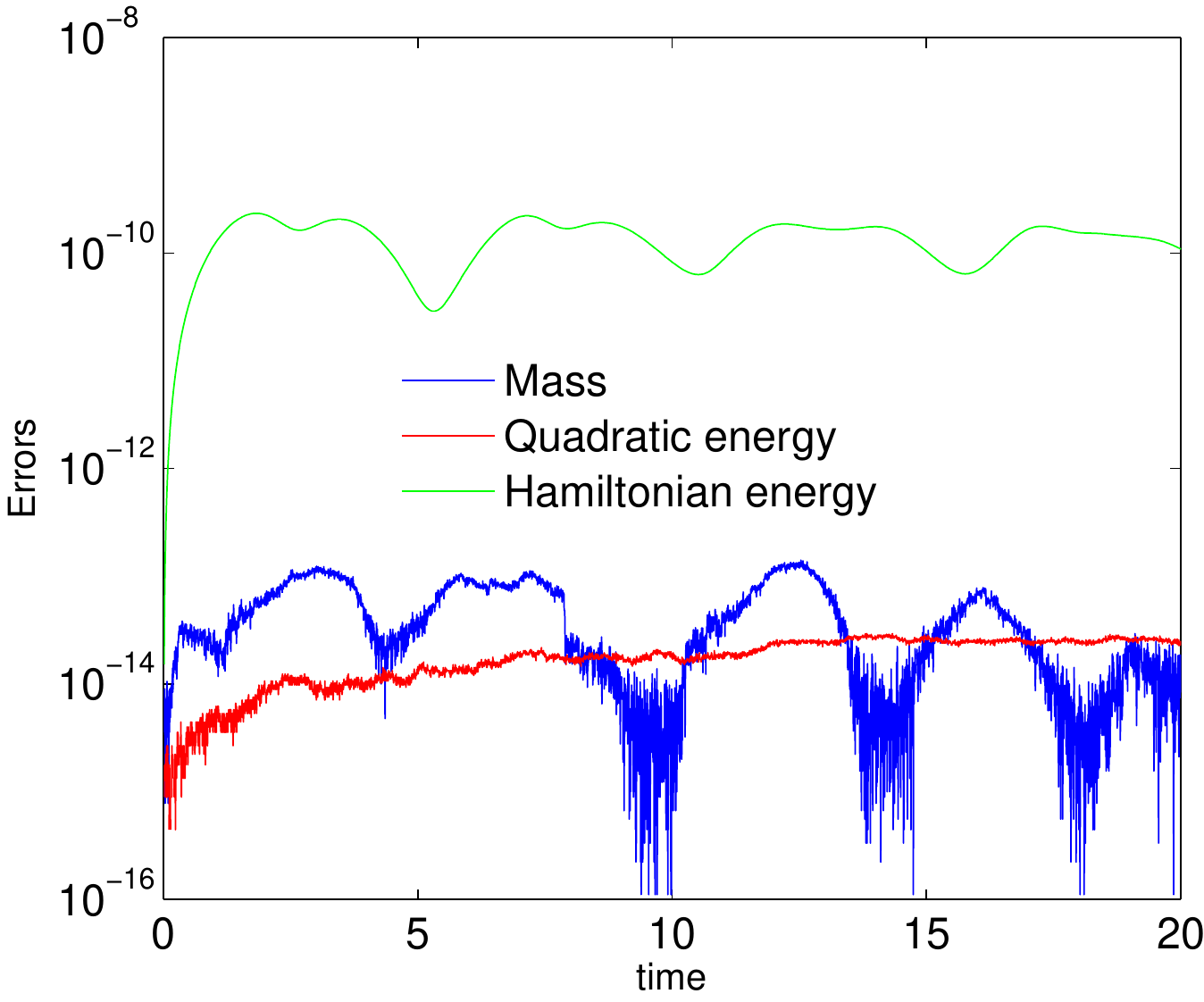}}
  \centerline{\footnotesize  (a) 4th-order HSAV}
\end{minipage}
\hfill
\begin{minipage}{0.48\linewidth}
  \centerline{\includegraphics[height=5.5cm,width=0.88\textwidth]{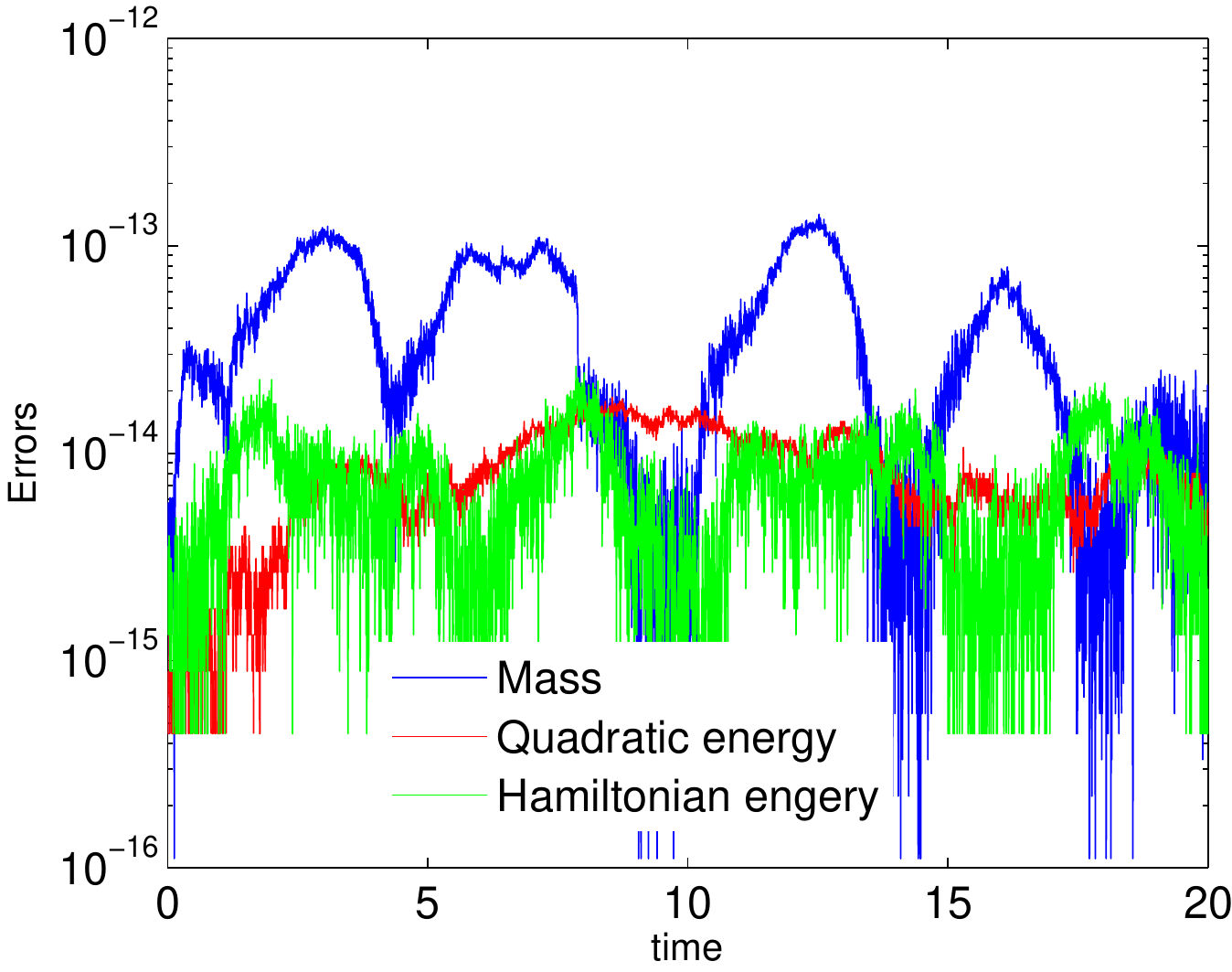}}
  \centerline{\footnotesize  (b) 6th-order HSAV}
\end{minipage}
\caption{\footnotesize Evolution of discrete mass and energy with $\tau=0.005, \beta=400$ and $\Omega=0.8$.}
\label{fig5.-34}
\end{figure}

\section{Conclusions}
\indent In this paper, we combine the SAV idea with the classcial structure-preserving discretization strategy to
develop a novel class of high-order methods for the rotational GP equation in three dimensions. The proposed schemes can reach arbitrarily high-order accuracy in time and preserve exactly both the discrete mass and modified energy of the reformulated system. Three numerical examples are addressed to illustrate the efficiency and accuracy of our new method. Compared with low-order structure-preserving schemes, the proposed schemes, which produce more accurate numerical solutions, are more suitable for longtime dynamic simulations with larger time steps. Last but not least, as far as we know, there are some works (e.g., see \cite{BC13b,baoc,cuijincpc,wangg19}) on optimal error estimates of second-order energy-preserving schemes for the GP equation \eqref{eq1.1}, but the error estimate of high-order ones are still not available. Thus, how to establish optimal error estimates for high-order energy-preserving schemes will be an interesting topic for future studies.

\section*{Acknowledgments}
\indent Jin Cui's work is supported by Natural Research Fund of Nanjing Vocational College of Information Technology (Grant No. YK20160901). Chaolong Jiang's work is partially supported by the National Natural Science Foundation of China (Grant No. 11901513), the Yunnan Provincial Department of Education Science Research Fund Project (Grant No. 2019J0956) and the Science and Technology Innovation Team on Applied Mathematics in Universities of Yunnan. Yushun Wang's work is partially supported by the National Natural Science Foundation of China (Grant No. 11771213) and the National Key Research and Development Project of China (Grant Nos. 2018YFC0603500, 2018YFC1504205). The authors are in particular grateful to Dr. Yuezheng Gong for fruitful discussions on the fast solver presented in Section 5.


%

%

%
%
%

%

%
%
%
%

%

%
\end{document}